\documentclass[11pt]{article}
\usepackage{amsmath}
\usepackage{amssymb}
\usepackage{pgfplots}
\pgfplotsset{compat=1.18}
\usetikzlibrary{patterns}
\usepackage[title]{appendix}
\usepackage{amsthm}
\usepackage{bbm}
\usepackage{dsfont}
\usepackage[pagewise]{lineno}

\newtheorem{theorem}{Theorem}[section]
\newtheorem{lemma}[theorem]{Lemma}

\newtheorem{hypothesis}[theorem]{Hypothesis}

\newtheorem{proposition}[theorem]{Proposition}

\usepackage{mathrsfs}
\usepackage{enumitem}
\usepackage{graphicx}
\usepackage{subcaption}
\usepackage{footmisc}

\usepackage{pifont}
\makeatother

\usepackage{hyperref}
\usepackage[left=2cm,right=2cm,top=2cm,bottom=2cm]{geometry}

\numberwithin{equation}{section}
\providecommand{\keywords}[1]
{
  \small	
  \textbf{\textit{Keywords: }} #1
}
\providecommand{\MSC}[1]
{
  \small	
  \textbf{\textit{Mathematics Subject Classification: }} #1
}

\usepackage{cleveref}
\usepackage{float}
\crefname{fig}{Figure}{Figure}
\crefname{lemma}{Lemma}{Lemmas}

\def\D{{\mathcal{D}}}

\def\R{{\mathbb{R}}}

\def\T{{\Delta}}

\title{Acceleration or finite speed propagation in integro-differential equations with logarithmic Allee effect}
\author{Emeric Bouin\footnote{CEREMADE - Universit\'e Paris-Dauphine, PSL Research University, UMR CNRS 7534, Place du Mar\'echal de Lattre de Tassigny, 75775 Paris Cedex 16, France. E-mail: \texttt{bouin@ceremade.dauphine.fr}}
	\and J\'er\^ome Coville\footnote{UR 546 Biostatistique et Processus Spatiaux, INRAE, Domaine St Paul Site Agroparc, F-84000 Avignon, France. E-mail: \texttt{jerome.coville@inrae.fr}}{ $^,$}\footnote{INRAE, CNRS, Ecole Centrale de Lyon, INSA Lyon, Universite Claude Bernard Lyon 1, Université Jean Monnet, ICJ
UMR5208, 69603 Villeurbanne, France.}
	\and Xi Zhang\footnote{School of Mathematics and Statistics, HNP-LAMA, Central South University, Changsha, Hunan 410083, People's Republic of China. E-mail:  \texttt{xizhangmath@gmail.com}}}

\begin{document}

	\maketitle
	\begin{abstract}
This paper is devoted to studying propagation phenomena in  integro-differential equations with a weakly degenerate non-linearity. The reaction term can be seen as an intermediate between the classical logistic (or Fisher-KPP) non-linearity and the standard weak Allee effect one. We study the effect of the tails of the dispersal kernel on the rate of expansion. When the tail of the kernel is sub-exponential, the exact separation between existence and non-existence of travelling waves is exhibited. This, in turn, provides the exact separation between finite speed propagation and acceleration in the Cauchy problem. Moreover, the exact rates of acceleration for dispersal kernels with sub-exponential and algebraic tails are provided. Our approach is generic and covers a large variety of dispersal kernels including those leading to convolution and fractional Laplace operators. Numerical simulations are provided to illustrate our results.

 \end{abstract}
 
\par
\MSC{35B40, 35C07, 35K57, 35R09, 35R11, 45G05, 47G10, 92D25}

\tableofcontents

      	\section{Introduction}\label{sec1}

This paper sheds light on propagation phenomena arising in semilinear equations of the form
	\begin{equation}\label{oeq1}
	 \frac{\partial u}{\partial t}(t,x)=\mathcal{D}[u](t,x)+f(u(t,x)),\quad (t,x)\in \mathbb{R}^+\times \mathbb{R},
	\end{equation}
    where  $f\in C^1([0,1],\R^{+})$ is a monostable function such that $f(0)=f(1)=0, f'(1)<0$ and $\D$ a nonlocal operator defined as\footnote{The notation P.V. indicates that the integral is understood in the sense of the Cauchy principal value.}
\[
   \mathcal{D}[u](t,x) := \text{P.V.} \int_{\mathbb{R}} \Big(u(t,x - y) - u(t,x)\Big) J(y) dy,
\]
with $J$ a non-negative kernel. 

 This type of equation  arises naturally in population dynamics when long-distance dispersal events are considered. In such a context, the value $u(t,x)$ represents an adimensional density of individuals of a species at time $t$ and spatial location $x$. The operator $\D$ corresponds to the description of the movement of the population modelled by a jump process of probability density $J$, while the nonlinearity $f$ describes the demography.
   \par
In recent decades,  propagation phenomena in reaction-diffusion and integro-differential equations have been the focus of considerable research efforts, starting from the seminal contributions of Kolmogorov, Petrovskii and Piskunov \cite{kolmogorov1937etude} and Fisher \cite{fisher1937wave}.

When the growth rate per capita is maximal at small densities, i.e. $f$ satisfies  the so-called Fisher-KPP condition,
\begin{equation}  \label{kppcon} 
 0<f(z)\le f'(0)z \quad\text{for all }z\in(0,1),
\end{equation}
 a solution to \eqref{oeq1} exhibits some propagation phenomena. That is, starting with some initial data with nonnegative nontrivial compact support, the corresponding solution $u$ converges to $1$, its stable steady state, for large times, locally uniformly in space. This is referred as the \textit{hair trigger effect} \cite{aronson1978multidimensional}.
 
 \par
 A substantial body of research has been dedicated to classifying this invasion phenomenon according to the properties of $J$ and $f$, specifically investigating the propagation speed of solutions to \eqref{oeq1} when it arises. The dispersal kernel plays a significant role in influencing the propagation speed. 
 When it is exponentially bounded, in the sense,
\begin{equation}\label{deb}
    \exists \varepsilon>0 , \quad \int_\mathbb{R}J(y)e^{\varepsilon|y|}dy<\infty,
\end{equation}
travelling wave solutions exist and the solution to the corresponding Cauchy problem propagates at finite speed \cite{carr2004uniqueness,coville2007travelling,coville2008nonlocal,coville2007non,Kot2003,Schumacher19805470,weinberger1982long,yagisita2010existence}. 
Having an exponential moment for $J$, i.e. condition \eqref{deb}, is actually a necessary condition for the existence of travelling wave solution when $f$ is such that $f'(0)>0$, see \cite{yagisita2010existence}.

\par 
When the  kernel $J$ is exponentially unbounded, that is \eqref{deb} is not met, and $f$ is monostable with $f'(0)>0$, Garnier \cite{garnier2011accelerating} proved the existence of accelerating solutions and gave the first estimate of rates of propagation. 
From his results, we learn that if the kernel $J$ is sub-exponential, that is, $J\propto e^{-\mu |\cdot|^\beta}$ with $\beta\in(0,1)$ and $\mu>0$, then the propagation is algebraic, while if the kernel $J$ is algebraic, that is $J\propto |\cdot|^{-(1+2s)}$ for $s> \frac{1}{2}$, then it is exponential. At the same time, Cabr\'e and Roquejoffre \cite{cabre2013influence} obtain similar results for the fractional Laplacian: $\ln |X_\lambda (t)|\sim_{t\to \infty}\frac{f'(0)}{1+2s}t$, whenever $u(t,X_\lambda(t))=\lambda$, for all $t>0$.
\par
Later, for integrable kernels $J$, the bounds obtained by Garnier were improved by  Finkelshtein and collaborators \cite{finkelshtein2019multid,finkelshtein2019} and  by the first author with Garnier, Henderson and Patout \cite{MR3817762}.   The latter, by constructing a pair of  precise sub- and super- solutions,  obtained the following  sharp estimate for $t$ large enough: $X_\lambda(t)\asymp_\lambda J^{-1}(e^{-t})$. 
Although this estimate was obtained for an integrable kernel $J$,  the arguments used do not require such  a specific assumption and  can be adapted to the fractional Laplacian,  thus extending the scope of validity of  this estimate. 
For such a monostable nonlinearity $f$, a clear dichotomy in the propagation behaviour can then  be  observed: if the dispersal kernel $J$ is exponentially bounded, finite speed propagation occurs; otherwise, it leads to acceleration.

\smallskip
\par

When  the growth per capita does not imply an exponential growth at low density, \textit{i.e.} $f'(0)= 0$, propagation still occurs but its study is more subtle. When \eqref{oeq1} is set with a monostable nonlinearity $f$ satisfying: $f(z)\sim _{z\to 0^+}rz^{\alpha+1}$ with $r>0$, Gui and Huan \cite{gui2015traveling} proved existence  of travelling front solutions for \eqref{oeq1} when $\alpha>1$ and $\D$ is the fractional Laplace operator with $2>2s \ge 1+\frac{1}{\alpha}$. They also showed that this condition is sharp, in the sense that no front solution can exist when $2s<\min\left\{1+\tfrac{1}{\alpha},2\right\}$. Later, for an integrable kernel $J$ such that $J\propto |\cdot|^{-(1+2s)}$  the second author with Alfaro \cite{alfaro2017propagation} showed that for any $\alpha>0$, the value $2s = 1+ \frac{1}{\alpha}$ is the sharp threshold for the existence of travelling fronts. More precisely, for $J$ having  a bounded first moment and that satisfies
\begin{equation}
    \label{thde}
  J(z)\lesssim |z|^{-\left(2+\frac{1}{\alpha}\right)}\quad \text{for all } z\ge 1,
\end{equation}
 they proved that travelling waves exist and that the solution to the corresponding Cauchy problem propagates at a constant speed, whereas if the kernel $J$ does not enjoy such a decay no travelling waves can exist. In addition, they provided first estimates of the invasion rate  for solutions to the associated Cauchy problem  with front-like initial data $u_0$, showing a polynomial invasion rate\footnote{Explicitly, their result say that for any $\lambda \in (0,1)$ and any $t \in \R^+$ large enough, $t^{1/2s\alpha}\le_\lambda X_\lambda(t)\le_\lambda t^{1/2s\alpha+1/2s}$, where $u(t,X_\lambda(t))=\lambda$.}. 
  Subsequently, the second author, with Gui and Zhao, considered in \cite{coville2021propagation} the case of $\D$ being a fractional Laplacian and derived similar acceleration estimates of polynomial type when $2s<\min\left\{1+\tfrac{1}{\alpha},2\right\}$. 
  More recently, the first two authors with Legendre \cite{bouin2021sharp} established the following sharp estimate for generic algebraically decaying kernel $J$: 
  $$X_\lambda(t)\asymp_\lambda t^\frac{\alpha+1}{2s\alpha},$$
with $1+\frac{1}{\alpha}>2s$. It is worth mentioning that a parallel work by Zhang and Zlato\v s \cite{zhang2023optimal} obtains similar estimates when $\D$ is the fractional Laplacian and $2s<\min\left\{1+\tfrac{1}{\alpha},2\right\}$, using a different construction.

It is worth mentioning that an acceleration phenomenon also occurs in classical reaction-diffusion equations, especially with fat-tailed initial data. We refer to \cite{alfaro2017slowing,Alfaro2019,Hamel2010} and references therein for more details.

\par
These two dichotomies show that the description of propagation phenomena in \eqref{oeq1} is governed by a strong interplay between the property of the nonlinearity $f$ near zero and the decay of the dispersal kernel $J$ at infinity. The behaviour of $f$ near zero seems critical in determining the right decay of $J$ for which we can observe propagation at a finite speed. While these two classical types of monostable non-linearities encompass a wide range of growth models, there are scenarios where they do not model growth adequately. This raises the question of how sensitive the above results are to the degeneracy of the nonlinearity $f$.
\par
The primary goal of this paper is to complement the knowledge on how the degeneracy of $f$ at zero affects different transitions, particularly aiming to gain fresh insights on the dichotomy between accelerated and finite speed regimes, depending on the behaviour of the nonlinearity 
$f$ near zero. To do so, we examine a class of monostable non-linearities that we call "weakly degenerate", which lie between $f(z)\sim_{z\to 0^+} rz$ and $f(z)\sim_{z\to 0^+}rz^{\alpha+1}$ for $\alpha>0$. More precisely, the non-linearities $f$ satisfy the following conditions:
 
\begin{hypothesis}\label{fu}
		The function $f$ is of class $C^1([0,1],\mathbb{R})$, with
		\[
			f(0)=f(1)=0>f'(1),\quad   f(z)>0  \quad\text{for any }z\in(0,1),
		\]
 and there exist $\xi_0\in (0,1)$, $K\ge0$, $\alpha>0$ and $r>0$ such that 
		\[
			f(z)\le \frac{rz}{(1+\lvert\ln z\rvert)^\alpha} \qquad \text{for all } z\in(0,1),
		\]
  and 
		\[
		f(z)\ge \frac{rz(1-Kz)}{(1+\lvert\ln z\rvert)^\alpha} \qquad \text{for all } z\in(0,\xi_0],
		\]
 (See Fig.\ref{figcom}).
\end{hypothesis}

\begin{figure}
    \centering
\begin{tikzpicture} 
\begin{axis}[
 axis lines = left,
xlabel = $z$,
ylabel = $\left.\frac{f(z)}{z}\right|_{z \approx 0}$, 
xmin=-0.0005,
xmax=0.1,
ymin=0,
ymax=1.2,
xtick = {0.00},
xticklabels = {0},
ytick =\empty,
]
 \addplot[
 color=blue!50!black,
 samples = 100, 
 domain= 1e-10:0.1,
 line width =1.5,
 dotted]{x/x};
\addlegendentry{$r$};
    \addplot[
    color = red,
    line width = 2.0,
    samples = 600,
    domain = 1e-100:0.1,
    ]{1/(-log10(x))};
    \addlegendentry{$\tfrac{r}{\lvert\ln z\rvert^\alpha}$};  
    \addplot[color = green!50!black,
    samples = 100, 
    domain= 0:0.1,
    line width =1.5, 
    dashed
    ]{x^0.5};
        \addlegendentry{$r z^\alpha$};
    \end{axis}
\end{tikzpicture}
%------
\begin{tikzpicture} 
\begin{axis}[
axis lines = left,
xlabel = $z$,
ylabel = $f(z)$,
xmin=0,
xmax=1.2,
ymin=0,
ymax=0.4,
xtick = {0,1},
xticklabels = {0,1},
ytick =\empty,
]
 \addplot[
 color=blue!50!black,
 samples = 100, 
 domain= 1e-10:1,
 line width =1.5,
 dotted]{x*(1-x)};
\addlegendentry{$rz(1-z)$};
    \addplot[
    color = red,
    line width = 2.0,
    samples = 600,
    domain = 1e-100:1,
    ]{x*(1-x)/(1-log10(x))};
    \addlegendentry{$r\frac{z(1-z)}{(1+\lvert\ln z\rvert)^\alpha}$};  
    \addplot[color = green!50!black,
    samples = 100, 
    domain= 0:1, 
    line width =1.5,
    dashed
    ]{x^2*(1-x)};
        \addlegendentry{$r z^{\alpha+1}(1-z)$};
    \end{axis}
\end{tikzpicture}
  \caption{Comparison of the size of three types of non-linearities with the positive parameters $r$ and $\alpha$.}
 \label{figcom}
\end{figure}
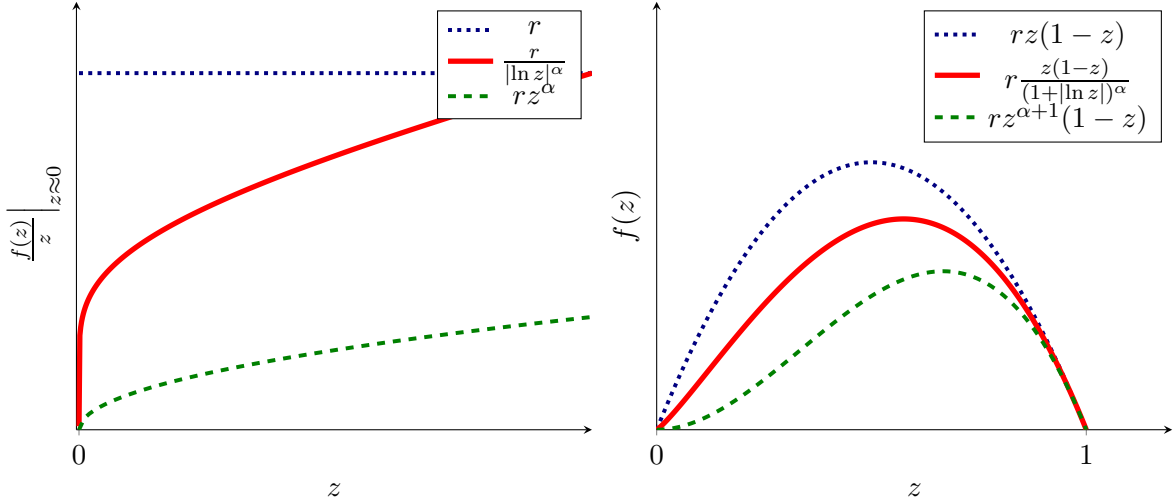
The weakly degenerate non-linearity serves as a framework to examine the transition between degeneracy and non-degeneracy.
The introduction of logarithmic corrections is not only mathematically interesting but also supported by extensive literature modelling specific effects, such as the Gompertz growth rate in population dynamics \cite{Gompertz1825} and in quantum physics \cite{Bialynicki1976}. Here, the logarithmic correction in growth can be interpreted as an Allee effect induced by non-trivial factors, in line with Dennis's work \cite{Dennis1989}. As a motivating curiosity, non-linearities with logarithmic terms naturally emerge from non-local Fisher-KPP equations, and phenomena with non-trivial propagation features have been studied \cite{MR3817762, BOUIN2021}. In a previous paper \cite{BOUIN2024113557}, we have investigated rates of propagation with a local dispersal operator (a laplacian) under a wide class of initial data. For sub-exponential initial data, finite speed propagation may happen or not, while for sub-exponentially unbounded initial data, acceleration necessarily occurs. Some precise rates of invasion were given. One may connect the present contribution to the former one by having in mind that a nonlocal Cauchy problem such as \eqref{oeq1} starting for a Heaviside initial data gives immediately  (that is, at $t>0$) to $u(t,\cdot)$ tails that are comparable to the decay of $J$ at infinity. It is thus natural that the local version studied in  \cite{BOUIN2024113557} also yielded discussions related to sub-exponential or algebraic tails of $u_0$.

\par
Let us now make precise the dispersal kernels $J$ we will consider all along this paper. We will always assume that $J$ satisfies:

\begin{hypothesis}\label{ker}
    The kernel $J$ is nonnegative, symmetric and satisfies the following conditions: there exist positive constants $\mathcal{J}_0$, $\mathcal{J}_1$ and $R_0>1$, along with a decaying function $\mathcal{Z}\in L^1 ([1,+\infty);\mathbb{R}^+)$, such that
    \[
   \int_{|z|\le1}z^2J(z)dz\le 2\mathcal{J}_1 \qquad\text{and}\qquad
  \mathcal{J}_0^{-1}\mathcal{Z}(\cdot)\mathds{1}_{\{|z|\ge R_0\}}(\cdot)  \le J(\cdot) \le \mathcal{J}_0\mathcal{Z}(\cdot)\quad \text{in }[1,+\infty). 
    \]
\end{hypothesis}
There are two common cases covered by the above hypothesis: the standard convolution operators $J*u-u$ with integrable kernels and the fractional Laplacian $(-\triangle)^s$.
\par 
The tails of the dispersal kernel play a key role on the propagation phenomena in equation \eqref{oeq1}. Our first result is about existence or non-existence of travelling wave solutions.

	\begin{theorem}\label{lin1}
		Let $\alpha>0$ and $\beta>0$ be such that $\beta \ge \frac{1}{\alpha+1}.$ Assume that the dispersal kernel $J$ satisfies Hypothesis \ref{ker} with a sub-exponentially bounded tail, that is,
\begin{equation}\label{twa}
	\mathcal{Z}(z)\le e^{-z^\beta} \quad\text{for all } z\ge 1,
\end{equation}
    and that the non-linearity $f$ satisfies Hypothesis \ref{fu}.
Then there exists $c^*>0$ such that for any $c\ge c^*$ equation \eqref{oeq1} admits travelling wave solutions $(c,u)$, whereas, for all $c<c^*$ equation \eqref{oeq1} does not admit travelling wave solutions.
	\end{theorem}
Sub-exponential kernels turn out to be critical for the existence of travelling waves, in the following sense,
	\begin{theorem}
	    \label{net}
	 Let $\alpha>0$ and $\beta>0$ be such that $\beta<\frac{1}{\alpha+1}.$ Assume that the dispersal kernel $J$ satisfies Hypothesis \ref{ker} with
	\[
		 \mathcal{Z}(z)\ge e^{-z^\beta} \quad \text{for all } z\ge 1,
	\]
 and that the non-linearity $f$ satisfies Hypothesis \ref{fu}.
Then there is no travelling wave solution to \eqref{oeq1}.
	\end{theorem}
We shall now consider the Cauchy problem associated to \eqref{oeq1} with the following initial condition:
\begin{hypothesis}\label{ic1}
The initial condition $u_0:\mathbb{R}\to [0,1]$ is of class $C^1$ and satisfies
	\[
		0\le u_0\le 1 \ \text{on} \ \mathbb{R}, \quad \liminf_{x\to-\infty}u_0>0\text{ and } u_0\equiv0 \text{ on }[a,+\infty) \text{ for some }a.
	\]
\end{hypothesis}
Usually, existence of travelling waves leads to finite speed propagation. When the kernel of the integral operator $\mathcal{D}[\cdot]$ has sub-exponential tails with $\beta<\frac{1}{\alpha+1}$, an acceleration phenomenon occurs, and we are able to provide the rate of invasion. In order to estimate this accelerated rate of invasion, we introduce some notations.

For any $\lambda\in(0,1)$, the (lower) level set of $u(t,\cdot)$ at time $t$ is defined by
\[
	E_\lambda(t):=\{x\in\mathbb{R}:u(t,x)\le \lambda\}.
\]
Let $X_\lambda(t)$ be the smallest element of the level set of $u(t,\cdot)$ defined by\footnote{When the initial data $u_0$ is non-increasing, this element can also be defined by
$X_\lambda(t) = \sup\left\{ x \in \R, \, u(t,x) \geq \lambda \right\}.$}
\[
	X_\lambda(t):=\inf E_\lambda(t).
\]

\par
Our first result closes the case of sub-exponential tails.
 \begin{theorem}[Algebraic propagation]\label{ei1}
		Assume $0<\beta<\frac{1}{\alpha+1}$ and that $J$ satisfies Hypothesis \ref{ker} with a sub-exponential tail, that is,
\begin{equation}\label{J2}
		 \mathcal{Z}(z)= e^{-z^\beta} \quad \text{for all } z\ge 1,
\end{equation}
and that the non-linearity $f$ and the initial data $u_0$ satisfy Hypothesis \ref{fu} and Hypothesis \ref{ic1}, respectively. Then, for any $\lambda\in(0,1)$, we have\footnote{The notation $X\asymp_{\Lambda_1, \Lambda_2,\dots} Y$ means that there exists a positive constant $C_{\Lambda_1, \Lambda_2,\dots}$, depending only on some parameters $\Lambda_1$, $\Lambda_2$,\dots, such that $C_{\Lambda_1, \Lambda_2,\dots} Y\le X \le C^{-1}_{\Lambda_1, \Lambda_2,\dots} Y$.}
		\[
			X_\lambda(t)\asymp_\lambda t^{\frac{1}{\beta(\alpha+1)}}.
		\]
	\end{theorem}

To have a more clear picture of our results, we summarize them  in the  diagram below (\Cref{diagram}).

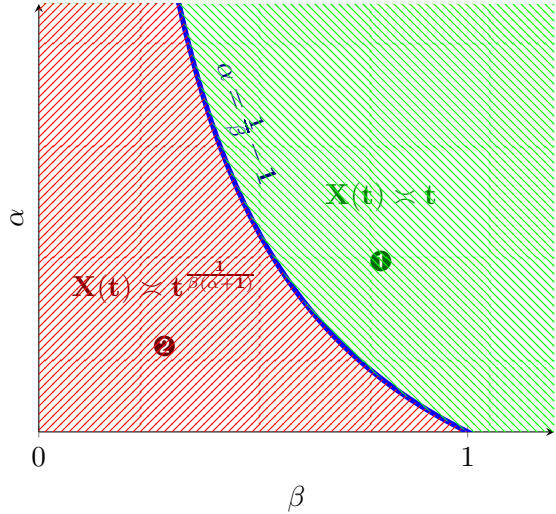
\begin{figure}[H]
     \centering
\begin{tikzpicture}
    \begin{axis}[
    axis lines = left,
    xlabel = {$\beta$},
    ylabel = {$\alpha$},
    xmin = 0.01,
    xmax = 1.2,
    ymin = 0,
    ymax = 2,
    xtick = {0.01,1},
    xticklabels = {0,1},
    ytick = \empty,
    ]
    \addplot[ 
    color = blue,
    domain = 0.1:2,
    samples = 100,
    line width =2,
    ] {1/x-1}; 
\node[anchor= center] at (axis cs:0.3,0.7) {$\textcolor{black!50!red}{\bf X(t)\asymp t^\frac{1}{\beta(\alpha+1)}}$};
\node[anchor= center] at (axis cs:0.3,0.4) {\textcolor{black!60!red}{\ding{203}}};

\node[anchor=center] at (axis cs:0.8,1.1) {$\textcolor{black!50!green}{\bf X(t)\asymp t}$};
\node[anchor=center] at (axis cs:0.8,0.8) {\textcolor{black!60!green}{\ding{202}}};

\node[anchor=north east] at (axis cs:0.6,1.8) {\rotatebox{-70}{$\textcolor{black!50!blue}{\bf \alpha=\frac{1}{\beta}-1}$}};

\addplot [
        domain=0.1:1.2, 
        samples=20,
        pattern color=green,
        pattern=north west lines,
        fill opacity=0.3,
        draw=none,
    ] {1/x-1} -- (axis cs:1.2,2) -- (axis cs:0,2) -- cycle;
 
    \addplot [
        domain=0.01:1, 
        samples=20,
        pattern color=red,
        pattern=north east lines,
        fill opacity=0.3,
        draw=none,
    ] {1/x-1} -- (axis cs:1,0) -- (axis cs:0.01,0) -- cycle;
\end{axis}
\end{tikzpicture}
    \caption{The rate of propagation in the $(\beta,\alpha)$ plane for sub-exponential kernels $J$, that is, $J(z)\asymp e^{-|z|^\beta}$ for $z
    \ge 1$. In the green zone \textcolor{black!60!green}{\ding{202}}, including the critical curve, the solution to \eqref{oeq1} propagates at a finite speed, as shown in Theorem \ref{lin1}. In the red zone \textcolor{black!60!red}{\ding{203}}, acceleration occurs: $X_\lambda(t)\asymp t^\frac{1}{\beta(\alpha+1)}$, see Theorem \ref{ei1}.}
    \label{diagram}
 \end{figure}

Let us now consider  kernels with algebraic tails. 
\begin{theorem}[Exponential propagation]\label{al}
 Assume $s>0$. Suppose that $J$ satisfies Hypothesis \ref{ker} with an algebraic tail, that is, 
 \begin{equation}\label{Je}
     \mathcal{Z}(z)=\frac{1}{z^{1+2s}} \quad\text{for all }z\ge 1,
 \end{equation}
 and that the non-linearity $f$ and the initial data $u_0$ satisfy Hypothesis \ref{fu}  and Hypothesis \ref{ic1}, respectively.
	Then, for any $\lambda\in(0,1)$, we have
 \[
 \ln X_\lambda(t)\sim \frac{1}{2s}[r(\alpha+1)t]^\frac{1}{\alpha+1}\quad \text{as }t\to \infty.
 \]
\end{theorem}
Interestingly, when $\alpha=0$ and $0<s<1$, one recovers the rate of the Fisher-KPP equation with fractional diffusion \cite{cabre2013influence},  and Theorem \ref{al} also fills a (small) gap for $\alpha=0$ and  $s\ge 1$.  
\par
Let us comment on our strategies to prove Theorem \ref{lin1}, Theorem \ref{net}, Theorem \ref{ei1} and Theorem \ref{al}. 
\cite{coville2007non}.
In the proof of Theorem \ref{lin1}, that is, existence of travelling wave solutions, we rely on methods designed in \cite{alfaro2017propagation,bouin2024,coville2008nonlocal}.
We identify a clever form for a super-solution, a sub-exponential with a polynomial prefactor, the latter appearing to be of crucial importance in the proof of the critical case $\beta=\frac{1}{\alpha+1}$. 
Theorem \ref{net} can be regarded as a direct corollary of Theorem \ref{ei1}; therefore, we will not provide a special proof in this paper.
The proofs of Theorem \ref{ei1} and Theorem \ref{al} are 
mainly based on the sub- and super-solutions technique. 
We build on and extend a paper by the two first authors and Legendre \cite{bouin2021sharp}.
For the algebraic propagation, that is, Theorem \ref{ei1}, we construct a super-solution by using the solution to the ordinary differential equation $u_t=\rho \frac{u}{(1-\ln u)^\alpha}$ with a free parameter $\rho$. To obtain a precise estimate of the invasion rate, it is necessary to incorporate an algebraic correction into the (initial) profile, as the weakly degenerate nonlinearity introduces a logarithmic correction that significantly impacts the asymptotic behaviour.
In the exponential propagation, much more precise sub- and super- solutions have to be developed to get sharp estimates on the propagation speed. 
On the one hand, we incorporate a logarithmic correction into the profiles.
On the other hand, since the accelerating solution to the Cauchy problem of \eqref{oeq1} flattens with time \cite{bouin2021sharp,flatteningeffect}, it is essential to take the flattening effect into consideration to get exact rates of invasion.
 For the construction of sub-solutions in the exponential propagation regime (Theorem \ref{al}), the key to get a precise sub-solution is the solution to the full ordinary differential equation $u_t=f(u)$, which causes a lot of technical difficulties since its solution is not explicit. We point out that some steps of both algebraic and exponential proofs are similar but with so significantly different computations that redoing some of them entirely is justified. It is worth mentioning that we prove a flattening estimate of independent interest, which plays a crucial role in proving the existence of sub-solutions at later times.
 \par
The rest of this paper is organized as follows. In Section \ref{s2}, we shall prove the existence of travelling waves, that is, Theorem \ref{lin1}.  In Section \ref{s3}, we prove Theorem \ref{ei1}. In Section \ref{s4}, we prove Theorem \ref{al}. In Section \ref{s5}, we provide the discretisation of the singular integral operator $\mathcal{D}[\cdot]$ that lead to the numerical simulations.
\section{Existence of travelling waves: proof of Theorem \ref{lin1}}\label{s2}	
		\par 
		In this section, we construct travelling wave solutions to \eqref{oeq1} in the regime $\beta\ge \frac{1}{\alpha+1}$, by using the strategy developed in \cite{coville2008nonlocal} for integrable kernel $J$ and in \cite{bouin2024} for non-integrable kernel $J$. We introduce the following perturbation of the travelling wave equation,
  \begin{equation}\label{ptwe} 
	\left\{
\begin{aligned}
&\varepsilon\phi''+\mathcal{D}[\phi]+c\phi'+f(\phi)=0 &\text{in }\mathbb{R},\\
&\phi(-\infty)=1,\quad \phi(\infty)=0.
\end{aligned}
\right.
  \end{equation}
The key step in these two approaches is to build a non trivial super-solution independent of $\varepsilon$ of the above problem \eqref{ptwe}. To construct our super-solution we proceed as follows. 
\par	 Let $w$ be a decreasing smooth function, at least $C^2$, defined on $\mathbb{R}$ by
\par
\medskip
\noindent
 \begin{minipage}{0.45\linewidth}
\[
		w(x):=
		\begin{cases}
			1-e^x, &x\le -1,\\
		C_Lx^pe^{-x^\beta}, &x\ge L,
		\end{cases}
 \]
 \end{minipage}
 \ 
 \begin{minipage}{0.45\linewidth}
    \centering
    \includegraphics[width=0.7\linewidth]{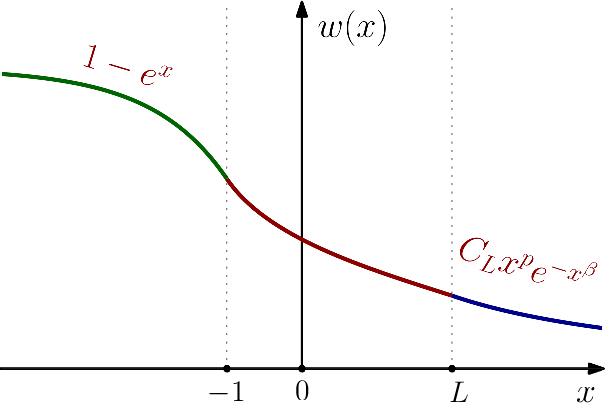}
    \captionsetup{type=figure}
   \caption{Plot of the function $w$.}
    \label{wfig}
\end{minipage}
\par\noindent
where $C_L=\frac{1}{2}\left(1-e^{-1}\right)L^{-p} e^{L^\beta}$ and $L>L_0:=\max\left\{1,\left(\frac{p}{\beta}\right)^{\frac{1}{\beta}}\right\}$, so that $w$ is decreasing on $[L,+\infty)$, and $p>0$ has to be determined, see Figure \ref{wfig}.
\par
  Some direct calculations show that for $x\ge L$,
  \[
      w'(x)=C_L\left(px^{-\beta}-\beta \right) x^{p+\beta-1}e^{-x^\beta},
  \]
  and 
  \[
      w''(x)=C_L\left(p(p-1)x^{-2\beta}-\beta(2p+\beta-1)x^{-\beta}+\beta^2\right)x^{p+2\beta -2}e^{-x^\beta}.
  \]
  Thus, there exists $L_1>L_0$ such that for any $x\ge L$ and $L\ge L_1$, we have 
  \begin{equation}\label{trw'w''}
    w'(x)\asymp -x^{p+\beta-1}e^{-x^\beta}\quad \text{and}\quad  0\le w''(x)\lesssim x^{p+2\beta -2}e^{-x^\beta}.
  \end{equation}
		\par  We shall now prove that $w$ is a super-solution to the travelling wave equation \eqref{ptwe} when selecting an appropriate parameter $p$.
		\begin{lemma}\label{utr}
  Let $\alpha>0$ and $\beta>0$ be such that $\frac{1}{\alpha+1}\le \beta< 1.$ Assume that Hypothesis \ref{ker} holds, where the function $\mathcal{Z}$ is sub-exponentially bounded, that is,
\[
	\mathcal{Z}(z)\le e^{-z^\beta}\quad \text{ for all }z\ge 1,
\]
and that the non-linearity $f$ satisfies Hypothesis \ref{fu}. If $p\ge 2-2\beta$, then, for any $0\le\varepsilon\le1$, there is $c_0>0$ such that for all $x\in\mathbb{R}$,
			\begin{equation}\label{eti}
				\varepsilon w''(x)+\mathcal{D}[w](x)+c_0w'(x)+ f(w(x))\le 0.
			\end{equation}
		\end{lemma}

\par
Before proving the lemma, let us explain how to use $w$. Equipped with the super-solution $w$ to \eqref{ptwe}, similarly to the proof of Theorem 1.3 in \cite{coville2008nonlocal}, one may obtain existence of travelling waves, that is, there is a minimum speed $c^*\le c_0$ such that, for all $c\ge c_*$, equation \eqref{oeq1} admits travelling waves $(c,u)$, whereas, for all $c<c^*$, equation \eqref{oeq1} does not admit any travelling wave solutions.
\par
In view of \cite[Lemma 3.2]{alfaro2017propagation}, one may get that the minimum speed $c^*$ is positive. More precisely, we have the following lemma, which proof is standard: we omit it here and send the reader to \cite{alfaro2017propagation}.
\begin{lemma}
  Assume Hypothesis \ref{fu} and Hypothesis \ref{twa} hold. Let $(c,u)$ be a travelling wave. Then the speed $c$ is positive. Moreover, so is the minimum speed $c^*$.
\end{lemma}
Let us complete the proof of Theorem \ref{lin1} by proving Lemma \ref{utr}.
\begin{proof}[Proof of Lemma \ref{utr}.]
 We prove the lemma by dividing it into three zones of positions $x$: $x\le -2$, $x\ge 2L$ and $-2<x<2L$.
 \medskip
\par
\noindent
 \textbf{Zone 1:} $x\le -2$. Now, let us briefly estimate $\mathcal{D}[w]$ for any $x\le -2$. Since $w$ is decreasing, by the definition of $\mathcal{D}[w]$, we have 
\begin{equation}\label{etd2}
    \mathcal{D}[w]\le  P.V. \int_{-1}^1J(y)\Big(w(x-y)-w(x)\Big)dy+\int_{1}^{+\infty}J(y)\Big(w(x-y)-w(x)\Big)dy.
\end{equation}
Repeating the same steps as for the derivation of \eqref{trI2}, we have
\[
     P.V. \int_{-1}^1J(y)\Big(w(x-y)-w(x)\Big)dy \le -\mathcal{C}_7e^x,
\]
for some positive constant $\mathcal{C}_7$. For the second integral of \eqref{etd2}, since $w\le 1$ and $y\mapsto J(y)\in L^1([1,+\infty))$, we obtain
\[
    \int_{1}^{+\infty}J(y)\Big(w(x-y)-w(x)\Big)dy=\int_{1}^{+\infty}J(y)\Big(w(x-y)-1\Big)dy+e^x\int_{1}^{+\infty}J(y)dy\le e^x\int_{1}^{+\infty}J(y)dy.
\]
As a result, for $x\le-2$, we have
\[
   \mathcal{D}[w]\le- \mathcal{C}_7e^x+e^x\int_{1}^{+\infty}J(y)dy.
\]
Therefore, for $x\le-2$, since $f(w)\le \frac{rw}{(1-\ln w)^\alpha}(1-w)\le rw(1-w)$ for $w\in(0,1)$, we have
\[
\begin{aligned}	
\varepsilon w''+\mathcal{D}[w] +c_0w' +f(w)&\le -\varepsilon e^x-\mathcal{C}_7e^x+e^x\int_{1}^{+\infty}J(y)dy-c_0 e^x+r(1-e^x)e^x\\
&\le -\left(c_0-\int_{1}^{+\infty}J(y)dy-r\right)e^x\le 0,
\end{aligned}
\]
as long as $c_0>c_{left}:=\int_{1}^{+\infty}J(y)dy+r$.

\medskip
\par
\noindent
 \textbf{Zone 2:} $-2< x< 2L$. By the definition of $\mathcal{D}[w]$, since $w$ is decreasing and $0<w< 1$, we write
\[
\mathcal{D}[w]\le P.V.\int_{-1}^1\Big(w(x-y)-w(x)\Big)J(y)dy +\int_1^{+\infty}J(y)dy.
\]
 Repeating the same steps as for the derivation of \eqref{trI2}, since $w\in C^2$ and $y\mapsto J(y)\in L^1([1,+\infty))$, we have, for all $x\in(-2,2L)$,
    \[\mathcal{D}[w](x)\le \mathcal{C}_8,\]
  for some positive constant $\mathcal{C}_8$.
It follows from $\varepsilon\le1$ and $f(w)\le r$ that, for any $x\in(-2,2L)$, we have
\[
\varepsilon w''(x)+ \mathcal{D}[w](x) +c_0 w'(x) +f(w(x))\le \varepsilon w''(x)+\mathcal{C}_8+c_0 w'(x)+r\le 0,
\]
as soon as $c_0>c_{middle}:=\frac{\left(\max_{x\in[-2,2L]}w''(x)\right)^++\mathcal{C}_8+r}{\min_{x\in[-2,2L]}|w'(x)|}$.

 \medskip
 \par
 \noindent
 \textbf{Zone 3: $x\ge 2L$}. Let us first check \eqref{eti} when $x\ge 2L$. By the definition of $\mathcal{D}[w]$, we write 
\[
\begin{aligned}
\mathcal{D}[w]=&\int_{-\infty}^{-1}J(y)\Big(w(x-y)-w(x)\Big)dy+ P.V. \int_{-1}^{1}J(y)\Big(w(x-y)-w(x)\Big)dy\\
&\qquad+\int_{1}^{\frac{x}{2}}J(y)\Big(w(x-y)-w(x)\Big)dy+ \int_{\frac{x}{2}}^{x-L}J(y)\Big(w(x-y)-w(x)\Big)dy \\
&\qquad \qquad +\int_{x-L}^{+\infty}J(y)\Big(w(x-y)-w(x)\Big)dy
:=I_1+I_2+I_3+I_4+I_5.
\end{aligned}
\]
\par
\# \textbf{For $I_1$,} it follows from the fact that $w$ is decreasing for $x\ge L$ that
\begin{equation}\label{trI1}
    I_1\le 0.
\end{equation}
\par 
\# \textbf{Let us now estimate $I_2$.} Since $J$ is symmetric, we write
\[
    I_2 = \frac{1}{2}\int_{-1}^{1}J(y)\Big(w(x+y)+w(x-y)-2w(x)\Big)dy.
\]
It follows from Taylor's theorem and \eqref{trw'w''} that for all $y\in[-1,1]$,
\[
    w(x+y)+w(x-y)-2w(x)\le y^2\sup_{|z|\le 1} w''(x+z)\lesssim y^2 \sup_{|z|\le 1} \left[(x+z)^{p+2(\beta-1)}e^{-(x+z)^\beta}\right].
\]
Thus, since $0<\beta<1$, by Hypothesis \ref{ker}, we get 
\begin{equation}\label{trI2}
    I_2 \le  \mathcal{C}_1 x^{p+2\beta-2}e^{-x^\beta},
\end{equation}
for some positive constant $\mathcal{C}_1$.

\par
\# \textbf{Let us estimate $I_3$.} By the Taylor theorem, we get
\[
    I_3= -\int_{1}^{\frac{x}{2}}\int_0^1 yJ(y)w'(x-\tau y)d\tau dy.
\]
Since $w$ is convex for $x\ge 2L$ and $L\ge L_1$, by \eqref{trw'w''} and Hypothesis \ref{ker}, we get
\[
  I_3\le - \int_{1}^{\frac{x}{2}} yJ(y)w'(x-y)dy\lesssim\int_{1}^{\frac{x}{2}} y(x-y)^{p+\beta-1}e^{-y^\beta-(x-y)^\beta}dy.  
\]
By using the change of variables $z=\frac{y}{x}$ and assuming $p\ge 1-\beta$, we obtain, for 
all $x\ge 2L$, 
\begin{equation}\label{trI4f}
\begin{aligned}
    I_3\lesssim  x^{p+\beta+1}\int_{\frac{1}{x}}^{\frac{1}{2}} z(1-z)^{p+\beta-1}e^{-x^\beta \left[z^\beta+(1-z)^\beta\right]}dz\lesssim  x^{p+\beta+1}e^{-x^\beta}\int_{\frac{1}{x}}^{\frac{1}{2}} z e^{-x^\beta \left[z^\beta+(1-z)^\beta-1\right]}dz.
\end{aligned}
\end{equation}
There is $L_2>\max\{L_0,L_1\}$ such that for any $L\ge L_2$,
\[
    \frac{\ln^\frac{1}{\beta} (x^3)}{x}\in \left(\frac{1}{x},\frac{1}{2}\right)\quad \text{for all }x\ge 2L.
\]
Then, we split the above integral into two parts and obtain
\begin{equation}\label{twsi3}
\begin{aligned}
  \int_{\frac{1}{x}}^{\frac{1}{2}} z e^{-x^\beta \left[z^\beta+(1-z)^\beta-1\right]}dz=&\int_{\frac{\ln^\frac{1}{\beta} (x^3)}{x}}^{\frac{1}{2}} z e^{-x^\beta \left[z^\beta+(1-z)^\beta-1\right]}dz+\int_{\frac{1}{x}}^{\frac{\ln^\frac{1}{\beta} (x^3)}{x}} z e^{-x^\beta \left[z^\beta+(1-z)^\beta-1\right]}dz.
\end{aligned}
\end{equation}
Notice that since $\beta\in(0,1)$ and $z\mapsto z^\beta+(1-z)^\beta$ is increasing on $\left[0,\frac{1}{2}\right]$, one has for all $z\in\left[\frac{\ln^\frac{1}{\beta} (x^3)}{x},\frac{1}{2}\right]$,  
\[
    z^\beta+(1-z)^\beta\ge \frac{\ln (x^3)}{x^\beta}+\left(1-\frac{\ln^\frac{1}{\beta} (x^3)}{x}\right)^\beta \ge \frac{\ln (x^3)}{x^\beta}+1 -\frac{\ln^\frac{1}{\beta} (x^3)}{x}.
\]
Thus, since $\beta\in(0,1)$, for the first integral in \eqref{twsi3}, we have 
\[
    \int_{\frac{\ln^\frac{1}{\beta} (x^3)}{x}}^{\frac{1}{2}} z e^{-x^\beta \left[z^\beta+(1-z)^\beta-1\right]}dz\le  e^{- \ln (x^3)+\frac{\ln^\frac{1}{\beta}( x^3)}{x^{1-\beta}}}=o_{x\to+\infty}\left(\frac{1}{x^2}\right).
\]
For the second integral in \eqref{twsi3}, we have
\[
\int_{\frac{1}{x}}^{\frac{\ln^\frac{1}{\beta} (x^3)}{x}} z e^{-x^\beta \left[z^\beta+(1-z)^\beta-1\right]}dz
    \le \exp\left\{x^\beta\left[1-\left(1-\frac{\ln^\frac{1}{\beta} (x^3)}{x}\right)^\beta\right] \right\} \int_{\frac{1}{x}}^{\frac{\ln^\frac{1}{\beta} (x^3)}{x}} z e^{-x^\beta z^\beta}dz. 
\]
Since $0<\beta<1$, by using the inequality $1-(1-y)^\beta\le y$ for all $y\in(0,1)$,  we have for all $x\ge 2L$ and $L\ge L_2$, up to enlarging $L_2$ if necessary,
\[
   \exp\left\{x^\beta\left[1-\left(1-\frac{\ln^\frac{1}{\beta} (x^3)}{x}\right)^\beta\right] \right\} =\exp\left\{\frac{\ln^\frac{1}{\beta} (x^3)}{x^{1-\beta}}\right\}\le 2.
\]
On the other hand, since $y\mapsto y e^{-y^\beta}\in L^1([1,+\infty))$, by using the change of variables $y=zx$, we obtain
\[
     \int_{\frac{1}{x}}^{\frac{\ln^\frac{1}{\beta} (x^3)}{x}} z e^{-x^\beta z^\beta}dz = \frac{1}{x^2}\int_1^{\ln^\frac{1}{\beta} (x^3)}y e^{-y^\beta}dy \le \frac{1}{x^2}\int_1^{+\infty}y e^{-y^\beta}dy.
\]
 We then arrive at
\[
\int_{\frac{1}{x}}^{\frac{1}{2}} z e^{-x^\beta \left[z^\beta+(1-z)^\beta-1\right]}dz
      \le  \frac{2}{x^2}\int_1^{+\infty}y e^{-y^\beta}dy+o_{x\to+\infty}\left(\frac{1}{x^2}\right).    
\]
As a result, for all $x\ge 2L$ and $L>L_2$, up to enlarging $L_2$ if necessary, we achieve, recalling \eqref{trI4f}, 
\begin{equation}\label{trI4}
   I_3\le  \mathcal{C}_3 x^{p+\beta-1}e^{-x^\beta} ,    
\end{equation}
for some positive constant $\mathcal{C}_3$.
\par 
\# \textbf{Now, let us estimate $I_4$.} Since $w\ge 0$, one may directly get 
\begin{equation*}
    I_4\le \int_{\frac{x}{2}}^{x-L}J(y)w(x-y)dy.
\end{equation*}
By using the change of variables $z=1-\frac{y}{x}$, the definition of $w$ and Hypothesis \ref{ker}, we have
\[
I_4\le \mathcal{J}_0 x^{p+1}e^{-x^\beta}\int^{\frac{1}{2}}_{\frac{1}{x}}z^p e^{-x^\beta\left[z^\beta+(1-z)^\beta-1\right] }dz.
\]
There is $L_3\ge\max\{L_0,L_1\}$ such that for $L\ge L_3$,
\[
\frac{1}{x}\le \frac{\ln^\frac{1}{\beta}(x^{p+2})}{x}\le \frac{1}{2}\quad \text{for all }x\ge 2L.
\]
As for the estimate of $I_4$, the above integral can be split into two parts:
\[
\begin{aligned}
    \int^{\frac{1}{2}}_{\frac{1}{x}}z^p e^{-x^\beta\left[z^\beta+(1-z)^\beta-1\right] }dz&=\int^{\frac{1}{2}}_{\frac{\ln^\frac{1}{\beta}(x^{p+2}) }{x}}z^p e^{-x^\beta\left[z^\beta+(1-z)^\beta-1\right] }dz+\int^{\frac{\ln^\frac{1}{\beta}(x^{p+2}) }{x}}_{\frac{1}{x}}z^p e^{-x^\beta\left[z^\beta+(1-z)^\beta-1\right] }dz.
\end{aligned}
\]
By repeating the steps of the estimate of $I_4$, we get
\[
\int^{\frac{1}{2}}_{\frac{\ln^\frac{1}{\beta}(x^{p+2}) }{x}}z^p e^{-x^\beta\left[z^\beta+(1-z)^\beta-1\right] }dz=O_{x\to+\infty}\left(\frac{1}{x^{p+2}}\right),
\]
and 
\[
\int^{\frac{\ln^\frac{1}{\beta}(x^{p+2}) }{x}}_{\frac{1}{x}}z^p e^{-x^\beta\left[z^\beta+(1-z)^\beta-1\right] }dz\le \frac{1}{x^{p+1}}\int_1^{+\infty}y^p e^{-y^\beta}dy.
\]
As a result, for all $L\ge L_3$ up to enlarging $L_3$ if necessary, we have
\begin{equation}\label{trI5}
    I_4\le \mathcal{J}_0\left[\int_1^{+\infty}y^p e^{-y^\beta}dy+O_{x\to+\infty}\left(\frac{1}{x}\right)\right]e^{-x^\beta}\le \mathcal{C}_4 e^{-x^\beta}\quad \text{for any }x\ge 2L,
\end{equation}
for some positive constant $\mathcal{C}_4$.
\par 
\# \textbf{Let us now estimate $I_5$.} Since $0<w\le 1$, by Hypothesis \ref{ker}, we have
\[
	I_5=\int_{x-L}^{+\infty}J(y)\Big(w(x-y)-w(x)\Big)dy
	\le \mathcal{J}_0\int_{x-L}^{+\infty} e^{-y^\beta}dy.
\]
It follows from L'H\^opital's rule that 
\[
\lim_{x\to +\infty}\frac{\int_{x-L}^{+\infty} e^{-y^\beta}dy}{\frac{1}{\beta}x^{1-\beta}e^{-x^\beta}}=1,
\]
whence there exists some constant $\mathcal{C}_2>0$ such that for all $x>2L$,  we have
\begin{equation}\label{trI3}
I_5\le\mathcal{J}_0\int_{x-L}^{+\infty} e^{-y^\beta}dy \le \mathcal{C}_2x^{1-\beta}e^{-x^\beta}.
\end{equation}
%-----
\par 
Collecting \eqref{trI1}, \eqref{trI2}, \eqref{trI4}, \eqref{trI5} and \eqref{trI3}, for all $x\ge 2L$ and $L\ge \tilde L:=\max\{L_2,L_3\}$, we have
\[
\mathcal{D}[w](x)\le x^{p+\beta-1}e^{-x^\beta}\left(  \mathcal{C}_1 x^{\beta-1}+\mathcal{C}_2 x^{2(1-\beta)-p}+\mathcal{C}_3 +\mathcal{C}_4x^{1-\beta-p}\right).
\]
In view of Hypothesis \ref{fu}, by the definition of $w$, there is a positive constant $\mathcal{C}_5$ such that we have for all $x\ge 2L$,
\[
f(w)\le \frac{rw}{(1-\ln w)^\alpha}= \frac{rC_Lx^p e^{-x^\beta}}{(1-\ln C_L+x^\beta- p\ln x)^\alpha}\le \mathcal{C}_5 x^{p-\alpha\beta}e^{-x^\beta}.
\]
By \eqref{trw'w''}, there are two positive constants $\omega_1$ and $\omega_2$ such that
\[
w'(x)\le -\omega_1 x^{p+\beta-1}e^{-x^\beta} \quad\text{and}\quad w''(x)\le \omega_2 x^{p+2\beta-2}e^{-x^\beta}.
\]
Therefore, for any $0\le\varepsilon\le1$, $x\ge 2L$ and $L\ge \tilde L$, we have
\[
\begin{aligned}
&\varepsilon w''+\mathcal{D}[w]+c_0w'+f(w)\\
&\le x^{p+\beta-1}e^{-x^\beta}\Big((\varepsilon \omega_2+\mathcal{C}_1)x^{\beta-1}+\mathcal{C}_2 x^{2-2\beta-p}+ \mathcal{C}_3+\mathcal{C}_4 x^{1-\beta-p}-c_0\omega_1 +\mathcal{C}_5x^{1-\beta-\alpha\beta}\Big).
\end{aligned}
\]
Since $\frac{1}{\alpha+1}\le\beta<1$, by selecting $p\ge 2-2\beta$,  
\[
\varepsilon w''+\mathcal{D}[w]+c_0w'+f(w)
\le (-c_0\omega_1 + \omega_2+\mathcal{C}_1+\mathcal{C}_2+ \mathcal{C}_3+\mathcal{C}_4 +\mathcal{C}_5)x^{p+\beta-1}e^{-x^\beta},
\]
and then, by choosing $c_0>c_{right}:=\frac{1}{\omega_1 }(\omega_2+\mathcal{C}_1+\mathcal{C}_2+\mathcal{C}_3+\mathcal{C}_4+\mathcal{C}_6),$
one achieves, for any $0\le\varepsilon\le1$ and $L\ge \tilde L:=\max\{L_2,L_3\}$,
\[
\varepsilon w''+\mathcal{D}[w]+c_0w'+f(w)\le 0 \quad \text{for all }x\ge 2L.
\]
\par
Let us finish the proof of the lemma. By choosing a speed 
$c_0>\max\left\{ c_{left}, c_{middle},c_{right}\right\}$,
it follows from the above computations that, for any $\varepsilon\in[0,1]$, we have
\[
\forall x\in\mathbb{R},\quad \varepsilon w''+\mathcal{D}[w] +c_0w' +f(w)\le 0 .
\]
			This completes the proof.
		\end{proof}

 \section{Algebraic propagation: proof of Theorem \ref{ei1}}\label{s3}
In this section, we present the sub- and super- solution process that leads to the proof of Theorem \ref{ei1}, the latter being obtained at the end of the section.
	\subsection{The upper bound}
		Here, we prove the upper bound in Theorem \ref{ei1}. To do so, we shall construct an adequate super-solution.
\subsubsection{Definition of the super-solution and preliminary computations}
		Let us define 
	\begin{equation}\label{v0x}
 v_0(x):=\left\{
		\begin{aligned}
			&1,&x<L,\\
			&C_L x^pe^{-x^\beta},& x\ge L,
		\end{aligned}\right.
	\end{equation}
where $p>0$ has to be determined, $C_L:=L^{-p}e^{L^\beta}$  and $L\ge\max\{(\frac{p}{\beta})^{\frac{1}{\beta}},1, a\}$, so that $v_0$ is continuous and decreasing, $v_0\le 1$ and $C_L$ is greater than $e$. Here, $a$ comes from Hypothesis \ref{ic1}, and thus, one may notice that 	
\begin{equation}
    \label{apubv0u0}
    v_0\ge u_0.
\end{equation}
Solving the Cauchy problem
\[
    \left\{
        \begin{aligned}
				&\frac{\partial w}{\partial t}(t,x)=\rho\frac{w(t,x)}{(1-\ln w(t,x))^\alpha},&t>0,x\in\mathbb{R},\\
				&w(0,\cdot)=v_0\ge 0, 
        \end{aligned}
   \right.
\]
leads to defining 
		\[
			w(t,x):=\exp{\left\{1-\left[(1-\ln v_0(x))^{\alpha+1}-\rho(\alpha+1)t\right]^{\frac{1}{\alpha+1}}\right\}},
		\]
  where $\rho\ge r$ has to be chosen later. Observe though that $w$ is not defined for all times: for any $x\in\mathbb{R}$, $w(t,x)$ is defined only for $t\in \left[0,\frac{(1-\ln v_0(x))^{\alpha+1}}{\rho(\alpha+1)}\right)$. For later use, we  state below several properties related to $v_0$ and $w$, the proofs of which can be found in Appendix \ref{appendix-alg-acc-upper}.
\begin{lemma}\label{apub2l}
$(i).$    For any $y\in(0,1)$, denoting by $v_0^{-1}$ the inverse of $v_0$ on $[L,+\infty)$,
\begin{equation}
    \label{aiv0}
 \left[\ln \frac{C_L}{y}+\frac{p}{\beta}\ln\ln\frac{e^{L^\beta}}{y} \right]^\frac{1}{\beta} \le v_0^{-1}(y) \le \left[\ln \frac{C_L}{y}+\left(\frac{p}{\beta}+\gamma\right)\ln\ln \frac{e^{L^\beta}}{y}\right],
\end{equation}
for some positive constant $\gamma>0$, and
\begin{equation}
    \label{alubvv}
   (v_0'\circ v_0^{-1})(y)\gtrsim- \lvert\ln y\rvert^{\frac{\beta-1}{\beta}}y.
\end{equation}
\par
$(ii).$ For any $t\ge t^\star:=\frac{1}{r(\alpha+1)}$, define
  \[      x_\Lambda(t):=v_0^{-1}\left\{\exp\left\{1-\left[\rho(\alpha+1)t+(1-\ln\Lambda)^{\alpha+1}\right]^{\frac{1}{\alpha+1}}\right\}\right\}\quad \text{for all }\Lambda\in(0,2].
  \]
 Then, we have $w(t,x_\Lambda(t))=\Lambda$,
\begin{equation}\label{Xte}
  x_\Lambda(t)\asymp_\Lambda t^\frac{1}{\beta(\alpha+1)}\quad \text{for all }t\in [t^\star,\infty) \text{ and }\Lambda\in (0,2],
\end{equation}
 and for any $0<\Lambda_1< \Lambda_2\le 2$,
 \begin{equation}\label{apdx}
    x_{\Lambda_1}(t)-x_{\Lambda_2}(t)\gtrsim_{\Lambda_1,\Lambda_2}t^{\frac{1}{\beta(\alpha+1)}-1}\quad \text{for all }t\in [t^\star,\infty) .
 \end{equation}
\end{lemma}
\par
\medskip
For all $t\in [t^\star,\infty)$, let us now define the following function $m$ as our prototype of a super-solution to \eqref{oeq1}: 
\begin{equation}\label{um}
m(t,x):=\left\{
\begin{aligned}
&1,&x\le x_1(t),\\
&w(t,x),&x>x_1(t),.
\end{aligned}\right.
		\end{equation}
\par 

We will see that for the right choice of $p$ and $\rho$, the function $m$ is indeed a super-solution to \eqref{oeq1}. In order to prove that, we first collect in the next two lemmas some additional useful properties of $w$. We first have
\begin{lemma}\label{lewxx}
Let $\Lambda\in(0,2]$. For all $(t,x)\in[t^\star,\infty)\times [x_\Lambda(t),+\infty)$, the function $w$ is well-defined and non-increasing in $x$, and for any $t\in[t^\star,\infty)$, the function $w(t,\cdot)$ is convex and log-convex on $[x_\Lambda(t),+\infty)$.  Moreover, setting $\varphi_0=-\ln v_0$, for all $(t,x)\in[t^\star,\infty)\times [x_\Lambda(t),+\infty)$, we have
\begin{equation}
 \label{apubwv0}   
 w(t,x)\ge v_0(x),
\end{equation} 
\begin{equation}\label{wx0}
    \frac{\partial w}{\partial x}=-\frac{w}{(1-\ln w)^\alpha}\varphi_0'(1+\varphi_0)^\alpha,
\end{equation}
and
\begin{equation}
\label{apubwxxe} 
       0 \le \frac{\partial^2 w}{\partial x^2}(t,x)\le \mathcal{C}_0 \frac{w(t,x)}{(1-\ln w(t,x))^\alpha},
\end{equation}
    for some positive constant $\mathcal{C}_0$.
\end{lemma}

Last, in order to estimate $\mathcal{D}[w]$ later, we establish the following.
\begin{lemma}
    \label{asuble1}
    Let $\alpha>0$ and $0<\beta<1$. For all
    $z\in\left(\frac{1}{x}, 1-\frac{x_1(t)}{x}\right)$ and $(t,x)\in[t^\star,\infty)\times[x_1(t)+1,+\infty)$, one has
    \[
       \frac{w(t,x(1-z))}{w(t,x)}
       \le \exp\left\{\beta x^{\alpha\beta+\beta}z+ C_{\alpha,\beta} x^{2\alpha\beta+2\beta} z^2\right\},
    \]
    and
    \[
        \frac{1-\ln w(t,x(1-z))}{1-\ln w(t,x)}\ge   1-\beta x^{\alpha\beta+\beta}  z-  C_{\alpha,\beta}  x^{2\alpha\beta+2\beta} z^2,
    \]
    where $C_{\alpha,\beta}:=1-\beta+\alpha(\alpha+1)>0$.
\end{lemma}

With the above lemmas at hand, we are now ready to show that $m$ is a super-solution to \eqref{oeq1} when the constants $\rho>0$ and $p>0$ are chosen adequately. Unless otherwise specified, throughout the next subsections, we assume that $t\in[t^\star,\infty)$.
We start by proving an estimate for $\mathcal{D}[m]$.
\subsubsection{Estimation of \texorpdfstring{$\mathcal{D}[m](t,\cdot)$ on $(x_1(t),+\infty)$}{D[m](t,.) on [x1(t),+infty)}}
\begin{lemma}\label{fdm}
There exists a time $t_0\ge t^\star$ such that, for all $t\ge t_0$ and $x> x_1(t)$, we have
    \[
        \mathcal{D}[m](t,x)\le \mathcal{C}_1,
    \]
where $\mathcal{C}_1:=\max\left\{\frac{2\mathcal{J}_1\mathcal{C}_0}{(1-\ln 2)^\alpha}, 2\int_{|y|\ge 1} J(y)dy\right\}>0$.
\end{lemma}
\begin{proof}[Proof of Lemma \ref{fdm}] 
Let us denote $\tilde J(y)= J(y)\mathds{1}_{|y|\ge 1}(y)$. Since $J$ is symmetric, we write 
\[
    \mathcal{D}[m](t,x)= \frac{1}{2}\int_{-1}^{1}J(y)\Big(m(t,x+y)+m(t,x-y)-2m(t,x)\Big)dy+ \int_\mathbb{R}\tilde J(y)\Big(m(t,x-y)-m(t,x)\Big)dy:=I+II.
\]
\par \# \textbf{Let us estimate the integral $I$}. In view of \eqref{apdx},  there is a time $t_0\ge t^\star$ such that $x_1(t)-x_2(t)\ge 1$. By definition of $m$ for all $t>t^\star,x\ge x_2(t)$ we have $m(t,x)\le w(t,x)$. 
Therefore, for all $t>0$ and $x>x_1(t)$, we have

\[
\begin{aligned}
     I\le\frac{1}{2} \int_{-1}^{1}J(y)\Big(w(t,x+y)+w(t,x-y)-2w(t,x)\Big)dy.
\end{aligned}
\]
Recall that  for $t\ge t_0$ we have $x_1(t)>x_2(t)+1$. From this observation and since $w(t,\cdot)\in C^2$ on $[x_2(t),+\infty)$, by Taylor's theorem, we get  for all $(t,x)\in [t_0,\infty)\times(x_1(t),+\infty)$,
\[
I\le \frac{1}{2}\int_{-1}^1 \int_0^1\int_0^1\frac{\partial^2 w}{\partial x^2}(t,x+\tau \theta y) \theta y^2 d\tau d \theta dy\le \frac{1}{2}\sup_{|z|\le 1} \frac{\partial^2 w}{\partial x^2}(t,x+ z)\int_{-1}^1 y^2J(y)dy.
\]
Thus, by Hypothesis \ref{ker} and \eqref{apubwxxe}, since $w(t,\cdot)$ is non-increasing and less than $2$ on $[x_2(t),+\infty)$, for all $(t,x)\in [t_0,\infty)\times(x_1(t),+\infty)$, we have
\[
I\le\mathcal{J}_1\mathcal{C}_0\sup_{|z|\le 1} \frac{w(t,x+ z)}{(1-\ln w(t,x+ z))^\alpha} \le \mathcal{J}_1\mathcal{C}_0 \frac{w(t,x-1)}{(1-\ln w(t,x-1))^\alpha}\le  \frac{2\mathcal{J}_1\mathcal{C}_0}{(1-\ln 2)^\alpha}.
\]

\par \#
\textbf{Let us estimate the integral $II$}. Since $0< m\le 1$ for all $x\in\mathbb{R}$ and $\tilde J\in L^1(\mathbb{R})$, we get
\[
    \int_\mathbb{R}\tilde J(y)\Big(m(t,x-y)-m(t,x)\Big)dy\le 2\int_\mathbb{R}\tilde J(y)dy.
\]
Therefore, combining the estimates of $I$ and $II$, for all $(t,x)\in [t_0,\infty)\times(x_1(t),+\infty)$, we achieve
\[
    \mathcal{D}[m](t,x)\le \mathcal{C}_1.
\]
where $\mathcal{C}_1=\max\left\{\frac{2\mathcal{J}_1\mathcal{C}_0}{(1-\ln 2)^\alpha}, 2\int_{|y|\ge 1} J(y)dy\right\}$.   
\end{proof}

%-----
\begin{lemma}\label{apld} 
Let $\tilde p:= \frac{2}{\beta(\alpha+1)-1}$ and $\varrho\in (0,1)$. 
Then there is a time $t^\ddagger\ge t^\star$ such that for all $t\ge t^\ddagger$ and $x\ge x_\varrho(t)$, we have
\[
	\mathcal{D}[m](t,x)
\le \mathcal{C}_2\frac{w(t,x)}{(1-\ln w(t,x))^{\frac{p+\beta-1}{\beta}}}+\mathcal{C}_3\frac{w(t,x)}{(1-\ln w(t,x))^\alpha}+\mathcal{C}_4\frac{w(t,x)}{(1-\ln w(t,x))^{\frac{p-\tilde p}{\beta}}}.
\]
for some positive constants $\mathcal{C}_2$, $\mathcal{C}_3$ and $\mathcal{C}_4$.
\end{lemma}

\begin{proof}[Proof of Lemma \ref{apld}]
 By \eqref{apdx}, we get 
\begin{equation}\label{xrx1}
    x_\varrho(t)-x_1(t)\to +\infty\quad \text{as }t\to+\infty.
    \end{equation}
Thus, there exists $t_1\ge t^\star$, such that for all $t\ge t_1$, we have
\[
    x_\varrho(t)-x_1(t)\ge\max\{L,2\},
\]
where the constant $L$ is defined in \eqref{v0x}. 
Since $0< m \le 1$, $m(t,x)=w(t,x)$ on $[t^\star,+\infty)\times [x_1(t),+\infty)$ and $m(t,\cdot)$ is non-increasing on $\mathbb{R}$, for all $t\ge t_1$ and $x\ge x_\varrho(t)$, we write  
\[
\begin{aligned}
\mathcal{D}[m](t,x)&\le\int^{+\infty}_{x-x_1(t)}J(y)dy +\int^{x-x_1(t)}_{1}J(y)\Big(w(t,x-y)-w(t,x)\Big)dy\\
&+ P.V. \int_{-1}^{1}J(y)\Big(w(t,x-y)-w(t,x)\Big)dy:=I_1+I_2+I_3.
\end{aligned}
\]
\par
$\#$ \textbf{Let us estimate $I_1$.} By Hypothesis \ref{ker} and L'H\^opital's rule, we have
\[
I_1\le  \mathcal{J}_0 \int_{x-x_1(t)}^{+\infty}e^{-y^\beta}dy\lesssim (x-x_1(t))^{1-\beta}e^{-(x-x_1(t))^\beta}.
\]
Let us denote $V_0(x):=(1-\ln v_0(x))^{\alpha+1}-1$ on $ (L,+\infty)$. By a direct calculation, for $x\in (L,+\infty)$, we get
\[
V_0'(x)= -(\alpha+1)(1-\ln v_0(x))^\alpha \frac{v_0'}{v_0}(x),
\]
and 
\[
V_0''(x)=(\alpha+1)(1-\ln v_0(x))^{\alpha-1}\left[ \alpha \left(\frac{v_0'}{v_0}\right)^2(x)-(1-\ln v_0(x))\left(\frac{v_0'}{v_0}\right)'(x)\right].
\]
Since $\frac{v_0'}{v_0}=- \frac{\beta}{x^{1-\beta}}+\frac{p}{x}$ and $\left(\frac{v_0'}{v_0}\right)'=\frac{\beta(1-\beta)}{x^{2-\beta}}- \frac{p}{x^2}$ for $x\in (L,+\infty)$, it follows that for $x\in (L,+\infty)$,
\[
V_0'(x) >0,
\]
and, by $\alpha\beta+\beta-1<0$, there is $\tilde L\ge L$ such that for all $x\ge \tilde L$, we have 
\[
\begin{aligned}
&\alpha \left(\frac{v_0'}{v_0}\right)^2(x)-(1-\ln v_0(x))\left(\frac{v_0'}{v_0}\right)'(x)\\
&=\frac{\beta}{x^{2-2\beta}}\left[\alpha\beta\left(1-\frac{p}{\beta}\frac{1}{x^\beta}\right)^2-(1-\beta)\left(1+\frac{1-\ln C_L-p\ln x}{x^\beta}\right)\left(1-\frac{p}{\beta(1-\beta)}\frac{1}{x^\beta}\right)\right]\\
&= \frac{\beta}{x^{2-2\beta}}\left(\alpha\beta+\beta-1+O_{x\to +\infty}\left(\frac{\ln x}{x^\beta}\right)\right)<0,
\end{aligned}
\]
whence, for all $x\ge \tilde L$, we have
\[
V_0''(x)\le 0.
\]
By \eqref{xrx1}, there is $\tilde t\ge t^\star$ such that $x_\varrho(t)-x_1(t)\ge \tilde L$ and $x_1(t)\ge \tilde L$ for all $t\ge \tilde t$,
and thus, for all $t\ge \tilde t$ and $x\ge x_\varrho(t)$, we have
\[
    V_0(x-x_1(t))+V_0(x_1(t))\ge V_0(x).
\]
By the definition of $x_1$, we have 
\[
    V_0(x_1(t))=[1-\ln v_0(x_1(t))]^{\alpha+1}-1=\rho (\alpha+1)t.
\]
Thus, it follows from $V_0$ being strictly increasing on $(L,+\infty)$ that, for all $t\ge\tilde t$ and $x\ge x_\varrho(t)$, we have
\begin{equation}\label{asx-x_1}
\begin{aligned}
    x-x_1(t)\ge V_0^{-1}\{V_0(x)-V_0(x_1(t))\}
    &=V_0^{-1}\{(1-\ln v_0(x))^{\alpha+1}-\rho(\alpha+1)t-1\}    \\
    &=V_0^{-1}\{(1-\ln w(t,x))^{\alpha+1}-1\} \\
    &=v_0^{-1}\{w(t,x)\}\\
    &\ge \left(\ln \frac{C_L}{w(t,x)}+\frac{p}{\beta}\ln \ln \frac{e^{L^\beta}}{w(t,x)}\right)^\frac{1}{\beta},
\end{aligned}
\end{equation}
thanks to \eqref{aiv0}.
By \eqref{xrx1}, one has $x-x_1(t)\ge x_\varrho(t)-x_1(t)\ge \left(\frac{1}{\beta}-1\right)^\frac{1}{\beta}$ for all $t\ge \tilde t$ and $x\ge x_\varrho(t)$, up to enlarge $\tilde t$ if necessary. It follows from the fact that the function $y\mapsto y^{1-\beta}e^{-y^\beta}$ is decreasing for $y\ge \left(\frac{1}{\beta}-1\right)^\frac{1}{\beta}$ and $w(t,\cdot)< 1$ on $[x_\varrho(t),+\infty)$ that for all $t\ge \max\{\tilde t, t_1\}$ and $x\ge x_\varrho(t)$,
\begin{equation}\label{alubi1}
   I_1\le \mathcal{C}_2(1-\ln w(t,x))^{\frac{1-\beta-p}{\beta}}w(t,x),
    \end{equation}
for some constant $\mathcal{C}_2>0$.
\par
$\#$ \textbf{Let us estimate $I_2$.} By using the change of variables $z=\frac{y}{x}$,  $I_2$ can be written as 
\[
    I_2=x\int^{\frac{x-x_1(t)}{x}}_{\frac{1}{x}}J(xz)\Big(w(t,x(1-z))-w(t,x)\Big) dz.
\]
We claim that there is a time $t_2>0$ such that for all $t\ge t_2$ and $x\ge x_\varrho(t)$,
\[
\frac{1}{x}\le \frac{\ln^\frac{1}{\beta}(x^q) }{x}\le \frac{x-x_1(t)}{x},
\]
for $q>1$ to be determined. Indeed, since $w(t,\cdot)$ is non-increasing and convex on $[x_1(t),+\infty)$, we have
\[
1=w(t,x_1(t))\ge w\left(t,x_1(t)+2\ln^\frac{1}{\beta}(x_1^q(t)) \right)\ge 1+2 \left[\ln^\frac{1}{\beta}(x_1^q(t)) \right]\frac{\partial w}{\partial x}(t,x_1(t)).
\]
By \eqref{wx0} and the facts that $\alpha\beta+\beta-1<0$ and $\lim_{t\to \infty}x_1(t)= +\infty$, we notice that
\[
\left[\ln^\frac{1}{\beta}(x_1^q(t)) \right]\frac{\partial w}{\partial x}(t,x_1(t))=-\left[\ln^\frac{1}{\beta}(x_1^q(t)) \right] x_1^{\alpha\beta+\beta-1}(t)\left[1+O_{t\to +\infty}\left(\frac{\ln x_1(t)}{x_1^\beta(t)}\right)\right]\to 0\quad \text{as }t\to \infty.
\]
It follows that 
\[
w\left(t,x_1(t)+2\ln^\frac{1}{\beta}(x_1^q(t)) \right)\to 1\quad \text{as }t\to \infty.
\]
Thus, since $0<\varrho<1$, there is $t_2>t^\star$ such that for all $t\ge t_2$,
\[
w\left(t,x_1(t)+2\ln^\frac{1}{\beta}(x_1^q(t)) \right)\ge\varrho= w(t,x_\varrho(t)),
\]
which implies that for $t\ge t_2$,
\[
x_\varrho(t)\ge x_1(t)+2\ln^\frac{1}{\beta}(x_1^q(t)),
\]
since $w(t,\cdot)$ is non-increasing on $[x_1(t),+\infty)$. Thus, for all $t\ge t_2$ and $x\ge x_\varrho(t)$, up to enlarging $t_2$ if necessary, we have
\[
\frac{1}{x}\le \frac{\ln^\frac{1}{\beta}(x^q) }{x}\le \frac{x-x_1(t)}{x}.
\]
Therefore, since $w>0$, denoting $B=\frac{\ln^\frac{1}{\beta} (x^q)}{x}$, for all $t\ge t_2$ and $x\ge x_\varrho(t)$, we write
\[
\begin{aligned}
    I_2\le  x\int^{B}_{\frac{1}{x}}J(xz)\Big(w(t,x(1-z))  -w(t,x)\Big)dz+x\int^{\frac{x-x_1(t)}{x}}_{B}J(xz)w(t,x(1-z)) dz:=I_4+I_5.
\end{aligned}
\]

$\#\#$ \textbf{Let us estimate $I_4$.} 
Before estimating $I_4$, we introduce a bounded function $H_{t,x}$. Recall $C_{\alpha\beta}=1-\beta+\alpha(\alpha+1)>0$ in Lemma \ref{asuble1}. Since $\alpha\beta+\beta-1<0$, one may notice that there is $t_3>t_2$ such that 
\[1-\beta x^{\alpha\beta+\beta-1}\ln^\frac{1}{\beta}(x^q) - C_{\alpha,\beta}  x^{2\alpha\beta+2\beta-2}\ln^\frac{2}{\beta}(x^q)>\frac{1}{2},\]
for all $t\ge t_3$ and $x\ge x_\varrho(t)$. 
Let $H_{t,x}$ be defined on $\left[0,\frac{\ln^\frac{1}{\beta} (x^q)}{x}\right]$ as 
\[
  H_{t,x}(z) :=\frac{\exp\left\{ \beta x^{\alpha\beta+\beta}   z+ C_{\alpha,\beta} x^{2\alpha\beta+2\beta} z^2\right\}}{ \left(  1-\beta x^{\alpha\beta+\beta}   z-  C_{\alpha,\beta}  x^{2\alpha\beta+2\beta} z^2\right)^{\alpha}},
\]
where $t\ge t_3$ and $x\ge x_\varrho(t)$. Then, for all $t\ge t_3$ and $x\ge x_\varrho(t)$, since $H_{t,x}$ is increasing on $\left[0,\frac{\ln^\frac{1}{\beta} (x^q)}{x}\right]$, we get, for $z\in \left[0,\frac{\ln^\frac{1}{\beta} (x^q)}{x}\right]$,
\begin{equation*}
   0 < H_{t,x}(z)\le \frac{\exp\left\{ \beta \frac{\ln^\frac{1}{\beta} (x^q)}{x^{1-\alpha\beta-\beta}}+ C_{\alpha,\beta} \frac{\ln^\frac{2}{\beta} (x^q)}{x^{2-2\alpha\beta-2\beta}}\right\}}{ \left[  1-\left(\beta  \frac{\ln^\frac{1}{\beta} (x^q)}{x^{1-\alpha\beta-\beta}}+ C_{\alpha,\beta} \frac{\ln^\frac{2}{\beta} (x^q)}{x^{2-2\alpha\beta-2\beta}}\right)\right]^{\alpha}}\to 1 \quad \text{as }x\to +\infty,
\end{equation*}
and it follows that $H_{t,x}$ is bounded uniformly on $\left[0,\frac{\ln^\frac{1}{\beta} (x^q)}{x}\right]$ for all $t\ge t_3$ and $x\ge x_\varrho(t)$, up to enlarging $t_3$ if necessary.
\par
In view of $I_4$, by Taylor's theorem and the convexity of $w(t,\cdot)$ on $[x_1(t),+\infty)$, $I_4$ can be written as 
\[
\begin{aligned}
    I_4\le-x^2\int^{B}_{\frac{1}{x}}\int_0^1 zJ(xz)\frac{\partial w}{\partial x}\Big(t,x(1-\tau z)\Big) d\tau dz\le -x^2\int^{B}_{\frac{1}{x}}zJ(xz)\frac{\partial w}{\partial x}\Big(t,x(1-z)\Big)  dz.
\end{aligned}
\]
By \eqref{J2} and \eqref{wx0}, we write
    \[
        -zJ(xz)\frac{\partial w}{\partial x}\Big(t,x(1-z)\Big) 
       \le  \mathcal{J}_0 ze^{-x^\beta z^\beta}\frac{w(t,x(1-z))}{[1-\ln w(t,x(1-z))]^\alpha}\varphi_0'(x(1-z))[1+\varphi_0(x(1-z))]^\alpha.
    \]
    Since $L\ge \left(\frac{e}{C_L}\right)^\frac{1}{p}$ and $\varphi_0(x)=-\ln v_0(x)=x^\beta-p\ln x-\ln C_L$ for all $x> L$, one has $\varphi_0'(x)=\beta x^{\beta-1}-px^{-1}\le \beta x^{\beta-1}$ and $[1+\varphi_0(x)]^\alpha\le x^{\alpha\beta}$   for all $x> L$, and it follows from the definition of $w$ that
    \[
    \begin{aligned}
        -zJ(xz)\frac{\partial w}{\partial x}\Big(t,x(1-z)\Big)\le  \mathcal{J}_0\beta x^{\alpha\beta+\beta-1} \frac{w(t,x)}{(1-\ln w(t,x))^\alpha}ze^{-x^\beta z^\beta}\frac{w(t,x(1-z))}{w(t,x)}\left(\frac{1-\ln w(t,x)}{1-\ln w(t,x(1-z))}\right)^\alpha .
    \end{aligned}
    \]
By Lemma \ref{asuble1}, we have
\[
    \begin{aligned}
        -zJ(xz)\frac{\partial w}{\partial x}\Big(t,x(1-z)\Big)
        \le  \mathcal{J}_0\beta  \frac{x^{\alpha\beta+\beta-1} w(t,x)}{(1-\ln w(t,x))^\alpha}ze^{-x^\beta z^\beta}
        H_{t,x}(z).
    \end{aligned}
    \]
Thus, since $H_{t,x}$ is bounded uniformly on $\left[0,B\right]$ for all $t\ge t_3$ and $x\ge x_\varrho(t)$, we get for all $t\ge t_3$ and $x\ge x_\varrho(t)$,
\[
\begin{aligned}
    I_4\le \mathcal{J}_0\beta \frac{ x^{\alpha\beta+\beta+1}w(t,x)}{(1-\ln w(t,x))^\alpha}\int_{\frac{1}{x}}^{B} z e^{-x^\beta z^\beta}H_{t,x}(z)dz
    &\lesssim   \frac{x^{\alpha\beta+\beta+1}w(t,x)}{(1-\ln w(t,x))^\alpha}\int_{\frac{1}{x}}^{B} ze^{-x^\beta z^\beta}dz.
    \end{aligned}
\]
Then, since $y\mapsto ye^{-y^\beta}\in L^1([1,+\infty))$, by using the change of variables $y=xz$, for all $t\ge t_3$ and $x\ge x_\varrho(t)$, we have 
\begin{equation}\label{asubi5}
    I_4\le C_4 x^{\alpha\beta+\beta-1} \frac{w(t,x)}{(1-\ln w(t,x))^\alpha},
\end{equation}
for some positive constant $C_4$.

\par
$\#\#$ \textbf{Let us estimate $I_5$.} By \eqref{J2} and the definition of $w$, we write 
\[
\begin{aligned}
    I_5=x\int_{B}^{\frac{x-x_1(t)}{x}}J(xz)w(t,x(1-z))dz\le \mathcal{J}_0x\int_{B}^{\frac{x-x_1(t)}{x}}\exp\left\{ -x^\beta z^\beta+\ln w(t,x(1-z))\right\}dz.
\end{aligned}
\]
By $\beta\in(0,1)$ and Lemma \ref{lewxx}, for all $x\ge x_1(t)$ and $z\in\left(\frac{1}{x},\frac{x-x_1(t)}{x}\right)$, one may get 
\[
\beta(1-\beta) x^\beta z^{\beta-2}+ x^2\frac{\partial^2}{\partial x}[ \ln w(t,x(1-z))]\ge 0,
\]   
it follows that, for fixed $(t,x)\in [t^\star,\infty)\times[x_1(t),+\infty)$, $z\mapsto -x^\beta z^\beta+\ln w(t,x(1-z))$ is convex on $\left(\frac{1}{x},\frac{x-x_1(t)}{x}\right)$. Thus, for all $z\in \left[B,\frac{x-x_1(t)}{x}\right]$, it follows that
\[
 -x^\beta z^\beta+\ln w(t,x(1-z)) \le \max\left\{ -x^\beta B^\beta+\ln w(t,x(1-B)),-(x-x_1(t))^\beta \right\}.
\]
By Lemma \ref{asuble1} and $B=\frac{\ln^\frac{1}{\beta}(x^q) }{x}$, we obtain
\[
-x^\beta  B^\beta+\ln  w(t,x(1-B))\le -\ln (x^q)+\ln w(t,x)+ \beta x^{\alpha\beta +\beta-1}\ln^\frac{1}{\beta}(x^q)+ C_{\alpha,\beta}  x^{2\alpha\beta+2\beta-2} \ln^\frac{2}{\beta}(x^q),
\]
and thus, since $\alpha\beta+\beta-1<0$ and $q>1$, there is $t_4>t_2$ such that for all $x\ge x_\varrho(t)$ and $t\ge t_4$, we have
\begin{equation}\label{asubi61}
x\exp\left\{ -x^\beta B^\beta+\ln w(t,x(1-B))\right\}\lesssim  x^{1-q} w(t,x)\lesssim \frac{w(t,x)}{(1-\ln w(t,x))^\frac{q-1}{\beta}}.
\end{equation}
On the other hand, for $t$ large enough and $x\ge x_\varrho(t)$, we have
\[
x\le (x-x_1(t))^{\tilde p},
\]
where $\tilde p=\frac{2}{1-\beta-\alpha\beta}>2$.
Indeed, since $x\mapsto x-(x-x_1(t))^{\tilde p}$ is decreasing for $t$ large enough and $x\ge x_\varrho(t)$, by \eqref{Xte} and \eqref{apdx}, there is $t_5>t_2$ such that for all $t\ge t_5$ and $x\ge x_\varrho(t)$,
\[
x-(x-x_1(t))^{\tilde p}\le x_\varrho(t)-(x_\varrho(t)-x_1(t))^{\tilde p}\le C' t^{\frac{1}{\beta(\alpha+1)}}-C'' t^{\frac{2}{\beta(\alpha+1)}}\le 0,
\]
for some positive constants $C'$ and $C''$.
It follows from \eqref{asx-x_1} that for all $t\ge t_5$ and $x\ge x_\varrho(t)$, up to enlarge $t_5$ if necessary,
\begin{equation}\label{asubi62}
x\exp\left\{ -x^\beta B^\beta+\ln w(t,x(1-B))\right\}\le (x-x_1(t))^{\tilde p } e^{-(x-x_1(t))^\beta}\lesssim \frac{w(t,x)}{(1-\ln w(t,x))^{\frac{p-\tilde p}{\beta}}}.
\end{equation}
As a result, collecting \eqref{asubi61} and \eqref{asubi62} and taking $q=\alpha\beta+1$, for all $t\ge \max\{t_4,t_5\}$ and $x\ge x_\varrho(t)$, we achieve 
\begin{equation}\label{asubi6}
    I_5\le  C_5\frac{w(t,x)}{(1-\ln w(t,x))^\alpha}+C_6\frac{w(t,x)}{(1-\ln w(t,x))^{\frac{p-\tilde p}{\beta}}},
\end{equation}
for some positive constant $C_5$ and $C_6$.
\par
Therefore, collecting \eqref{asubi5} and \eqref{asubi6}, since $\alpha\beta +\beta-1<0$, for all $t\ge \max\{t_3,t_4,t_5\}$ and $x\ge x_\varrho(t)$, we arrive at
\begin{equation}\label{alubi2}
    I_2\le  (C_4+C_5)\frac{w(t,x)}{(1-\ln w(t,x))^\alpha}+C_6\frac{w(t,x)}{(1-\ln w(t,x))^{\frac{p-\tilde p}{\beta}}}.
\end{equation}
\par 
$\#$ \textbf{Let us estimate $I_3$.} Since $J$ is symmetric and $w(t,\cdot)$ is $C^2$ on $[x_1(t),+\infty)$, we have
\[
\begin{aligned} 
    I_3=\frac{1}{2}\int_{-1}^{1}J(y)\Big(w(t,x+y)+w(t,x-y)-2w(t,x)\Big)dy
    &=\frac{1}{2}\int_{-1}^{1}\int_0^1\int_0^1\frac{\partial^2 w}{\partial x^2}(t,x+\tau \theta y)\theta y^2J(y)d\tau d\theta dy\\
    &\le \frac{1}{2}\sup_{|z|\le 1}\frac{\partial^2 w}{\partial x^2}(t,x+z)\int_{-1}^1 y^2 J(y)dy,
\end{aligned}
\]
and it follows from Hypothesis \ref{ker}, Lemma \ref{asuble1} and \eqref{apubwxxe} that, for all $t\ge t^\star$ and $x\ge x_\varrho(t)$, 
\begin{equation}\label{alubi3}
    I_3\le \mathcal{J}_1\mathcal{C}_0\sup_{|z|\le 1}\frac{w(t,x+z)}{(1-\ln w(t,x+z))^\alpha}\le \mathcal{J}_1\mathcal{C}_0\frac{w(t,x-1)}{(1-\ln w(t,x-1))^\alpha} \le C_3 \frac{w(t,x)}{(1-\ln w(t,x))^\alpha},
\end{equation}
 for some constant $C_3>0$,
\par 
Let us denote $t^\ddagger:=\max\{t^\star, t_1,t_2,t_3,t_4,t_5,\tilde t\}$. Therefore, owing to \eqref{alubi1}, \eqref{alubi2} and \eqref{alubi3}, for all $t\ge t^\ddagger$ and $x\ge x_\varrho(t)$, we achieve
\[
	\mathcal{D}[m](t,x)
\le \mathcal{C}_2\frac{w(t,x)}{(1-\ln w(t,x))^{\frac{p+\beta-1}{\beta}}}+\mathcal{C}_3\frac{w(t,x)}{(1-\ln w(t,x))^\alpha}+\mathcal{C}_4\frac{w(t,x)}{(1-\ln w(t,x))^{\frac{p-\tilde p}{\beta}}}.
\]
where $\mathcal{C}_3:= C_3+C_4+C_5$ and $\mathcal{C}_4:=C_6$.
\end{proof}

\subsubsection{Proof of the upper bound in Theorem \ref{ei1}}
\begin{proposition}\label{apubp}
Assume $p\ge\tilde p +\alpha\beta$. There is $\bar \rho\ge r$ such that, for all $\rho\ge\bar \rho$, we have
\begin{equation}
    \label{apubp60}
\frac{\partial m}{\partial t}(t,x)-\mathcal{D}[m](t,x)-f(m(t,x))\ge 0 \quad \text{for all } t\ge \bar t:=\max\{t_0,t^\ddagger\}\text{ and } x> x_1(t).
\end{equation}

\end{proposition}
\begin{proof}[Proof of Proposition \ref{apubp}]
     By Hypothesis \ref{fu} and the definition of $m$, for $x> x_1(t)$, we have 
     \begin{equation}
         \label{apubmf}
\frac{\partial m}{\partial t}(t,x)= \frac{\rho w(t,x)}{(1-\ln w(t,x))^\alpha}\quad \text{and}\quad f(m(t,x))\le \frac{rw(t,x)}{(1-\ln w(t,x))^\alpha}.  
     \end{equation}

Choosing $\rho>\rho_0:=r+2(1+\ln 2)^\alpha\mathcal{C}_1$ and $\varrho=\frac{1}{2}$, we have
\[
    \frac{\varrho}{(1-\ln\varrho)^\alpha}=\frac{1}{2(1+\ln 2)^\alpha}\ge \frac{\mathcal{C}_1}{\rho -r},
\]
where $\mathcal{C}_1$ is defined in Lemma \ref{fdm}. By definition of $x_\varrho$, we obtain $w(t,x_{\varrho}(t))=\varrho$ and $\varrho\le w(t,x)\le 1$ for all $x_1(t)\le x\le x_{\varrho}(t)$.  Thus, since the function $y\mapsto\frac{y}{(1-\ln y)^\alpha}$ is increasing on $y\in(0,1) $, we have, for $x_1(t)\le x\le x_{\varrho}(t)$,
\[
\frac{w(t,x)}{(1-\ln w(t,x))^\alpha}\ge\frac{w(t,x_{\varrho}(t))}{(1-\ln w(t,x_{\varrho}(t)))^\alpha}=\frac{\varrho}{(1-\ln\varrho)^\alpha}\ge\frac{\mathcal{C}_1}{\rho -r}.
\]
As a result, by \eqref{apubmf} and Lemma \ref{fdm}, for all $t\ge t_0$ and $x_1(t)< x\le x_{\varrho}(t)$, we have
\begin{equation}
    \label{apubp62}
\frac{\partial m}{\partial t}(t,x)-\mathcal{D}[m](t,x)-f(m(t,x))\ge(\rho-r)\frac{w(t,x)}{(1-\ln w(t,x))^\alpha}-\mathcal{C}_1\ge(\rho-r)\frac{\mathcal{C}_1}{\rho-r} -\mathcal{C}_1= 0.
\end{equation}
\par
On the other hand, in view of Lemma \ref{apld}, by \eqref{apubmf}, for all $t\ge t^\ddagger$ and $x\ge x_{\varrho}(t)$, we have
\[
	\begin{aligned}
&\frac{\partial m}{\partial t}(t,x)-\mathcal{D}[m](t,x)-f(m(t,x))\\
&\ge  \frac{w(t,x)}{(1-\ln w(t,x))^\alpha}\left(
\rho-r-\mathcal{C}_2(1-\ln w(t,x))^{\frac{\alpha\beta-\beta+1-p}{\beta}
}-\mathcal{C}_3-\mathcal{C}_4(1-\ln w(t,x))^{\frac{\tilde p +\alpha\beta-p}{\beta}}\right).
\end{aligned}
\]
As a result, recalling that $\tilde p=\frac{2}{1-\beta-\alpha\beta}$, by assuming $p\ge\tilde p +\alpha\beta\ge \alpha\beta -\beta+1$ and choosing $\rho\ge r+ \mathcal{C}_2+\mathcal{C}_3+\mathcal{C}_4$,   we have for all $t\ge t^\ddagger$ and $x\ge x_{\varrho}(t)$,
\begin{equation}
    \label{apubp63}
\frac{\partial m}{\partial t}(t,x)-\mathcal{D}[m](t,x)-f(m(t,x))\ge0.
\end{equation}
\par 
Collecting \eqref{apubp62} and \eqref{apubp63}, for all $\rho\ge \bar \rho:=\max\{\rho_0,r+ \mathcal{C}_2+\mathcal{C}_3+\mathcal{C}_4\}$, $m$ is then a super-solution  to \eqref{oeq1} for all $t\ge \bar t=\max\{t_0,t^\ddagger\}$ and $x> x_1(t)$.

\end{proof}

Let us denote 
\[
    \tilde m(t,x)= m(t+\bar t, x)\quad \text{for all }(t,x)\in\mathbb{R}^+\times \mathbb{R}.
\]
Thus, when $p=\tilde p +\alpha\beta$ and $\rho=\bar \rho$, for all $t\in\mathbb{R}^+$ and $x\in(x_1(t+\bar t),+\infty)$, by Proposition \ref{apubp}, the function $\tilde m$ is a super-solution to \eqref{oeq1}.
Now, we are ready to prove the upper bound in Theorem \ref{ei1}.

\par
In view of Hypothesis \ref{ic1}, by \eqref{apubv0u0}, \eqref{apubwv0} and the definition of $m$, one has that $u_0\le v_0\le m(\bar t, \cdot )=\tilde m(0,\cdot)$. 
Thus, since $u(t,x)\le 1=\tilde m(t,x)$ for all $t>0$ and $x\le x_1(t+\bar t)$, and $\tilde m$ is a super-solution to \eqref{oeq1} in $(x_1(t+\bar t),+\infty)$, the comparison principle yields that $u(t,x)< \tilde m(t,x)$ for all $(t,x)\in \mathbb{R}^+\times \mathbb{R}$. Thus, for any $\lambda\in(0,1)$, we have
\[
u(t,x_{\lambda}(t+\bar t))<\tilde m(t,x_\lambda(t+\bar t))=\lambda, \quad \forall t>0.
\]
Therefore, for $t>0$, we have
$
X_\lambda(t)\le x_\lambda(t+\bar t)\lesssim_{\lambda} t^{\frac{1}{\beta(\alpha+1)}}, 
$
which completes the proof of the upper bound of the level set.

\subsection{The lower bound}
Here, we prove the lower bound in Theorem \ref{ei1}. This is achieved by deriving a smooth sub-solution to \eqref{oeq1}.

\subsubsection{Definition of the sub-solution and basic computations}
Let us define
\begin{equation}\label{v0}
v_0(x):=\left\{
\begin{aligned}
&1,&x<0,\\
&e^{-x^\beta},& x\ge 0.
\end{aligned}\right.
\end{equation}
As in the previous subsection, let us define 
	\[
	w(t,x):=\exp{\left\{1-\left[(1-\ln v_0(x))^{\alpha+1}-\rho(\alpha+1)t\right]^{\frac{1}{\alpha+1}}\right\}},
\]		
which solves the Cauchy problem
			\[\
			\left\{
   \begin{aligned}
& \frac{\partial w}{\partial t}(t,x)=\rho\frac{w(t,x)}{(1-\ln w(t,x))^\alpha},&(t,x)\in\mathbb{R}^+\times\mathbb{R},\\
& w(0,x)=v_0(x),&x\in\mathbb{R},\qquad\qquad
			\end{aligned}
   \right.
		\]
 where $\rho>0$ is to be determined. 
Again, $w$ is not defined for all times: for any $x\in\mathbb{R}$, $w(t,x)$ is only defined in $t\in\left[0,\frac{(1-\ln v_0(x))^{\alpha+1}}{\rho(\alpha+1)}\right)$. As for the construction of the super-solution, we first collect some useful technical results.   We summarize everything we need in the following Lemma, whose proof can be found in Appendix \ref{appendix-alg-acc-low}.
\begin{lemma}\label{apubl1}
For all $t\ge 0$, define 
\begin{equation}\label{aplbdx} 
x_\Lambda(t):=\left\{-1+\left[(1-\ln \Lambda)^{\alpha+1}+\rho(\alpha+1)t\right]^{\frac{1}{\alpha+1}}\right\}^{\frac{1}{\beta}} \quad\text{for any } \Lambda\in(0,1).
\end{equation}
\begin{itemize}
    \item[(i).]
 For any $\Lambda\in\left(0,e^{-1}\right)$,
 \[
     x_\Lambda(t)>1\quad \text{for all }t\ge 0.
 \]
 \item[(ii).] For any $0<\Lambda_1<\Lambda_2<1$, we have
 \[
 x_{\Lambda_1}(t)-x_{\Lambda_2}(t)\to +\infty\quad \text{as }t\to \infty.
 \]
 \item[(iii).] 
 For any $\Lambda\in (0,1)$ and $t\ge 0 $, the function $w(t,\cdot)$ is well-defined, non-increasing and convex on $[x_\Lambda(t),+\infty)$. Moreover, denoting $\varphi_0=-\ln v_0$, we have 

 \begin{equation}\label{alwx0}
    \frac{\partial w}{\partial x}=-\frac{w}{(1-\ln w)^\alpha}\varphi_0'(1+\varphi_0)^\alpha,
\end{equation}
and
 \begin{equation}
     \label{alwxe}
 \lim_{t\to \infty}\frac{\partial w}{\partial x}(t,x_\Lambda(t))=0,\quad \text{uniformly on }\Lambda\in(0,1).
 \end{equation}
  \item[(iv).] For any $\Lambda\in (0,1)$, we have the following estimate for $w$:
 \begin{equation}\label{allbwu}
    w(t,x)\le e^{-x^\beta+[\rho(\alpha+1)t]^\frac{1}{\alpha+1}} \quad\text{for all }(t,x)\in \mathbb{R}^+\times [x_\Lambda(t),+\infty).
\end{equation}
\end{itemize}
\end{lemma}
\par
To construct a sub-solution to equation \eqref{oeq1}, one may want to cut the function $w$ at the value $\varepsilon$. However, to handle a possible singularity of $J$ at zero, one requires the sub-solution to \eqref{oeq1} to be at least of class $C^2$ in $x$. 
As in \cite{bouin2021sharp}, we introduce the function 
\[
    g_\varepsilon:y\mapsto 3\left(1-\frac{y}{\varepsilon}+\frac{y^2}{3\varepsilon^2}\right)y,\quad \text{on }[0,\varepsilon],
\]
for some $\varepsilon>0$. Some direct calculations show that 
\[
    g_\varepsilon'(y)=3\left(1-\frac{y}{\varepsilon}\right)^2\quad \text{and}\quad g_\varepsilon''(y)=-\frac{6}{\varepsilon}\left(1-\frac{y}{\varepsilon}\right), \quad \text{for }y\in [0,\varepsilon].
\]
Notice that for $y\in[0,\varepsilon]$, we have
\[
    y\le g_\varepsilon(y)\le 3y.
\]
\par 
Equipped with the above preparations, for any fixed $\varepsilon\in\left(0,\varepsilon_0:=\min\left\{e^{-1},\xi_0\right\}\right)$ and $t\ge 0$, let us define 
\[
v(t,x):=\left\{
\begin{aligned}
&\varepsilon,& x< x_\varepsilon(t),\\
&g_\varepsilon(w(t,x)),& x\ge x_\varepsilon(t).
\end{aligned}
\right.
\] 
Notice that the function $v$ has a good regularity, that is, $v(\cdot,x)$ is $C^{1,1}(\mathbb{R}^+)$ for all $x\in \mathbb{R}$ and $v(t,\cdot)$ is $C^{2}(\mathbb{R})$ for all $t>0$. Unless otherwise specified, throughout the subsection, we assume that $t\in \mathbb{R}^+$. 
\par 
Some direct calculations show that for all $x<x_\varepsilon(t)$, we have
\begin{equation}\label{alvdd}
    \frac{\partial v}{\partial t}(t,x)=\frac{\partial v}{\partial x}(t,x)=\frac{\partial^2 v}{\partial x^2}(t,x)=0,
\end{equation}
while for all $x\ge x_\varepsilon(t)$, 
\begin{equation}\label{alvt}
    \frac{\partial v}{\partial t}(t,x)=g_\varepsilon'(w)\frac{\partial w}{\partial t}=3\rho \left(1-\frac{w(t,x)}{\varepsilon}\right)^2\frac{w(t,x)}{(1-\ln w(t,x))^\alpha}\le 3\rho\frac{w(t,x)}{(1-\ln w(t,x))^\alpha},
\end{equation}
\begin{equation}\label{alvx}
\begin{aligned}
    \frac{\partial v}{\partial x}(t,x)=g_\varepsilon'(w)\frac{\partial w}{\partial x}&=-3 \left(1-\frac{w(t,x)}{\varepsilon}\right)^2\frac{w(t,x)}{(1-\ln w(t,x))^\alpha}(1-\varphi_0(x))^\alpha\varphi_0'(x)\\
    &\ge-3 \frac{w(t,x)}{(1-\ln w(t,x))^\alpha}(1-\varphi_0(x))^\alpha\varphi_0'(x),
\end{aligned}
\end{equation}
and 
\begin{equation}\label{alvxx}
\begin{aligned}
    \frac{\partial^2 v}{\partial x^2}=g_\varepsilon''(w)\left(\frac{\partial w}{\partial x}\right)^2+g'(w)\frac{\partial^2 w}{\partial x^2}&=3\left(1-\frac{w}{\varepsilon}\right)\left[-\frac{2}{\varepsilon}\left(\frac{\partial w}{\partial x}\right)^2+\left(1-\frac{w}{\varepsilon}\right)\frac{\partial^2 w}{\partial x^2}\right]\ge-\frac{6}{\varepsilon}\left(\frac{\partial w}{\partial x}\right)^2. 
\end{aligned}
\end{equation}
Let us investigate the convexity of the function $v$. 
\begin{lemma}\label{aplblxc}
For all $t>0$, $v(t,\cdot)$ is convex on $\left(x_{\frac{\varepsilon}{3}}(t),+\infty\right)$.
\end{lemma}
 \begin{proof}[Proof of Lemma \ref{aplblxc}]
 By \eqref{alwx0}, for all $t>0$ and $x>x_{\frac{\varepsilon}{3}}(t)$, we get
 \begin{equation*}
		\frac{\partial^2 w}{\partial x^2}=\frac{w}{(1-\ln w)^{\alpha}}\Big\{(\varphi_0')^2 (1+\varphi_0)^{2\alpha}\Big((1-\ln w)^{-\alpha}
		+\alpha (1-\ln w)^{-(\alpha+1)}-\alpha(1+\varphi_0)^{-(\alpha+1)} \Big)-\varphi_0''(1+\varphi_0)^\alpha \Big\}.
\end{equation*}
     It then follows from \eqref{alwx0} and \eqref{alvxx} that 
     \[
     \begin{aligned}
         \frac{\partial^2 v}{\partial x^2}&\ge 3\left(1-\frac{w}{\varepsilon}\right)(\varphi'_0)^2(1+\varphi_0)^{2\alpha}\frac{w}{(1-\ln w)^{2\alpha}}\left(1-\frac{3}{\varepsilon}w\right)\ge 0
     \end{aligned}
     \]
     as soon as $0<w\le \frac{\varepsilon}{3}$. Therefore, for all $x> x_{\frac{\varepsilon}{3}}(t)$, we have $\frac{\partial^2 v}{\partial x^2}\ge 0$. 
 \end{proof}

\subsubsection{Estimation of \texorpdfstring{$\mathcal{D}[v](t,\cdot)$ on $\mathbb{R}$}{D[v](t,.) on R}}
Now, let us estimate the dispersal term $\mathcal{D}[v]$. 
We split space into three space sub-zones, as follows.
\begin{proposition}\label{aplbed}
Assume $\varepsilon\in (0,\varepsilon_0)$. For any $t>0$, we have the following estimate for $\mathcal{D}[v](t,\cdot)$.
\begin{itemize}
   \item[I.]  When $x\in(-\infty, x_\varepsilon(t))$, for all $B_1>1$, we have
    \begin{equation} \label{allbl1}
        \mathcal{D}[v]\ge-\varepsilon \mathcal{C}_1 B_1^{1-\beta}e^{-B_1^\beta}-\frac{\mathcal{C}_2}{\varepsilon}\left(\frac{\partial w}{\partial x}\right)^2(t,x_\varepsilon(t)),
    \end{equation}
     where $\mathcal{C}_1>0$ and $\mathcal{C}_2:=6\left(\mathcal{J}_1+\mathcal{J}_0\int_1^{+\infty}z^2e^{-z^\beta}dz\right)$.
  \item[II.]   When $x\in \left[x_\varepsilon(t), x_{\frac{\varepsilon}{3}}(t)+1\right]$, for all $B_2>1$, we have
    \begin{equation}\label{allbl2}
    \mathcal{D}[v]\ge-\frac{\mathcal{C}_2}{\varepsilon} \left(\frac{\partial w}{\partial x}\right)^2(t,x_\varepsilon(t)) -\mathcal{C}_3 \varepsilon B_2^{1-\beta}e^{-B_2^\beta},
    \end{equation}
where $\mathcal{C}_3>0$.
\item[III.] When $x\in\left( x_{\frac{\varepsilon}{3}}(t)+1,+\infty\right)$, we have
  \begin{equation}  \label{allbl3}
        \mathcal{D}[v]\ge \mathcal{C}_4 \frac{\partial v}{\partial x}(t,x),
\end{equation}
    where $\mathcal{C}_4:=\int_1^{+\infty}z J(z)dz$.
    \end{itemize}
\end{proposition}
\begin{proof}[Proof of Proposition \ref{aplbed}]
~~
\medskip
\par
\noindent
\# \textbf{Start with I: $x< x_\varepsilon(t)$.} Since $v(t,\cdot)=\varepsilon$ on $(-\infty, x_\varepsilon(t))$, we write
\[
    \mathcal{D}[v]=\int_{x_\varepsilon(t)-x}^{+\infty}J(z)\Big(v(t,x+z)-\varepsilon\Big)dz.
\]
In view of Hypothesis \ref{ker}, one may estimate the above integral on $x\in(-\infty,x_\varepsilon(t)-B_1]$ and $x\in(x_\varepsilon(t)-B_1,x_\varepsilon(t))$ for some $B_1>1$, respectively.
\medskip
\par
\#\# \textbf{Part 1: $x\in(-\infty,x_\varepsilon(t)-B_1]$.}
By the hypothesis on $J$ and L'H\^opital's rule, we get
\begin{equation} \label{aslbed1}
    \mathcal{D}[v]
    \ge-\varepsilon\mathcal{J}_0\int_{B_1}^{+\infty}e^{-z^\beta}dz
    \ge-\varepsilon\mathcal{C}_1B_1^{1-\beta}e^{-B_1^\beta},
\end{equation}
for some $\mathcal{C}_1>0$.
\medskip
\par\#\# \textbf{Part 2: $x\in(x_\varepsilon(t)-B_1,x_\varepsilon(t))$.}
We write 
\[
    \mathcal{D}[v]=\int_{x_\varepsilon(t)-x+B_1}^{+\infty}J(z)\Big(v(t,x+z)-\varepsilon\Big)dz+\int_{x_\varepsilon(t)-x}^{x_\varepsilon(t)-x+B_1}J(z)\Big(v(t,x+z)-\varepsilon\Big)dz:=I_1+I_2.
\]
\par 
\#\#\# \textbf{Let us estimate $I_1$.} Similarly to calculations in \eqref{aslbed1}, one may get
\[
    I_1\ge -\varepsilon \mathcal{J}_0\int_{x_\varepsilon(t)-x+B_1}^{+\infty}e^{-z^\beta}dz\ge -\varepsilon \mathcal{J}_0\int_{B_1}^{+\infty}e^{-z^\beta}dz\ge -\varepsilon \mathcal{C}_1 B_1^{1-\beta}e^{-B_1^\beta},
\]
where $\mathcal{C}_1>0$. 
\par\#\#\# \textbf{Let us estimate $I_2$.} By the fundamental theorem of calculus, since $v(t,x)=\varepsilon$ for $x< x_\varepsilon(t)$, we have
\[
    I_2=\int_{0}^{x_\varepsilon(t)-x+B_1}J(z)\Big(v(t,x+z)-v(t,x)\Big)dz=\int_{0}^{x_\varepsilon(t)-x+B_1}\int_0^1zJ(z)\frac{\partial v}{\partial x}(t,x+\tau z)dz.
\]
Since $\frac{\partial v}{\partial x}(t,x)=0$ for $x< x_\varepsilon(t)$, by \eqref{alvxx} and the convexity of $w$,  we obtain
\[
\begin{aligned}
     I_2&=\int_{0}^{x_\varepsilon(t)-x+B_1}\int_0^1zJ(z)\Big(\frac{\partial v}{\partial x}(t,x+\tau z)-\frac{\partial v}{\partial x}(t,x)\Big)dz&\\
     &=\int_{0}^{x_\varepsilon(t)-x+B_1}\int_0^1\int_0^1\tau z^2J(z)\frac{\partial^2 v}{\partial x^2}(t,x+\sigma\tau z)d\sigma d\tau dz\\
     &\ge-\frac{6}{\varepsilon}\left(\frac{\partial w}{\partial x}\right)^2(t,x_\varepsilon(t))\int_{0}^{x_\varepsilon(t)-x+B_1} z^2J(z) dz.  
\end{aligned}
\]
It follows from Hypothesis \ref{ker} and $z\mapsto z^2J(z)\in L^1([1,+\infty))$ that 
\[
    \int_{0}^{x_\varepsilon(t)-x+B_1} z^2J(z) dz\le\int_{0}^{1} z^2J(z) dz+\int_{1}^{+\infty} z^2J(z) dz\le \mathcal{J}_1+\mathcal{J}_0\int_1^{+\infty}z^2e^{-z^\beta}dz,
\]
and thus, we have
\[
    I_2\ge-\frac{6}{\varepsilon}\left(\frac{\partial w}{\partial x}\right)^2(t,x_\varepsilon(t))\left(\mathcal{J}_1+\mathcal{J}_0\int_1^{+\infty}z^2e^{-z^\beta}dz\right).
\]
Therefore, collecting the estimates of $I_1$ and $I_2$, one obtains
\[
    \mathcal{D}[v]\ge-\varepsilon \mathcal{C}_1 B_1^{1-\beta}e^{-B_1^\beta}-\frac{\mathcal{C}_2}{\varepsilon}\left(\frac{\partial w}{\partial x}\right)^2(t,x_\varepsilon(t)),
\]
where $\mathcal{C}_2:=6\left(\mathcal{J}_1+\mathcal{J}_0\int_1^{+\infty}z^2e^{-z^\beta}dz\right)$.
\medskip
\par
\noindent
    \# \textbf{Continue with II: $x_\varepsilon(t)\le x\le x_{\frac{\varepsilon}{3}}(t)+1$.}
By the definition of $\mathcal{D}[v]$, since $v(t,\cdot)$ is non-increasing on $\mathbb{R}$, for some $B_2>1$, we write
    \[
        \mathcal{D}[v](t,x)\ge P.V. \int_{-B_2}^{B_2} J(z)\Big(v(t,x+z)-v(t,x)\Big)dz+\int_{B_2}^{+\infty} J(z)\Big(v(t,x+z)-v(t,x)\Big)dz:=II_1+II_2. 
    \]
    \par
\#\# \textbf{Let us estimate $II_1$.} Since $v(t,\cdot)$ is $C^2$ on $\mathbb{R}$ and $\frac{\partial^2 v}{\partial x^2}(t,\cdot)=0$ on $(-\infty, x_\varepsilon(t))$, by using the symmetry of $J$ and Taylor's theorem, we get
\[
\begin{aligned}
    II_1=\frac{1}{2}\int_{-B_2}^{B_2} J(z)\Big(v(t,x+z)+v(t,x-z)-2v(t,x)\Big)dz&=\frac{1}{2}\int_{-B_2}^{B_2} \int_0^1\int_{-1}^1 \tau z^2J(z) \frac{\partial^2 v}{\partial x^2}(t,x+\sigma\tau z)d\sigma d\tau dz\\
  &\ge \frac{1}{2}\min_{\substack{|\xi|\le B_2 \\ x+\xi \ge x_\varepsilon(t)}} \frac{\partial^2 v}{\partial x^2}(t,x+\xi)  \int_{-B_2}^{B_2} \int_0^1\tau z^2J(z) d\tau dz.
\end{aligned}
\]
By \eqref{alvxx} and the convexity of $w(t,\cdot)$ on $(x_\varepsilon(t),+\infty)$, for $x+\xi\ge x_\varepsilon(t)$, one gets
\[
    \frac{\partial^2 v}{\partial x^2}(t,x+\xi)\ge -\frac{3}{\varepsilon}\left(\frac{\partial w}{\partial x}\right)^2(t,x+\xi)\ge -\frac{3}{\varepsilon}\left(\frac{\partial w}{\partial x}\right)^2(t,x_\varepsilon(t)).
\]
It then follows from Hypothesis \ref{ker} and \eqref{J2} that 
\[
\begin{aligned}
    II_1&\ge -\frac{3}{\varepsilon}\left(\frac{\partial w}{\partial x}\right)^2(t,x_\varepsilon(t))\left(\int_{|z|\le 1}z^2J(z)dz+\int_{1\le|z|\le B_2}z^2J(z)dz\right)\\
    &\ge -\frac{6}{\varepsilon}\left(\frac{\partial w}{\partial x}\right)^2(t,x_\varepsilon(t))\left(\mathcal{J}_1+\mathcal{J}_0\int_{1}^{+\infty}z^2e^{-z^\beta}dz\right).
\end{aligned}
\]
\par 
\#\# \textbf{Now, let us estimate $II_2$.} By \eqref{J2}, L'H\^opital's rule and $v
\le \varepsilon$, we have
\[
    II_2\ge -\mathcal{J}_0 v(t,x)\int_{B_2}^{+\infty} e^{-z^\beta}dz\ge -\mathcal{C}_3 \varepsilon B_2^{1-\beta}e^{-B_2^\beta}.
\]
for some constant $\mathcal{C}_3>0$.
Therefore, collecting the above estimates, we arrive at 
\[
    \mathcal{D}[v]\ge-\frac{\mathcal{C}_2}{\varepsilon} \left(\frac{\partial w}{\partial x}\right)^2(t,x_\varepsilon(t)) -\mathcal{C}_3 \varepsilon B_2^{1-\beta}e^{-B_2^\beta},
\]
where $\mathcal{C}_2=6\left(\mathcal{J}_1+\mathcal{J}_0\int_{1}^{+\infty}z^2e^{-z^\beta}dz\right)$.
\medskip
\par
\noindent
\# \textbf{Finally focus on III: $x> x_{\frac{\varepsilon}{3}}(t)+1$.}
    By the definition of $\mathcal{D}[v]$, since $v(t,\cdot)$ is non-increasing on $\mathbb{R}$, we write 
\[
    \mathcal{D}[v](t,x)\ge P.V.\int_{-1}^1J(z)\Big(v(t,x+z)-v(t,x)\Big)dz+\int_1^{+\infty}J(z)\Big(v(t,x+z)-v(t,x)\Big)dz:=III_1+III_2.
\]
\par 
\#\# \textbf{Let us estimate $III_1$}. Since $J$ is symmetric, by Taylor's theorem, we write
\[
    III_1=\frac{1}{2}\int_{-1}^1J(z)\Big(v(t,x+z)+v(t,x-z)-2v(t,x)\Big)dz=\frac{1}{2}\int_{-1}^1
\int_{0}^1\int_{-1}^{1} \tau z^2J(z) \frac{\partial^2 v}{\partial x^2}(t,x+\sigma \tau z)d\sigma d\tau dz.
\]
It follows from the convexity of $v(t,\cdot)$ on $\left(x_{\frac{\varepsilon}{3}}(t),+\infty\right)$ that for all $x> x_{\frac{\varepsilon}{3}}(t)+1$,
\begin{equation}
    III_1\ge 0.
\end{equation}

\par \#\# \textbf{Let us estimate $III_2$.} By the fundamental theorem of calculus, we get
\[
    III_2=\int_1^{+\infty}\int_0^1 zJ(z)\frac{\partial v}{\partial x}(t,x+\tau z)d\tau dz.
\]
Since $v(t,\cdot)$ is convex on $\left(x_{\frac{\varepsilon}{3}}(t),+\infty\right)$ and $\tau z>0$ for $z>0$ and $\tau>0$, one has $\frac{\partial v}{\partial x}(t,x+\tau z)>\frac{\partial v}{\partial x}(t,x)$. Thus, 
\[
    III_2\ge \frac{\partial v}{\partial x}(t,x)\int_1^{+\infty}z J(z)dz.
\]
Therefore, collecting the above estimates, for all $x>x_\frac{\varepsilon}{3}(t)+1$, we obtain
\[
    \mathcal{D}[v]\ge \mathcal{C}_4 \frac{\partial v}{\partial x}(t,x),
\]
where $\mathcal{C}_4:=\int_1^{+\infty}z J(z)dz$.

\end{proof}

\subsubsection{The function \texorpdfstring{$v$}{v} is a sub-solution to \texorpdfstring{\eqref{oeq1}}{eq.1} at later times}
\begin{proposition}\label{aplbps}
For all $0<\varepsilon\le \underline\varepsilon:=\min\{\frac{1}{2K},\varepsilon_0\}$ and $0<\rho\le\frac{r}{12}$, there is a time $\tilde t>0$ such that for all $t\ge \tilde t$, we have
   \begin{equation}\label{alprop}
        \frac{\partial v}{\partial t}(t,x)-\mathcal{D}[v](t,x)-f(v(t,x))\le 0\quad \text{for all }x\in\mathbb{R}.
   \end{equation}
\end{proposition}
\begin{proof} [Proof of Proposition \ref{aplbps}]
Before proving the proposition, we estimate $\frac{\partial v}{\partial t}$, $f(v)$ and $\mathcal{D}[v]$. For $\frac{\partial v}{\partial t}$, by the definition of $v$ and \eqref{alvt}, we obtain
\begin{equation}\label{alvt1}
\frac{\partial v}{\partial t}(t,x)\le\left\{
\begin{aligned}
&0, &x< x_\varepsilon(t),\\
&3\rho\frac{w(t,x)}{(1-\ln w(t,x))^\alpha},& x\ge x_\varepsilon(t).
\end{aligned}
\right.
\end{equation}
Recalling that $f(z)\ge \frac{rz}{(1-\ln z)^\alpha}(1-Kz)$ for $z\in(0,\xi_0)$, when $x< x_\varepsilon(t)$, we have
\[
f(v(t,x))\ge r \frac{v(t,x)}{(1-\ln v(t,x))^\alpha}(1-Kv(t,x))
= r (1-K\varepsilon)\frac{\varepsilon}{(1-\ln \varepsilon)^\alpha},
\]
thanks to $v(t,x)=\varepsilon$ for all $x< x_\varepsilon(t)$.
When $x\ge x_\varepsilon(t)$, it follows from $w(t,x)\le v(t,x)\le\varepsilon$ for $x\ge x_\varepsilon(t)$ that
\[
		f(v(t,x))\ge r (1-K \varepsilon)\frac{w(t,x)}{(1-\ln w(t,x))^\alpha}.
\]
As a result, we have
\begin{equation}\label{allbfv}
	f(v(t,x))\ge\left\{
	\begin{aligned}
&\frac{r(1-K\varepsilon) \varepsilon}{(1-\ln \varepsilon)^\alpha},&x<x_\varepsilon(t),\\
& \frac{r(1-K\varepsilon)w(t,x)}{(1-\ln w(t,x))^\alpha},&x\ge x_\varepsilon(t).
	\end{aligned}	
	\right.
\end{equation}
One may find that there is $B_0>1$, such that for all $z\ge B_0$, we have
\begin{equation}
    \label{aplbb0}
    z^{1-\beta}e^{-z^\beta}\le z^{-1}.
\end{equation}
To prove the proposition, we split into three space zones: $\left(-\infty,  x_{\varepsilon}(t)\right)$, $\left[ x_{\varepsilon}(t) , x_{\frac{\varepsilon}{3}}(t)+1\right]$ and $\left(x_{\frac{\varepsilon}{3}}(t)+1,+\infty\right)$. 
\medskip
\par\noindent
\# \textbf{Zone 1: $x\in \left(-\infty,  x_{\varepsilon}(t)\right)$.} 
In view of \eqref{allbl1}, let us choose $B_1:=\frac{2\mathcal{C}_1}{r}\frac{(1-\ln\varepsilon)^\alpha}{1-K\varepsilon}+B_0$. Then, by \eqref{aplbb0}, for all $x<x_\varepsilon(t)$, one gets 
\[
    \mathcal{D}[v]\ge-\frac{r}{2}\frac{(1-K\varepsilon)\varepsilon}{(1-\ln \varepsilon)^\alpha}-\frac{\mathcal{C}_2}{\varepsilon}\left(\frac{\partial w}{\partial x}\right)^2(t,x_\varepsilon(t)).
\]
In view of \eqref{alwxe}, there is a time $t_1>0$, depending on $\varepsilon$, such that for all $t\ge t_1$, we have
\[
    \left(\frac{\partial w}{\partial x}\right)^2(t,x_\varepsilon(t))\le \frac{r}{2\mathcal{C}_2}\frac{(1-K\varepsilon)\varepsilon^2}{(1-\ln\varepsilon)^\alpha},
\]
and then we conclude 
\[
    \mathcal{D}[v]\ge-\frac{r(1-K\varepsilon)\varepsilon}{(1-\ln \varepsilon)^\alpha}.
\]
As a result, by \eqref{alvdd} and \eqref{allbfv}, we arrive at 
\begin{equation}\label{allbi1}
    \frac{\partial v}{\partial t}(t,x)-\mathcal{D}[v](t,x)-f(v(t,x))\le 0\quad \text{for all }t\ge t_1\text{ and } x\le x_\varepsilon(t).
\end{equation}
\medskip
\par
\noindent
\# \textbf{Zone 2: $x\in\left[x_\varepsilon(t), x_{\frac{\varepsilon}{3}}(t)+1\right]$.}
We claim that there is a time $t^\#>0$, such that for all $t\ge t^\#$ and $x\le x_{\frac{\varepsilon}{3}}(t)+1$, we have
\[
    w(t,x)\ge \frac{\varepsilon}{6}.
\]
Indeed, by \eqref{apubl1}, there is a time $t^\#>0$ such that for $t\ge t^\#$,
\[
    x_{\frac{\varepsilon}{6}}(t)-x_{\frac{\varepsilon}{3}}(t)>1. 
\]
In view of \eqref{allbl2} and \eqref{aplbb0}, let us choose $B_2:=\frac{3\mathcal{C}_3}{\rho}  \left(1-\ln \frac{\varepsilon}{6}\right)^\alpha+B_0$ such that 
\[
    \mathcal{C}_3 \varepsilon B_2^{1-\beta}e^{-B_2^\beta} \le \mathcal{C}_3\varepsilon B_2^{-1}\le \frac{\frac{\varepsilon}{3}\rho}{\left(1-\ln \frac{\varepsilon}{6}\right)^\alpha}\le 2\rho\frac{w(t,x)}{(1-\ln w(t,x))^\alpha}\quad \text{ for all }t\ge t^\#\text{ and }x\le x_{\frac{\varepsilon}{3}}(t)+1.
\]
On the other hand, in view of \eqref{alwxe}, there is $t^*>0$ such that for $t\ge t^*$ we have 
\[
    \frac{\mathcal{C}_2}{\varepsilon}\left(\frac{\partial w}{\partial x}\right)^2(t,x_\varepsilon(t))\le\rho\frac{\frac{\varepsilon}{6}}{(1-\ln \frac{\varepsilon}{6})^\alpha}\le \rho \frac{w(t,x)}{(1-\ln w(t,x))^\alpha}\quad \text{for all }x\le x_{\frac{\varepsilon}{3}}(t)+1.
\]
Thus, by \eqref{allbl2}, for all $t\ge t_2:=\max\{t^\#,t^*\}$ and $x_\varepsilon(t)\le x\le x_{\frac{\varepsilon}{3}}(t)+1$ , we arrive at
\[
    \mathcal{D}[v]\ge -3\rho\frac{w(t,x)}{(1-\ln w(t,x))^\alpha}.
\]
As a result, by collecting \eqref{alvt1} and \eqref{allbfv}, for all $t\ge t_2$ and $x_\varepsilon(t)\le x\le x_{\frac{\varepsilon}{3}}(t)+1$, we achieve 
\begin{equation}\label{allbi2}
   \Big( \frac{\partial v}{\partial t}-\mathcal{D}[v]-f(v)\Big)(t,x)\le \frac{w(t,x)}{(1-\ln w(t,x))^\alpha}\Big( 3\rho+3\rho-r(1-K\varepsilon)\Big)\le 0,
\end{equation}
as soon as $\rho\le \frac{r}{12}$ and $\varepsilon\le \min\{\frac{1}{2K},\varepsilon_0\}$.
\medskip
\par
\noindent
\# \textbf{Zone 3: $x\in\left(x_{\frac{\varepsilon}{3}}(t)+1,+\infty\right)$.}
In view of \eqref{alvx} and \eqref{allbl3}, since $\alpha\beta+\beta-1< 0$, then
\[   (1+\varphi_0(x))^\alpha\varphi'_0(x)=x^{\alpha\beta+\beta-1}\left(1+\frac{1}{x^\beta}\right)^\alpha\to 0\quad \text{as }x\to +\infty,
\]
and it follows that there is a time $t_3>0$ such that for all $t\ge t_3$ and $x\ge x_{\frac{\varepsilon}{3}}(t)+1$, we have
\[
   \mathcal{C}_4(1+\varphi_0(x))^\alpha\varphi'_0(x) \le \rho.
\]
As a result, collecting \eqref{alvt1} and \eqref{allbfv}, for all $t\ge t_3$ and $x> x_{\frac{\varepsilon}{3}}(t)+1$, we achieve 
\begin{equation}\label{allbi3}
   \left(\frac{\partial v}{\partial t}-\mathcal{D}[v]-f(v)\right)(t,x)\le\frac{w(t,x)}{(1-\ln w(t,x))^\alpha}\Big(3\rho+3\rho- r(1-K\varepsilon)\Big)\le 0,
\end{equation}
as soon as $\rho\le\frac{r}{12}$ and $\varepsilon\le \underline\varepsilon:=\min\{\frac{1}{2K},\varepsilon_0\}$.
\par
Now, let us denote $\tilde t:=\max\{t_1,t_2,t_3\}$. Therefore, by collecting \eqref{allbi1}, \eqref{allbi2} and \eqref{allbi3}, for all $0<\rho\le\frac{r}{12}$ and $0<\varepsilon\le \underline\varepsilon$, the inequality \eqref{alprop} holds for all $t\ge \tilde t$ and $x\in\mathbb{R}$.
\end{proof}

\subsubsection{A flattening estimate}
 By Proposition \ref{aplbps}, the function $v$ is a sub-solution to \eqref{oeq1} for all $t\ge\tilde t$ and $x\in \mathbb{R}$. To validate the comparison principle, one would like to prove that there exist $t_1\ge 0$ and $R>0$ such that
 \[
  u(t_1,\cdot-R)\ge v(\tilde t,\cdot)\quad \text{on }\mathbb{R}.
 \]
 \par
To do so, we derive a flattening estimate for the solution to \eqref{oeq1}. Let $\underline u$ be the solution to the equation
  \begin{equation}\label{fu=0}
      \begin{cases}
         \frac{\partial \underline u}{\partial t}(t,x)-\mathcal{D}[\underline u](t,x)=0,& t>0,\ x\in\mathbb{R},\\
         \underline u(0,x)=u_0(x), &x\in\mathbb{R},
      \end{cases}
  \end{equation}
  where $u_0$ satisfies Hypothesis \ref{ic1}. Since $f\ge 0$, the comparison principle implies that $\underline u\le u$. As in \cite{bouin2023simple}, we have the following 
\begin{proposition}  \label{aplbf}
  Let $\underline u$ be the solution to \eqref{fu=0}. Then, for any positive constant $\kappa$, we have
    \[
       \lim_{x\to +\infty} e^{x^\beta}\underline u(t,x)\ge \kappa t\quad \text{ for all }t\in \mathbb{R}^+.
    \]
\end{proposition}
\begin{proof}[Proof of Proposition \ref{aplbf}]
    For all $t\in\mathbb{R}^+$, let us define 
    \[
         w(t,x):=
         \begin{cases}
             \frac{1}{2},&x\in\mathbb{R}\setminus \mathbb{R}^+,\\
             \frac{\kappa t}{e^{x^\beta}+2\kappa t}, &x\in\mathbb{R}^+,
         \end{cases}
    \]
    where $\kappa>0$. Some direct calculations show that for all $(t,x)\in \mathbb{R}^+\times \mathbb{R}^+$, 
    \[
        \frac{\partial w}{\partial t}(t,x)=\frac{\kappa e^{x^\beta}}{(e^{x^\beta}+2\kappa t)^2}>0,
    \]
    \[
    \frac{\partial w}{\partial x}(t,x)=-\frac{\beta\kappa t x^{\beta-1} e^{x^\beta}}{(e^{x^\beta}+2\kappa t)^2}<0,
    \]
    and
    \[
    \begin{aligned}   
     \frac{\partial^2 w}{\partial x^2}(t,x)
     = \frac{\beta \kappa t x^{2\beta-2}e^{x^\beta}}{(e^{x^\beta}+2\kappa t)^2}\left[\beta -\frac{4\beta\kappa t}{e^{x^\beta}+2\kappa t}+  (1-\beta) x^{-\beta}\right].
    \end{aligned}
    \]
    One may notice that, for all $(t,x)\in \mathbb{R}^+\times \mathbb{R}^+$, the function $w$ is increasing in $t$ and decreasing in $x$, and that, for $(t,x)\in \mathbb{R}^+\times \left[\ln^\frac{1}{\beta}(2\kappa t+1),+\infty\right)$, the function $w(t,\cdot)$ is convex.
    \par
    Let us estimate $\mathcal{D}[w]$. Let $B>1$ to be chosen later. In view of Hypothesis \ref{ker}, for all $x\ge R_0+B$, we write 
    \[
    \begin{aligned}
        \mathcal{D}[w](t,x)&=\int_{B}^{+\infty}J(z)\Big(w(t,x+z)-w(t,x)\Big)dz+\frac{1}{2}\int_{-B}^{B}J(z)\Big(w(t,x+z)+w(t,x-z)-2w(t,x)\Big)dz\\
        &\qquad\qquad +\int_{-\infty}^{-B}J(z)\Big(w(t,x+z)-w(t,x)\Big)dz\\
        &\ge -w(t,x)\int_{B}^{+\infty}J(z)dz+\frac{1}{2}\int_{-B}^{B}J(z)\Big(w(t,x+z)+w(t,x-z)-2w(t,x)\Big)dz\\
        &\qquad\qquad +\int_{-x}^{-B}J(z)\Big(w(t,x+z)-w(t,x)\Big)dz +\int_{-\infty}^{-x}J(z)\Big(w(t,x+z)-w(t,x)\Big)dz\\
        &\ge -\mathcal{J}_0w(t,x)\int_{B}^{+\infty}e^{-z^\beta}dz+ \frac{1}{2}\int_{-B}^{B}J(z)\Big(w(t,x+z)+w(t,x-z)-2w(t,x)\Big)dz\\
        &\qquad\qquad +\mathcal{J}_0^{-1}\left(\frac{1}{2}-w(t,x)\right)\int_x^{+\infty}e^{-z^\beta}dz.
    \end{aligned}
    \]
    For any $ \mathcal{C}>0$, we define $ t_\kappa:=\frac{2\mathcal{C}}{\kappa}$. Since, for all $t\in(0,t_\kappa)$, the function $w(t,\cdot)$ is convex on $\left[\ln^\frac{1}{\beta} (4\mathcal{C}+1),+\infty\right)$, then, letting $B\ge \ln^\frac{1}{\beta} (4\mathcal{C}+1)$,
    for all $(t,x)\in (0,t_\kappa)\times [R_0+B,+\infty)$, we get 
    \[
        \int_{-B}^{B}J(z)\Big(w(t,x+z)+w(t,x-z)-2w(t,x)\Big)dz= \int_{-B}^{B}\int_0^1\int_{-1}^1 z^2J(z)\tau \frac{\partial^2 w}{\partial x^2}(t,x+\sigma\tau z)d\sigma d\tau dz\ge 0.
    \]
    By using L'H\^opital's rule, one may find that
    \[
        \int_y^{+\infty}e^{-z^\beta}dz\asymp y^{1-\beta}e^{-y^\beta} \quad \text{for all }y> 0,
    \]
    and it follows that, for all $(t,x)\in (0,t_\kappa)\times [R_0+B,+\infty)$, 
    \[
    \begin{aligned}
        \mathcal{D}[w](t,x)&\ge -\mathcal{C}_1w(t,x)B^{1-\beta}e^{-B^\beta}+\mathcal{C}_2\Big(\frac{1}{2}-w(t,x)\Big)x^{1-\beta}e^{-x^\beta}\\
        &=-\mathcal{C}_1B^{1-\beta}e^{-B^\beta}\frac{\kappa t}{e^{x^\beta}+2\kappa t}+\frac{1}{2}\mathcal{C}_2\frac{x^{1-\beta}}{e^{x^\beta}+2\kappa t},
    \end{aligned}
    \]
    for some positive constants $\mathcal{C}_1$ and $\mathcal{C}_2$. Now, let us check, for all $(t,x)\in (0,t_\kappa)\times [R_0+B,+\infty)$, 
    \[
    \begin{aligned}
        \frac{\partial w}{\partial t}-\mathcal{D}[w]&=\frac{1}{e^{x^\beta}+2\kappa t}\left(\frac{\kappa e^{x^\beta}}{e^{x^\beta}+2\kappa t}+\mathcal{C}_1B^{1-\beta}e^{-B^\beta}\kappa t-\frac{1}{2}\mathcal{C}_2x^{1-\beta}\right)\\
        &\le \frac{1}{e^{x^\beta}+2\kappa t}\left(\kappa +\mathcal{C}_1B^{1-\beta}e^{-B^\beta}\kappa t-\frac{1}{2}\mathcal{C}_2B^{1-\beta}\right).
    \end{aligned}
    \]
    Notice that there is some $B_0>1$ such that for all $B\ge B_0$, we have
    \[
        B^{1-\beta}e^{-B^\beta}\le B^{-1}.
    \]
    We choose $B\ge B_c:=\left(\frac{2(\kappa+1)}{\mathcal{C}_2}\right)^{\frac{1}{1-\beta}}+2\mathcal{C}\mathcal{C}_1+B_0+\ln^\frac{1}{\beta} (4\mathcal{C}+1)$, so that, for all $t\in(0,t_\kappa)$,
    \[
    \begin{aligned}
        \kappa +\mathcal{C}_1B^{1-\beta}e^{-B^\beta}\kappa t-\frac{1}{2}\mathcal{C}_2B^{1-\beta}\le \kappa +2\mathcal{C}\mathcal{C}_1 B^{-1} -\frac{1}{2}\mathcal{C}_2B^{1-\beta}\le \kappa+1-\frac{1}{2}\mathcal{C}_2B^{1-\beta}\le 0.
    \end{aligned}
    \]
    Thus, we arrive at 
    \begin{equation} \label{aplbfll}
        \frac{\partial w}{\partial t}-\mathcal{D}[w]\le0\quad \text{for all }t\in(0,t_\kappa) \text{ and } x\in[R_0+B_c,+\infty).
    \end{equation}
\par 
In view of Hypothesis \ref{ic1}, denoting $M:=\liminf_{x\to -\infty}u_0(x)$, there is $b\in\mathbb{R}$ such that $u_0(x)\ge \frac{M}{2}$ for all $x\le b$. One may notice that $\underline u= \frac{M}{2}$ is a sub-solution to \eqref{fu=0} on $[0,\infty)\times(-\infty,b]$. Thus, the comparison principle implies that
    \[
   \label{allbin}
        \underline u(t,x)\ge \frac{M}{2}\quad \text{for all } t>0\text{ and } x\le b.
     \]
 Since $\underline u(t,\cdot)$ is decreasing for each $t>0$, by the definition of $w$, we get $\underline u(t,x-R_0-B_c+b)\ge \frac{M}{2}\ge M w(t,x)$ for $t>0$ and $x\le R_0+B_c$. Denoting $\tilde u(t,x)=\underline u(t,x-R_0-B_c+b)$, it follows that we have
\[
    \begin{cases}
         \frac{\partial \tilde u}{\partial t}(t,x)-\mathcal{D}[\tilde u](t,x)=0, &\text{for all }t\in(0, t_\kappa)\text{ and }x\in\mathbb{R},\\
        \tilde u(t,x)\ge M w(t,x), &\text{for all }t\in[0,t_\kappa]\text{ and }x\le R_0+B_c,\\
        \tilde u(0,x)> M w(0,x), &\text{for all }x\in\mathbb{R}.
    \end{cases}
\]
    Thus, combining \eqref{aplbfll} and using an adapted comparison principle from \cite[Theorem $2.2$]{bouin2023simple}, we get $\tilde u(t,x)\ge M w(t,x)$ for $t\in(0, t_\kappa)$ and $x\in\mathbb{R}$, and it follows from the definition of $t_\kappa$ that, for all $x\ge -R_0-B_c+b$,
    \[
        \underline u\left(\frac{t_\kappa}{2},x\right)\ge M w\left(\frac{t_\kappa}{2},x+R_0+B_c-b\right)=\frac{M\mathcal{C}}{e^{(x+R_0+B_c-b)^\beta}+2\mathcal{C}}.
    \]
    Thus, we have
    \[
        \lim_{x\to +\infty} e^{x^\beta} \underline u\left(\frac{t_\kappa}{2},x\right)\ge \lim_{x\to +\infty}\frac{M\mathcal{C}e^{x^\beta}}{e^{(x+R_0+B_c-b)^\beta}+2\mathcal{C}}=M\mathcal{C},
    \]
    or, equivalently, that, for all positive real number $\mathcal{C}$,
    \[
        \lim_{x\to +\infty} e^{x^\beta} \underline u\left(\frac{\mathcal{C}}{\kappa},x\right)\ge M\mathcal{C}.
    \]
    Therefore, we conclude that
    \[
        \lim_{x\to +\infty}e^{x^\beta}\underline u(t,x)\ge M\kappa t\quad \text{for all }t\in(0,+\infty).
    \]
    This completes the proof.
\end{proof}

\subsubsection{Proof of the lower bound in Theorem \ref{ei1}}\label{sslb}
Now, with Proposition \ref{aplbf}, it follows from the comparison principle that $\lim_{x\to+\infty}e^{x^\beta}u(t,x)\ge  t$ for all $t>0$. Then, there is $\xi_1>0$ such that, for all $x\ge \xi_1$, we have 
\[
    u(t,x)\ge  \frac{1}{2}t e^{-x^\beta}.
\]
In view of the definition of $v$, by \eqref{allbwu}, we have
\[
    v(t,x)\le 3 w(t,x)\le 3 e^{[\rho(\alpha+1)t]^\frac{1}{\alpha+1}}e^{-x^\beta}.
\]
Thus, for all $t\ge t_1':= 6e^{[\rho(\alpha+1)\tilde t]^\frac{1}{\alpha+1}}$, we have
\begin{equation}\label{aslbu1}
   u(t,x)\ge  \frac{1}{2}t e^{-x^\beta}\ge 3 e^{[\rho(\alpha+1)\tilde t]^\frac{1}{\alpha+1}} e^{-x^\beta}\ge v(\tilde t,x)\quad \text{for all } x\ge \xi_1.
\end{equation}
Now, since propagation in integro-differential equations with Allee effect occurs \cite{alfaro2017propagation}, it follows from the comparison principle that there are $t_1''\in \mathbb{R}^+$ and $\xi_2\in \mathbb{R}$ such that 
\[   
u(t,\xi_2)\ge \varepsilon \quad \text{for all }t\ge t_1'',
\]
and thus, since $u( t,\cdot)$ is non-increasing, one gets
\[
    u( t,x)\ge \varepsilon \quad \text{for all } t\ge t_1'' \text{ and } x\le \xi_2.
\]
It follows from the definition of $v$ that 
\[
    u(t,x)\ge\varepsilon\ge  v(\tilde t,x)\quad \text{for all }t\ge t_1''\text{ and }x\le \xi_2 .
\]
Thus, by \eqref{aslbu1} and denoting $t_1:= \max\{t_1',t_2''\}$, one obtains
\[
    u(t_1,x-|\xi_1-\xi_2|)\ge v(\tilde t,x)\quad \text{for all }x\in \mathbb{R}.
\]
As a result, the comparison principle implies 
\[
    u(t,x-|\xi_1-\xi_2|)\ge v(t+\tilde t-t_1,x)\quad \text{for all }t\ge t_1\text{ and }x\in\mathbb{R}.
\]
\par With the above inequality, let us give a lower bound when $\lambda$ is small, that is $\lambda\in(0,\varepsilon]$. Since $v(t,x_{\varepsilon}(t))=\varepsilon$, we have
 \begin{equation}\label{u4m}
u(t,x-|\xi_1-\xi_2|)\ge \varepsilon\quad \text{for all }t\ge t_1 \text{ and } x\le x_\varepsilon(t+\tilde t-t_1).
 \end{equation}
Therefore, by \eqref{aplbdx}, for any $0<\lambda\le\varepsilon$, there is a time $\underline T_1>t_1$ such that we have 
\[
X_\lambda(t)\ge x_\varepsilon(t+\tilde t-t_1)-|\xi_1-\xi_2|\ge \underline C_1t^{\frac{1}{\beta(\alpha+1)}}\quad \text{ for all }t \ge \underline T_1,
\]
where $\underline C_1:=\left[\frac{\rho(\alpha+1)}{2}\right]^\frac{1}{\beta(\alpha+1)}$. This gives a lower bound to the level set $E_\lambda(t)$ for $\lambda\in(0,\varepsilon]$.
\par 
 Now, we upgrade the previous bound to $\lambda\in(\varepsilon,1)$. Let $w$ be the solution to \eqref{oeq1} with the following initial data 
 \begin{equation}\label{w0}
w_0(x):=
\begin{cases}
\varepsilon, & x\le -1,\\
-\varepsilon x,&-1<x<0,\\
0, &x\ge 0  .
\end{cases}
 \end{equation}
 Since propagation occurs \cite{alfaro2017propagation}, the comparison principle implies that there exists a time $t_\lambda>0$ such that
\begin{equation}\label{pp}
	w(t_\lambda,x)>\lambda \quad\text{for any }x\le 0.
\end{equation}
On the other hand, it follows from \eqref{u4m} and \eqref{w0} that
\[
u(t,x-|\xi_1-\xi_2|)\ge w_0(x-x_{\varepsilon}(t+\tilde t-t_1)) \quad \text{for any } t\ge t_1 \text{ and }x\in \mathbb{R} .
\]
Thus, by the comparison principle, we have
\[
u(T+t,x-|\xi_1-\xi_2|)\ge w(T,x-x_{\varepsilon}(t+\tilde t-t_1)) \quad \text{ for any }\ T\ge 0, t\ge t_1 \text{ and } x\in\mathbb{R},.
\]
Then, by \eqref{pp} and letting $T=t_\lambda$, we obtain
\[
	u(t+t_\lambda,x-|\xi_1-\xi_2|)>\lambda \quad \text{for any } t\ge t_1 \text{ and } x\le x_{\varepsilon}(t+\tilde t-t_1).
\]
As a result, there exists a time $\underline T_{2}>\max\{t_\lambda, t_1\}$ and a constant $\underline C_2>0$ such that, for $t\ge \underline T_2$, we have
\[
X_\lambda(t)\ge x_{\varepsilon}(t+\tilde t-t_1-t_\lambda)-|\xi_1-\xi_2|\ge \underline C_2 t^{\frac{1}{\beta(\alpha+1)}},
\]
which gives the lower bound of the level set for any $\lambda\in(\varepsilon,1)$.  
Therefore, taking $\underline C:=\min\{\underline C_1,\underline C_2\}$ and $\underline T=\max\{\underline T_1,\underline T_2\}$, we get the lower bound in Theorem \ref{ei1} for all $t\ge \underline T$.

\section{Exponential propagation: proof of Theorem \ref{al}}\label{s4}

In this section, we prove the sharp exponential rate stated in Theorem \ref{al}. In the polynomial acceleration regime, the growth parameter $\rho$, used as a free parameter, does not change the  acceleration rate. In the regime of exponential  acceleration, it is not the case anymore: small perturbations in the growth rate have a significant impact on the estimates of the acceleration rate. In such a situation, using the sub- and super-solutions considered in the previous section would only provide rough estimates on the exponential acceleration rate.

 To obtain much sharper estimates, we shall construct much more precise sub- and super-solutions where the growth parameter $\rho=r$ is now fixed by the form of $f$. One way to get such precise sub- and super- solutions is to add a degree of freedom to the construction adding a logarithmic correction to the tails. Such an approach requires performing most of the previous calculations again with higher precision. 

\subsection{The upper bound}
Here, we prove the upper bound in Theorem \ref{al}. To do so, we construct an adequate super-solution. 
\subsubsection{Definition of the super-solution and preliminary computations}
\par Let us define
\[
	v_0(x):=\left\{
	\begin{aligned}
		&1,&x<L,\\
		&C_L\frac{\ln^{q} x}{x^{2s}},& x\ge L,
	\end{aligned}\right.
\]
where $q>0$, $C_L:=\frac{L^{2s}}{\ln^{q} L}$ and $L\ge \max\{e,e^{\frac{q}{2s}},a\}+1$ large enough, so that $v_0$ is decreasing, continuous and less than or equal to $1$.  Here, $a$ comes from Hypothesis \ref{ic1}, and thus, one may notice that 
\begin{equation}
    \label{epubv0u0}
    v_0\ge u_0.
\end{equation}
Let us  define 
\[
	\psi(t,x)=\exp{\left\{1-\left\{\left[1-\ln[(1+ \kappa t^{p})v_0(x)]\right]^{\alpha+1}-r(\alpha+1)t\right\}^{\frac{1}{\alpha+1}}\right\}},
\]		
where $p\ge0$ and $\kappa\ge e$ are to be determined. The domain of $\psi$ will be given later. For later use, we state several properties related to $v_0$ and $\psi$, the proofs of which can be found in Appendix \ref{appendix-exp-acc-upper}. We first obtain:
\begin{lemma}\label{epubl1}
   $ I.$
For $y\in(0,1)$, we have, denoting by $v_0^{-1}$ the inverse of $v_0$ on $[L,+\infty)$,
\begin{equation}\label{xysim}
  v_0^{-1}(y)\asymp \frac{\left(\ln L -\frac{1}{2s}\ln y\right)^\frac{q}{2s}}{y^\frac{1}{2s}},
\end{equation}
and 
\begin{equation}\label{esubv'vi}
    (v'\circ v^{-1})(y)\asymp- y^{1+\frac{1}{2s}}(\ln y^{-1})^{-\frac{q}{2s}}.
\end{equation}
    $II.$
For $t\ge t^\star:=\frac{1}{r(\alpha+1)}$, define 
\[
x_\Lambda(t):=v_0^{-1}\left\{(1+\kappa t^{p})^{-1}\exp\left\{1-\left[r(\alpha+1) t+(1-\ln \Lambda)^{\alpha+1}\right]^{\frac{1}{\alpha+1}}\right\}\right\} \quad \text{for any }\Lambda\in(0,2].
\]
Then, for all $t\ge t^\star$, we have 
\begin{equation}\label{esubxte}
    x_\Lambda(t)\asymp_\Lambda t^{\frac{1}{2s}\left(p+\frac{q}{\alpha+1}\right)}\exp\left\{\frac{1}{2s}[r(\alpha+1)t]^\frac{1}{\alpha+1}\right\},
\end{equation}
and  for any $0<\Lambda_1<\Lambda_2\le 2$,
    \begin{equation}\label{esubfl}
    x_{\Lambda_1}(t)-x_{\Lambda_2}(t)\asymp_{\Lambda_1,\Lambda_2} t^{-\frac{\alpha}{\alpha+1}+\frac{1}{2s}\left(\frac{q}{\alpha+1}+p\right)}e^{\frac{1}{2s}[r(\alpha+1)t]^\frac{1}{\alpha+1}}.
     \end{equation}
\end{lemma}
Let us now summarize in the next two lemmas some basic properties for $\psi$ that we will constantly use.   
\begin{lemma}\label{esuble1} 
Let $\Lambda\in(0,2]$.
For $(t,x)\in [t^\star,\infty)\times [x_\Lambda(t),+\infty)$, the function $\psi$ is well defined,  non-decreasing in $t$ and non-increasing in $x$, and for any $t\in[t^\star,\infty)$, $\psi(t,\cdot)$ is convex and log-convex on $[ x_\Lambda(t),+\infty)$.
In addition, there exist $t_0>t^\star$, $\gamma_1>0$ and $\gamma_2>0$, such that for all $t\ge t_0$ and $x\ge x_\Lambda(t)$, we have
\begin{equation}\label{naemt}
\begin{aligned}
        \frac{\partial \psi}{\partial t}\ge\left(\gamma_1 t^{-\frac{1}{\alpha+1}}+r\right)\frac{\psi}{(1-\ln \psi)^\alpha},
        \end{aligned}
\end{equation}
and
\begin{equation}\label{naemx}
  -\gamma_2^{-1}\frac{\ln^\alpha x}{x}\frac{\psi}{(1-\ln \psi)^\alpha}\le   \frac{\partial\psi}{\partial x}\le  -\gamma_2 \frac{t^\frac{\alpha}{\alpha+1}}{x}\frac{\psi}{(1-\ln\psi)^\alpha}.
\end{equation}
 Moreover, there is some positive constant $\mathcal{C}_0$, such that for all $t\ge t^\star$ and $x\ge  x_\Lambda(t)$, we have 
\begin{equation}\label{epubwxxe}
    0\le \frac{\partial^2 \psi}{\partial x^2}(t,x)\le \frac{\mathcal{C}_0 }{x}\frac{\psi(t,x)}{(1-\ln\psi(t,x))^\alpha}.
\end{equation}

\end{lemma}

\begin{lemma}\label{lee2}
 There is a time $t^\dagger\ge t^\star$ such that, for all $z\in\left(\frac{1}{x},\min\left\{\frac{1}{2},\frac{x- x_1(t)}{x}\right\}\right)$ and $(t,x)\in [t^\dagger,\infty)\times [x_1(t)+1,+\infty)$, we have 
\[
\begin{aligned}
    \frac{1-\ln \psi (t,x(1-z))}{1-\ln \psi(t,x)}
    \ge 1-2(2s)^{\alpha+1}(\ln^\alpha x)z-C_\alpha(\ln^{2\alpha} x)z^2,
\end{aligned}
\]
and
\[
\begin{aligned}
   \frac{\psi(t,x(1-z))}{\psi(t,x)}
   \le \exp\left\{ 2(2s)^{\alpha+1}(\ln^\alpha x)z+C_\alpha(\ln^{2\alpha} x)z \right\},
\end{aligned}
\]
for some positive constant $C_\alpha:=4\alpha(\alpha+1)(2s)^{2\alpha+2}$.
\end{lemma}

\par
Now, we are ready to properly define the super-solution to \eqref{oeq1} we need. Let us define, for $t\in [t^\star,\infty)$, 
\begin{equation}\label{um2}
	\Psi(t,x)=\left\{
	\begin{aligned}
		&1,&x<  x_1(t),\\
		&\psi(t,x),&x\ge  x_1(t).
	\end{aligned}\right.
\end{equation}
Unless otherwise specified, throughout the subsection later, we assume that $t\in[t^\star,\infty)$.
\par
We start by proving an estimate for $\mathcal{D}[\Psi]$.
\subsubsection{Estimation of \texorpdfstring{$\mathcal{D}[\Psi](t,\cdot)$ on $[x_1(t),+\infty)$}{D[Psi]}}
\begin{lemma}\label{epubed}
Denote $R(t):=\gamma_3 t^{-1}x_1(t)$ for some $\gamma_3>0$ and $\bar q:=(2s+1)(\alpha+1)-q$. We have the following estimate for $\mathcal{D}[\Psi]$.
\begin{itemize}
\item[I.]
     When $ x_1(t)\le x<  x_1(t)+R(t)$, there is a time $\tilde t_1\ge t^\star$ such that, for all $t\ge \tilde t_1$, we have
     \[
         \mathcal{D}[\Psi](t,x)\le  \mathcal{C}_1 e^{-\frac{1}{2s}[r(\alpha+1)t]^\frac{1}{\alpha+1}}\frac{\psi(t,x)}{(1-\ln \psi(t,x))^\alpha}-\mathcal{J}_3R(t)\frac{\partial \psi}{\partial x}(t, x_1(t)),
     \]
     for some positive constant $\mathcal{C}_1$, where $\mathcal{J}_3:=\int^{\infty}_{1}J(z)dz$.
     \item[II.] When $x\ge x_1(t)+R(t)$, there is a time $\tilde t_2>t^\star$ such that for all $t\ge \tilde t_2$ and $x\ge  x_1(t)+R(t)$, we have
\[
		\mathcal{D}[\Psi](t,x)\le\left(\mathcal{C}_3\frac{\ln^{\alpha+1}x}{x^{\min\{2s,1\}}}+ \mathcal{C}_4\frac{\ln^{\bar q}x}{x^{s}}+\mathcal{C}_5\frac{\kappa^{-1} t^{-p}}{(1-\ln \psi(t,x))^{2\alpha-q+1}}+\mathcal{C}_6 e^{-\frac{1}{2s}[r(\alpha+1)t]^\frac{1}{\alpha+1}}\right) \frac{\psi(t,x)}{(1-\ln \psi(t,x))^{\alpha}},
\]
for some positive constants $\mathcal{C}_3$, $\mathcal{C}_4$, $\mathcal{C}_5$ and $\mathcal{C}_6$. 
\end{itemize}
\end{lemma}
\begin{proof}
We divide the proof into two parts: $x_1(t)\le x< x_1(t)+R(t)$ and $x\ge x_1(t)+R(t)$.
\medskip
\par
\noindent
\# \textbf{Part I: $x_1(t)\le x< x_1(t)+R(t)$.}
   By the definition of $\mathcal{D}[\Psi]$, since $\Psi(t,\cdot)$ is non-increasing on $\mathbb{R}$, we write 
\[
    \begin{aligned}
        \mathcal{D}[\Psi]&\le P.V.\int_{|z|<1}J(z)\Big(\Psi(t,x+z)-\Psi(t,x)\Big)dz
        +\int_{z\le -1}J(z)\Big(\Psi(t,x+z)-\Psi(t,x)\Big)dz\\
        &:=I_1+I_2.
        \end{aligned}
 \]
 \par
\#\# \textbf{Estimating $I_1$.} 
In view of \eqref{esubfl}, there is a time $\tilde t_1\ge t^\star$ such that $x_1(t)-x_2(t)\ge 1$ for all $t\ge \tilde t_1$.
Since $\Psi(t,x)=1\le \psi(t,x)$ for $x_2(t)\le x\le  x_1(t)$ and $\Psi(t,x)=\psi(t,x)$ for $x\ge  x_1(t)$, it follows from Taylor's theorem that, for all $(t,x)\in [\tilde t_1,\infty)\times [x_1(t),+\infty)$,
    \[
    \begin{aligned}
        I_1&\le  \frac{1}{2}\int_{|z|<1}J(z)\Big(\psi(t,x+z)+\psi(t,x-z)-2\psi(t,x)\Big)dz\le \mathcal{J}_1\sup_{|B|\le 1}\frac{\partial^2 \psi}{\partial x^2}(t,x+B).
        \end{aligned}
    \]
    As a result, by \eqref{esubxte} and \eqref{epubwxxe}, for all $(t,x)\in [\tilde t_1,\infty)\times [x_1(t),+\infty)$, we arrive at
    \begin{equation}\label{esubl31}
        I_1\le \mathcal{J}_1\frac{\mathcal{C}_0}{x-1}\frac{\psi(t,x-1)}{(1-\ln \psi(t,x-1))^\alpha} \le \mathcal{C}_1 e^{-\frac{1}{2s}[r(\alpha+1)t]^\frac{1}{\alpha+1}}\frac{\psi(t,x)}{(1-\ln \psi(t,x))^\alpha},
    \end{equation}
    for some positive constant $\mathcal{C}_1$.
    \par\#\# \textbf{Let us estimate $I_2$.}
Since $\psi(t,\cdot)$ is non-increasing on $[x_1(t),+\infty)$ and $J\in L^1((-\infty,-1])$, then for $x\in [ x_1(t), x_1(t)+R(t))$, we have
    \[
    I_2\le \mathcal{J}_3[1-\psi(t, x_1(t)+R(t))],
    \]
        where $\mathcal{J}_3=\int^{\infty}_{1}J(z)dz$.
    By convexity of $\psi(t,\cdot)$, one may get 
    \[
    \psi(t, x_1(t)+R(t)\ge \psi(t, x_1(t))+R(t)\frac{\partial \psi}{\partial x}(t, x_1(t)),
    \]
    and thus, since $\psi(t, x_1(t))=1$, one arrives at
    \begin{equation}  \label{esubl33}
    I_2\le -\mathcal{J}_3R(t)\frac{\partial \psi}{\partial x}(t, x_1(t)).
    \end{equation}
    Therefore, collecting \eqref{esubl31}  and \eqref{esubl33}, for all $(t,x)\in [\tilde t_1,\infty)\times [x_1(t),x_1(t)+R(t))$, we achieve
    \[
    \mathcal{D}[\Psi](t,x)\le \mathcal{C}_1 e^{-\frac{1}{2s}[r(\alpha+1)t]^\frac{1}{\alpha+1}}\frac{\psi(t,x)}{(1-\ln \psi(t,x))^\alpha}-\mathcal{J}_3R(t)\frac{\partial \psi}{\partial x}(t, x_1(t)).
    \]
\medskip
\par
\noindent
    \# \textbf{Part II: $x\ge x_1(t)+R(t)$.} By the definition of $\mathcal{D}[\Psi]$, since $0<\Psi\le 1$ and $\Psi(t,\cdot)$ is non-increasing, we write
\[
	\begin{aligned}
		\mathcal{D}[\Psi]&\le \int_{-\infty}^{  x_1(t)-x}J(z)dz +\int_{ x_1(t)-x}^{-1} J(z)\Big(\psi(t,x+z)-\psi(t,x)\Big)dz
   \\
   &\qquad \qquad +P.V. \int_{-1}^{1} J(z)\Big(\psi(t,x+z)-\psi(t,x)\Big)dz:=II_1+II_2+II_3.
	\end{aligned}
\]
 \par
\#\# \textbf{Let us estimate $II_1$.} By \eqref{Je}, we have
\[
	II_1\le \int_{-\infty}^{ x_1(t)-x}J(z)dz\le \frac{\mathcal{J}_0}{2s}\frac{1}{[x- x_1(t)]^{2s}}.
\]
For any fixed $t\in [t^\star,+\infty)$, we denote $V_t(x):=\{1-\ln[ (1+\kappa t^p)v_0(x)]\}^{\alpha+1}-1$ for $x>L$. 
By some direct calculations, for all $x>L$, one may get 
\[
    V_t'(x)=-(\alpha+1)\{1-\ln[ (1+\kappa t^p) v_0(x)]\}^\alpha \frac{v_0'}{v_0}(x),
\]
and 
\[
    V_t''(x)= (\alpha+1)\{1-\ln[(1+\kappa  t^p)v_0(x)]\}^{\alpha-1}\left[\alpha \left(\frac{v_0'}{v_0}\right)^2(x)-\{1-\ln[(1+\kappa  t^p)v_0(x)]\} \left(\frac{v_0'}{v_0}\right)'(x)\right].
\]
Since $\frac{v_0'}{v_0}(x)=\left(-2s+\frac{q}{\ln x}\right)\frac{1}{x}< 0$ and $\left(\frac{v_0'}{v_0}\right)'(x)=\left(2s-\frac{q}{\ln x}-\frac{q}{\ln^2 x}\right)\frac{1}{x^2}> 0$ for $x$ large enough, there is $\tilde L\ge L$ such that, for all $x\ge \tilde L$, we have
\[
\begin{aligned}
   & \alpha \left(\frac{v_0'}{v_0}\right)^2(x)-\{1-\ln[(1+\kappa  t^p)v_0(x)]\} \left(\frac{v_0'}{v_0}\right)'(x)\\
    \le &\left[\alpha \left(-2s+\frac{q}{\ln x}\right)^2 -[r(\alpha+1)t]^\frac{1}{\alpha+1}\left(2s- \frac{q}{\ln x}-\frac{q}{\ln^2 x}\right)\right]\frac{1}{x^2}\le 0.
\end{aligned}
\]
 It follows that 
\[
    V_t'(x)> 0\quad\text{ and }\quad V_t''(x)\le 0\quad \text{for all }x\ge \tilde L.
\]
Since $R(t)$ and $x(t)$ go to infinity as $t$ to infinity, there is $t_1\ge t^\star$ such for all $t\ge t_1$, $R(t)$ and $x_1(t)$ are greater than $\tilde L$, and thus,  for all $t\ge t_1$ and $x\ge  x_1(t)+R(t)$, we get
\[
\begin{aligned}
    V_t(x- x_1(t))\ge V_t(x)-V_t( x_1(t))&=\{1-\ln[(1+\kappa t^p)v_0(x)]\}^{\alpha+1}-r(\alpha+1)t-1\\
    &=(1-\ln \psi)^{\alpha+1}-1\\
    &= \left\{1-\ln \left[(1+\kappa t^p)v_0 \left(v_0^{-1}\left(\left(1+\kappa t^{p}\right)^{-1}\psi\right)\right)\right]\right\}^{\alpha+1}-1,
\end{aligned}
\]
and then, since $V_t$ is strictly increasing on $[\tilde L,+\infty)$, by \eqref{xysim}, for all $t\ge t_1$ and $x\ge x_1(t)+R(t)$, we achieve
\begin{equation}\label{esubx-xt}
    x- x_1(t)\ge v_0^{-1}\left(\left(1+\kappa t^{p}\right)^{-1}\psi\right) \gtrsim\frac{\left\{\ln L-\frac{1}{2s}\ln[(1+\kappa t^p)^{-1}\psi]\right\}^{\frac{q}{2s}}}{[\kappa^{-1} t^{-p}\psi]^\frac{1}{2s}}.
\end{equation}
It follows that for all $t\ge t_1$ and $x\ge  x_1(t)+R(t)$,\footnote{One may notice from here that it is necessary to have the factor $\ln^q x$ in $v_0$. Indeed, if $v_0$ is defined without $\ln^q x$, then we get $I_1\le \mathcal{C}_2\kappa^{-1} t^{-p}\psi$ as an estimate of $I_1$, which is much larger than $f(\psi)$. This is not what we expect.}
\begin{equation}
    \label{esubi1}
   II_1\le \mathcal{C}_2\frac{\kappa^{-1}t^{-p}\psi}{\left\{2s\ln L-\ln[ (1+\kappa t^{p})^{-1}\psi]\right\}^q}\le \mathcal{C}_2 \frac{\kappa^{-1}t^{-p}\psi}{(1-\ln \psi)^q},
\end{equation}
for some $\mathcal{C}_2>0$.
\par
\#\#\textbf{ Let us estimate $II_2$.} 
We denote $B:= \frac{\gamma_4}{\ln^{\alpha+1} x}$ where $\gamma_4>0$ is small enough such that
\begin{equation}
    \label{gamma4}
    1-2(2s)^{\alpha+1}\gamma_4-C_\alpha \gamma_4^2\ge \frac{1}{2}.
\end{equation}
By taking $\gamma_4<\frac{1}{2}$, one may notice that $x\mapsto x-\frac{\gamma_4 x}{\ln^{\alpha+1}x}$ is increasing for all $x\ge x_1(t)+R(t)$, it follows that, for all $x\ge x_1(t)+R(t)$,
\[
x-\frac{\gamma_4 x}{\ln^{\alpha+1}x}\ge x_1(t)+R(t)- \gamma_4\frac{x_1(t)+R(t)}{\ln^{\alpha+1}(x_1(t)+R(t))}.
\]
Recalling $R(t)=\gamma_3 t^{-1} x_1(t)$ and using \eqref{esubxte}, one has $R(t)-\gamma_4\frac{x_1(t)+R(t)}{\ln^{\alpha+1}(x_1(t)+R(t))}\ge 0$ for all $t\ge t^\star$ by taking $\gamma_4$ small enough, and thus, one achieve 
\[
x- \frac{\gamma_4x}{\ln^{\alpha+1}x}\ge x_1(t).
\]
Therefore, there is a time $t_2>t^\star$ and constant $\gamma_4>0$ such that for all $t\ge t_2$ and $x\ge  x_1(t)+R(t)$, we have
\[
\frac{1}{x}\le \frac{\gamma_4}{\ln^{\alpha+1} x}\le\min\left\{\frac{1}{2}, \frac{x- x_1(t)}{x}\right\}.
\]
By using the change of variables $u=-\frac{z}{x}$, for all $t\ge t_2$ and $x\ge x_1(t)+R(t)$, we have
\[
\begin{aligned}
II_2&\le x\int^{\frac{x- x_1(t)}{x}}_{\frac{1}{x}}J(xu)\Big(\psi(t,x(1-u))-\psi(t,x)\Big) du\\
&\le x\int^{B}_{\frac{1}{x}}J(xu)\Big(\psi(t,x(1-u))-\psi(t,x)\Big) du+ x\int^{\frac{x- x_1(t)}{x}}_{B}J(xu)\psi(t,x(1-u)) du:=II_4+II_5.
\end{aligned}
\]

\par
\#\#\# \textbf{Let us estimate $II_4.$} By Taylor's theorem and the convexity of $\psi(t,\cdot)$ on $[ x_1(t),+\infty)$, $II_4$ can be written as 
\[
II_4=  -x^2\int^{B}_{\frac{1}{x}}\int_0^1 uJ(xu) \frac{\partial \psi}{\partial x}\Big(t,x(1-\tau u)\Big)d\tau du\le  -x^2\int^{B}_{\frac{1}{x}} uJ(xu) \frac{\partial \psi}{\partial x}\Big(t,x(1- u)\Big) du.
\]
Since $y\mapsto \frac{y}{(1-\ln y)^\alpha}$ is non-decreasing on $(0,1)$ and $\psi(t,\cdot)$ is non-increasing on $[x_1(t),+\infty)$, by \eqref{naemx},  for all $(t,x)\in [\max\{t_0,t_2\},\infty)\times [x_1(t),+\infty)$, we get 
\[
\begin{aligned}
II_4&\le -x^2 \int^{B}_{\frac{1}{x}} uJ(xu) \frac{\partial \psi}{\partial x}\Big(t,x(1- u)\Big)du\\
&\le \gamma_2^{-1}\mathcal{J}_0 x^{-2s} \int^{B}_{\frac{1}{x}}\frac{\ln^\alpha x(1-u)}{u^{2s}(1-u)}\frac{\psi(t,x(1-u))}{[1-\ln \psi(t,x(1-u))]^\alpha}du\\
&\le \gamma_2^{-1}\mathcal{J}_0 x^{-2s}(\ln^\alpha x)  (1-B)^{-1}\frac{\psi(t,x(1-B))}{[1-\ln \psi(t,x(1-B))]^\alpha}\int^{B}_{\frac{1}{x}}u^{-2s}du.
\end{aligned}
\]
On the one hand, by Lemma \ref{lee2}, recalling $B=\frac{\gamma_4}{\ln^{\alpha+1} x}$, we have
\[
\frac{\psi(t,x(1-B))}{[1-\ln \psi(t,x(1-B))]^\alpha}
\le \frac{\psi(t,x)}{[1-\ln \psi(t,x)]^\alpha} \frac{\exp\left\{ 2(2s)^{\alpha+1}\frac{\gamma_4}{\ln x}+C_\alpha \frac{\gamma_4^2}{\ln^2 x} \right\}}{\left[ 1-2(2s)^{\alpha+1}\frac{\gamma_4}{\ln x}-C_\alpha \frac{\gamma_4^2}{\ln^2 x}\right]^\alpha}.
\]
Thus, by \eqref{gamma4}, for all $t\ge \max\{t_0,t_2\}$ and $x\ge  x_1(t)+R(t)$, we obtain
\[
\frac{\psi(t,x(1-B))}{[1-\ln \psi(t,x(1-B))]^\alpha}
\lesssim \frac{\psi(t,x)}{[1-\ln \psi(t,x)]^\alpha}.
\]
On the other hand, we have
\[
\int_{\frac{1}{x}}^B u^{-2s}du = 
\begin{cases}
    \frac{1}{1-2s}\left(B^{1-2s}-\frac{1}{x^{1-2s}}\right) ,&0<s<\frac{1}{2},\\
    \frac{1}{2s-1}\left(x^{2s-1}-\frac{1}{B^{2s-1}}\right), &s>\frac{1}{2},\\
    \ln (Bx), &s=\frac{1}{2}.
\end{cases}
\]
It follows from $B=\frac{\gamma_4}{\ln^{\alpha +1}x }$ that 
\[
\begin{aligned}
(1-B)^{-1}x^{-2s}\ln^\alpha x\int_{\frac{1}{x}}^B u^{-2s}du &\lesssim 
\begin{cases}
   x^{-2s}\ln^{2s(\alpha+1)-1}x, &0<s<\frac{1}{2}, \\
   x^{-1}\ln^{\alpha} x, &s>\frac{1}{2},\\
    x^{-1}\ln^{\alpha+1} x, &s=\frac{1}{2},
\end{cases}
\\
&\le O(x^{-\min\{2s,1\}}\ln^{\alpha+1} x)\quad \text{as }x\to +\infty.
\end{aligned}
\]
As a result, for all $t\ge \max\{t_0,t_2\}$ and $x\ge  x_1(t)+R(t)$, we arrive at
\begin{equation}\label{esubi6}
    II_4\le \mathcal{C}_3\frac{\ln^{\alpha+1}x}{x^{\min\{2s,1\}}} \frac{\psi(t,x)}{[1-\ln \psi(t,x)]^\alpha},
\end{equation}
for some positive constant $\mathcal{C}_3$.
\par \#\#\# \textbf{Let us estimate $II_5$.} By \eqref{Je}, we have
\[
II_5\le\mathcal{J}_0 x\int^{\frac{x- x_1(t)}{x}}_{B}\exp\left\{-(2s+1)\ln (xu)+\ln\psi(t,x(1-u))\right\} du.
\]
Since $\psi(t,\cdot)$ is log-convex on $[ x_1(t),+\infty)$ by Lemma \ref{esuble1}, we have that, for any $x\in[ x_1(t)+R(t),+\infty)$, $u\mapsto -(2s+1)\ln (xu)+\ln\psi(t,x(1-u))$ is convex on $\left(\frac{1}{x},\frac{x- x_1(t)}{x}\right)$.
It follows that for any $u\in \left(B,\frac{x- x_1(t)}{x}\right)$,
\[
x\exp\left\{-(2s+1)\ln (xu)+\ln\psi(t,x(1-u))\right\}\le  \frac{x}{(xB)^{2s+1}}\psi(t,x(1-B))+ \frac{x}{(x- x_1(t))^{2s+1}}.
\]
\par  Let us estimate both terms in the r.h.s. of the above inequality.  On the one hand, by Lemma \ref{lee2} and $\psi(t,x)\ge v_0(x)=C_L\frac{\ln^q x}{x^{2s}}$, recalling $B=\frac{\gamma_4}{\ln^{\alpha +1}x }$, we get, for all $t$ large enough, say $t\ge t_3$, and $x\ge  x_1(t)+R(t)$,
\[
  \frac{x}{(xB)^{2s+1}}\psi(t,x(1-B))\lesssim  x^{-2s} \left(\ln ^{(\alpha+1)(2s+1)}x\right)\psi(t,x)\lesssim \frac{\ln^{\bar q}x}{x^{s}} \psi^{\frac{3}{2}}(t,x),
\]
where $\bar q:=(2s+1)(\alpha+1)-q$. As a result, since $y\mapsto y^\frac{1}{2}(1-\ln y)^\alpha$ is bounded on $(0,1)$, for all $t\ge t_3$ and $x\ge  x_1(t)+R(t)$, we have
\[
\frac{x}{(xB)^{2s+1}}\psi(t,x(1-B))\le \mathcal{C}_4\frac{\ln^{\bar q}x}{x^{s}} \frac{\psi(t,x)}{(1-\ln \psi(t,x))^\alpha},
\]
for some positive constant $\mathcal{C}_4$.
\par
On the other hand, recalling $R(t)=\gamma_3 t^{-1} x_1(t)$, it follows from \eqref{esubxte} that there is a positive constant $\gamma_5$ such that for all $t\ge t_3$, up to enlarging $t_3$ if necessary, we have 
\[
 x_1(t)+R(t)-\gamma_5 R(t)\ln^{\alpha+1} R(t)\le 0.
\]
One may notice that $x\mapsto x-\gamma_5(x- x_1(t))\ln^{\alpha+1} (x- x_1(t))$ is decreasing on $[ x_1(t)+R(t),+\infty)$ for all $t\ge t_3$, up to enlarging $t_3$ if necessary, and then for all $t\ge t_3$ and $x\ge  x_1(t)+R(t)$, 
\[
x-\gamma_5 (x- x_1(t))\ln^{\alpha+1} (x- x_1(t))\le  x_1(t)+R(t)-\gamma_5 R(t)\ln^{\alpha+1} R(t)\le 0.
\]
Thus, since $y\mapsto \frac{\ln^{\alpha+1}y}{y^{2s}}$ is decreasing for $y$ large enough, by \eqref{xysim} and  \eqref{esubx-xt}, for all $t \ge t_3$ and $x\ge  x_1(t)+R(t)$, up to enlarging $t_3$ if necessary,
\[
 \frac{x}{(x- x_1(t))^{2s+1}}\le\gamma_5  \frac{\ln^{\alpha+1} (x- x_1(t))}{(x- x_1(t))^{2s}}\le \mathcal{C}_5\frac{\kappa^{-1}t^{-p}\psi}{[1-\ln \psi(t,x)]^{q-\alpha-1}},
\]
for some positive constant $\mathcal{C}_5$. As a result, for all $t\ge t_3$ and $x\ge  x_1(t)+R(t)$, we arrive at 
\begin{equation}\label{esubi7}
    II_5\le \mathcal{C}_4\frac{\ln^{\bar q}x}{x^{s}} \frac{\psi(t,x)}{(1-\ln \psi(t,x))^\alpha}+\mathcal{C}_5\frac{\kappa^{-1}t^{-p}\psi}{[1-\ln \psi(t,x)]^{q-\alpha-1}}.
\end{equation}
\par
\#\#
\textbf{Let us estimate $II_3$.} 
In view of \eqref{esubl31}, we have 
\begin{equation}\label{esubi24}
    II_3\le \mathcal{C}_6 e^{-\frac{1}{2s}[r(\alpha+1)t]^\frac{1}{\alpha+1}}\frac{\psi(t,x)}{(1-\ln\psi(t,x))^\alpha}\quad \text{for all }(t,x)\in[\tilde t_1,\infty]\times[x_1(t)+R(t),+\infty),
\end{equation}
for some positive constant $\mathcal{C}_6$.
\par
Owing to \eqref{esubi1},  \eqref{esubi24}, \eqref{esubi6} and \eqref{esubi7}, for all $t\ge \tilde t_2:=\max\{t_0,t_1,t_2,t_3,\tilde t_1\}$ and $x\ge  x_1(t)+R(t)$, we have
\[
		\mathcal{D}[\Psi](t,x)\le \left(\mathcal{C}_3\frac{\ln^{\alpha+1}x}{x^{\min\{2s,1\}}}+ \mathcal{C}_4\frac{\ln^{\bar q}x}{x^{s}}+\mathcal{C}_5\frac{\kappa^{-1} t^{-p}}{(1-\ln \psi(t,x))^{2\alpha-q+1}}+\mathcal{C}_6 e^{-\frac{1}{2s}[r(\alpha+1)t]^\frac{1}{\alpha+1}}\right) \frac{\psi(t,x)}{(1-\ln \psi(t,x))^{\alpha}} .
\]

\end{proof}

\subsubsection{The function \texorpdfstring{$\Psi$}{Psi} is a super-solution to \texorpdfstring{\eqref{oeq1}}{1} at later times}
\begin{proposition}\label{epubpk}
Assume $p=\frac{1}{\alpha+1}$, $q=2\alpha+1$ and $\kappa=\bar\kappa:=\max\left\{e,\frac{2\mathcal{C}_5}{\gamma_1}\right\}$. There is a time $\bar t>t^\star$ such that for all $t\ge \bar t$ and $x> x_1(t)$, we have
    \[
	\frac{\partial \Psi}{\partial t}(t,x)-\mathcal{D}[\Psi](t,x)-f(\Psi(t,x))\ge0.
\] 
\end{proposition}
\begin{proof}[Proof of Proposition \ref{epubpk}]
 In view of Hypothesis \ref{fu}, since $\Psi(t,\cdot)=\psi(t,\cdot)$ on $(x_1(t),+\infty)$, we obtain 
\begin{equation}
    \label{epubfl}
    f(\Psi(t,\cdot))\le \frac{r\psi(t,\cdot)}{(1-\ln \psi(t,\cdot))^\alpha}\quad \text{on }(x_1(t),+\infty).
\end{equation}
 We split the proof according to two space zones: $(x_1(t), x_1(t)+R(t))$ and $[ x_1(t)+R(t),+\infty)$.
 \medskip
 \par
 \noindent
\# \textbf{Zone 1: $x\in (x_1(t), x_1(t)+R(t))$.} 
By \eqref{naemt}, \eqref{naemx}, \eqref{epubfl} and part II in Lemma \ref{epubed}, for all $t\ge \tilde t_1$ and $x\in( x_1(t), x_1(t)+R(t))$, we arrive at 
\[
\begin{aligned}
& \frac{\partial \Psi}{\partial t}(t,x)-\mathcal{D}[\Psi](t,x)-f(\Psi(t,x))
\\
&\ge \gamma_1 t^{-\frac{1}{\alpha+1}}\frac{\psi(t,x)}{(1-\ln \psi(t,x))^\alpha}-\mathcal{C}_1 e^{-\frac{1}{2s}[r(\alpha+1)t]^\frac{1}{\alpha+1}}\frac{\psi(t,x)}{(1-\ln \psi(t,x))^\alpha}+\mathcal{J}_3R(t)\frac{\partial \psi}{\partial x}(t, x_1(t))\\
&\ge \left(\frac{1}{2}\gamma_1 t^{-\frac{1}{\alpha+1}}-\mathcal{C}_1 e^{-\frac{1}{2s}[r(\alpha+1)t]^\frac{1}{\alpha+1}}\right)\frac{\psi(t,x)}{(1-\ln \psi(t,x))^\alpha}+\frac{1}{2}\gamma_1 t^{-\frac{1}{\alpha+1}}\frac{\psi(t,x)}{(1-\ln \psi(t,x))^\alpha}-\gamma_2^{-1} \mathcal{J}_3R(t) \frac{\ln^\alpha  x_1(t)}{ x_1(t)}.
\end{aligned}
\]
We recall $R(t)=\gamma_3 t^{-1} x_1(t)$ for some $\gamma_3>0$ to be chosen later.
By \eqref{esubfl}, we get \[x_{\frac{1}{2}}(t)- x_1(t)\asymp t^{-\frac{\alpha}{\alpha+1}+\frac{1}{2s}\left(\frac{q}{\alpha+1}+p\right)}e^{\frac{1}{2s}[r(\alpha+1)t]^\frac{1}{\alpha+1}},\]
and it follows from \eqref{esubxte} that  there is $\bar t_1\ge \tilde t_1$ such that for all $t\ge \bar t_1$, we have
\[
x_{\frac{1}{2}}(t)\ge  x_1(t)+ R(t),
\]
It follows from the fact that $\psi(t,\cdot)$ is non-increasing on $(x_1(t),+\infty)$ that, for all $t\ge \bar t_1$ and $x\in (x_1(t), x_1(t)+R(t))$, 
\[
\frac{\psi(t,x)}{(1-\ln \psi(t,x))^\alpha}\ge \frac{\psi(t, x_1(t)+R(t))}{[1-\ln \psi(t, x_1(t)+R(t))]^\alpha}\ge \frac{\psi \left(t,x_{\frac{1}{2}}(t)\right)}{\left[1-\ln \psi \left(t,x_{\frac{1}{2}}(t)\right)\right]^\alpha}=\frac{1}{2(1+\ln 2)^\alpha}.
\]
Then, by \eqref{esubxte}, for all $t\ge \bar t_1$ and $x\in( x_1(t), x_1(t)+R(t))$, we achieve
\[
\frac{1}{2}\gamma_1 t^{-\frac{1}{\alpha+1}}\frac{\psi(t,x)}{(1-\ln \psi(t,x))^\alpha}-\gamma_2^{-1} \mathcal{J}_3 R(t)\frac{\ln^\alpha  x_1(t)}{ x_1(t)}\ge  t^{-\frac{1}{\alpha+1}}\left(\frac{\gamma_1}{4(1+\ln 2)^\alpha}-\mathcal{J}_3\gamma _2^{-1}\gamma_3\frac{\ln^\alpha  x_1(t)}{t^\frac{\alpha}{\alpha+1}}\right)\ge 0,
\]
by taking $\gamma_3$ small enough, say $\gamma_3\le \tilde \gamma_3$.
On the other hand, for $t\ge \bar t_1$ up to enlarging $\bar t_1$ if necessary, we have 
\[
\frac{1}{2}\gamma_1 t^{-\frac{1}{\alpha+1}}-\mathcal{C}_1 e^{-\frac{1}{2s}[r(\alpha+1)t]^\frac{1}{\alpha+1}}\ge 0.
\]
Therefore, for all $t\in[ \bar t_1,\infty)$ and $x\in( x_1(t), x_1(t)+R(t))$, we achieve
\begin{equation}
    \frac{\partial \Psi}{\partial t}(t,x)-\mathcal{D}[\Psi](t,x)-f(\Psi(t,x))\ge 0.
\end{equation}
\medskip
\par
\noindent
\# \textbf{Zone 2: $x\ge  x_1(t)+R(t)$.} Let us estimate $\mathcal{D}[\Psi]$ for $x\ge  x_1(t)+R(t)$. 
In view of \eqref{naemt} and Hypothesis \ref{fu}, therefore, we have, for all $t\ge \tilde t_2$ and $x\ge  x_1(t)+R(t)$,
\[
	\begin{aligned}
		&\frac{\partial \Psi}{\partial t}(t,x)-\mathcal{D}[\Psi](t,x)-f(\Psi(t,x))\\
		&\ge t^{-\frac{1}{\alpha+1}}\left\{\gamma_1- \mathcal{C}_5\kappa^{-1} t^{\frac{1}{\alpha+1}-p}(1-\ln\psi(t,x))^{2\alpha-q+1}+o_{t\to\infty}(1)\right\}\frac{\psi(t,x)}{(1-\ln \psi(t,x))^{\alpha}}.
	\end{aligned}
\]
To validate the above computations, we need $p\ge \frac{1}{\alpha+1}$, $q\ge2\alpha+1$ and $\kappa\ge \bar\kappa$. Thus, we can choose $p=\frac{1}{\alpha+1}$, $q=2\alpha+1$ and $\kappa=\bar\kappa$. 
As a result, for such $p$, $q$ and $\kappa$, there is a time $\bar t_2>\tilde t_2$ such that for all $t\ge \bar t_2$ and $x\ge  x_1(t)+R(t)$, we have
\[
	\frac{\partial \Psi}{\partial t}(t,x)-\mathcal{D}[\Psi](t,x)-f(\Psi(t,x))\ge0.
\]
\par 
Therefore, collecting the estimates in the two zones, by taking $\bar t:=\max\{\bar t_1,\bar t_2\}$, $p=\frac{1}{\alpha+1}$, $q=2\alpha+1$ and $\kappa=\bar\kappa$, for all $t\ge \bar t$ and $x> x_1(t)$, we achieve
\[
	\frac{\partial \Psi}{\partial t}(t,x)-\mathcal{D}[\Psi](t,x)-f(\Psi(t,x))\ge0.
 \]

\end{proof}

\subsubsection{Proof of the upper bound in Theorem \ref{al}}
	In view of Hypothesis \ref{ic1}, by \eqref{epubv0u0}, one has $u_0\le v_0\le\Psi(0,\cdot)\le \Psi(\bar t,\cdot)$. 
	Thus, by proposition \ref{epubpk} and the fact that $u(t,x)\le 1= \Psi(t+\bar t,x)$ for all $t>0$ and $x\le x_1(t+\bar t)$, the comparison principle yields $u(t,x)\le \Psi(t+\bar t,x)$ for all $(t,x)\in\mathbb{R}^+\times \mathbb{R}$. Thus, for any $\lambda\in(0,1)$, we have
	\[
		\lambda \le u(t,X_\lambda(t))\le \Psi(t+\bar t,X_\lambda(t)),\quad \forall t>0.
	\]
 Therefore, by \eqref{xysim}, there exists $T_\lambda>0$ such that, for $t>T_\lambda$, we have
	\[
 \begin{aligned}
		X_\lambda(t)&\le v^{-1}_0\left\{(1+\kappa (t+\bar t)^p)^{-1}\exp\left\{1-[(1-\ln \lambda)^{\alpha+1}+r(\alpha+1)(t+\bar t)]^{\frac{1}{\alpha+1}}\right\}\right\}\lesssim t^{\frac{1}{s}} e^{\frac{1}{2s} [r (\alpha+1)t]^{\frac{1}{\alpha+1}}}.
  \end{aligned}
	\]
The proof of the upper bound in Theorem \ref{al} is now complete.
\subsection{The lower bound}
To obtain a precise lower bound on the level sets of the solution to \eqref{oeq1}, we need to construct a subtle sub-solution. For equation \eqref{oeq1}, it turns out that it is not enough to construct a suitable sub-solution by using the solution of the simplified ordinary differential equation $\frac{\partial u}{\partial t}=\frac{ru}{(1-\ln u)^\alpha}$. Therefore, we should use the full ordinary differential equation $\frac{\partial u}{\partial t}=\frac{ru}{(1-\ln u)^\alpha}(1-u^{p})$ for $p\in(0,1)$. Unfortunately, its solution is not explicit, which causes some tedious difficulties in subsequent calculations.
\subsubsection{Definition of the sub-solution and basic computations}
Let us define
\[
	\varphi(t,x)=\exp{\left\{1-\left[\left(1+2s\ln x\right)^{\alpha+1}- r(\alpha+1)[t-h(t,x)]\right]^{\frac{1}{\alpha+1}}\right\}},
\]
where 
\[
h(t,x):=\kappa\frac{ t^{\frac{\alpha}{\alpha+1}}e^{p[r(\alpha+1)t]^\frac{1}{\alpha+1}}}{x^{2sp}},
\]
for some constants $\kappa\in(0,+\infty)$ and $p\in(0,1)$ to be chosen later. Here, we list some useful lemmas that we prove in Appendix \ref{appendix-exp-acc-low}. We first obtain:
\begin{lemma}\label{eslbl1}
Assume $\kappa \le\kappa_0:=\min \left\{\frac{[r(\alpha+1)]^\frac{\alpha}{\alpha+1}}{2^{\alpha+1} rpe^{2p}},\frac{1}{2e^{2p}[r(\alpha+1)]^\frac{1}{\alpha+1}},\frac{1}{4\tilde \gamma_1 e^{2p}},\frac{s[r(\alpha+1)]^\frac{\alpha}{\alpha+1}}{2^\alpha r\tilde \gamma_2}\right\}$, where $\tilde \gamma_1:= \frac{p}{\alpha+1}[r(\alpha+1)]^\frac{1}{\alpha+1}$ and  $\tilde \gamma_2:=2sp(2sp+1)e^{2p}$. Let $x_\Lambda(t)$ be the position of the level set of $\varphi$ of height $\Lambda\in(0,e]$, that is, $\varphi(t,x_\Lambda(t))=\Lambda$.
\begin{itemize}
    \item[(i).] For any $\Lambda\in(0,e)$, there is a time $\tilde t>0$ such that, for any $t\ge \tilde t$ and  $x\ge x_\Lambda(t)$, $\varphi(t,x)$ is well-defined, strictly increasing in $t$, and strictly decreasing and convex in $x$.
    \item[(ii).]  For any $\Lambda\in(0,e)$, the position $x_\Lambda(t)$ is  strictly decreasing in $\Lambda$  and  unique  such that $\varphi(t,x_\Lambda(t))=\Lambda$, and satisfies
\begin{equation}\label{eslbxtl}
x_\Lambda(t)> e^{-\frac{1}{s} }e^{\frac{1}{2s}\left[r(\alpha+1)t\right]^\frac{1}{\alpha+1}}.
\end{equation}
\end{itemize}

\end{lemma}

To make the above lemma valid, we fix $\kappa=\kappa_0$ in this subsection.
With the above lemma, an estimate of $\varphi$ can be derived. 
\begin{lemma}\label{varphietx}
Let $\Lambda\in(0,1)$. We have the following estimate for $\varphi$:
    \begin{equation}\label{eslbephi}
  \frac{1}{x^{2s}}\le \varphi(t,x)\le \frac{e^{[r(\alpha+1)t]^\frac{1}{\alpha+1}}}{x^{2s}}\quad \text{ for all }t> \tilde t \text{ and }x\ge x_\Lambda(t).
\end{equation}
 For all $t> \tilde t$ and $x\ge x_\Lambda(t) $, we have
\begin{equation}\label{eslbphit0}
    \frac{\partial \varphi}{\partial t}\le \frac{r\varphi}{(1-\ln \varphi)^\alpha}(1-\tilde \gamma_1 \kappa \varphi^p),
\end{equation}
where $\tilde \gamma_1= \frac{p}{\alpha+1}[r(\alpha+1)]^\frac{1}{\alpha+1}$, and if $\varphi\le\min\{ e^{\frac{1}{2s}-\alpha},e\}$, then
\begin{equation}\label{eslbphix0}
\frac{\partial \varphi}{\partial x}  \ge  -2s\varphi^{1+\frac{1}{2s}}.
\end{equation}
Moreover, we have
    \begin{equation}
     \label{eslbpx0}   
 \lim_{t\to \infty}   \frac{\partial \varphi}{\partial x}(t,x_\Lambda(t))=0 \quad \text{uniformly for }\Lambda\in(0,1) .
    \end{equation}

\end{lemma}

\begin{lemma}\label{eslbl2}
   For any $\Lambda\in(0,1)$, we have
    \[
    x_\Lambda(t)\asymp_\Lambda e^{\frac{1}{2s}[r(\alpha+1)t]^\frac{1}{\alpha+1}}\quad \text{for all }t\ge \tilde t,
\]
and for any $0<\Lambda_1<\Lambda_2 < 1$, we have
\[
x_{\Lambda_1}(t)-x_{\Lambda_2}(t)\to +\infty\quad \text{as } t\to \infty.
\]
\end{lemma}

\begin{lemma}\label{ge}
  For any $\varepsilon\in\left(0,\frac{2}{3}\right)$,  there exist a function $g_\varepsilon\in C^2:[0,1]\rightarrow[0,\varepsilon]$ such that  
    \[
    g_\varepsilon(y) = y \quad\text{for all } y\in \left[0, \frac{\varepsilon}{2}\right]\quad\text{and}\quad g_\varepsilon(y)=\varepsilon\quad\text{for all } y\in \left[\frac{3}{2}\varepsilon,1\right],
    \]
 and for all $y\in(0,1)$, 
\begin{equation}\label{nb}
g_\varepsilon(y)\le y,\quad g_\varepsilon'(y)\frac{y}{(1-\ln y)^\alpha}\le \frac{g_\varepsilon(y)}{(1-\ln g_\varepsilon(y))^\alpha},\quad  0 \le g'_\varepsilon(y)\le 1\quad\text{and}\quad  g_\varepsilon''(y)\ge -C_\varepsilon \quad \text{for some }C_\varepsilon>0. 
\end{equation}
\end{lemma}

\par
Fix any $\varepsilon\in\left(0,\varepsilon_0:=\min\{\frac{1}{4},\xi_0\}\right)$, where $\xi_0$ is defined in Hypothesis \ref{fu}, and let
\[
  p=\frac{1}{2}\min\left\{\frac{1}{2s},2s\right\}.
\] 
Let $g_\varepsilon$ be a function given by Lemma \ref{ge}. For all $t\ge \tilde t$, we then define 
\[
	\Phi(t,x):=\left\{
	\begin{aligned}
		&\varepsilon,& x<  x_{\frac{3}{2}\varepsilon}(t) ,\\
		&g_\varepsilon(\varphi(t,x)),& x\ge  x_{\frac{3}{2}\varepsilon}(t) .
	\end{aligned}
	\right.
\]
Observe that for any $t\ge \tilde t$, $\Phi(t,\cdot)$ is at least $C^2$ on $\mathbb{R}$.
It follows that for $x<  x_{\frac{3}{2}\varepsilon}(t) $, we have $\frac{\partial \Phi}{\partial t}=0$, $\frac{\partial \Phi}{\partial x}=0$ and $\frac{\partial^2 \Phi}{\partial x^2}=0$, while for $x\ge  x_{\frac{3}{2}\varepsilon}(t) $, by the definition of $g_\varepsilon$ and Lemma \ref{eslbl1}, for all $t\ge \tilde t$, we have
\begin{equation}\label{eslbPhit}
    \frac{\partial \Phi}{\partial t}=g_\varepsilon'(\varphi)\frac{\varphi}{(1-\ln \varphi)^\alpha}\left[r -r\frac{\partial h}{\partial t}(t,x)\right]\le \frac{r\varphi}{(1-\ln \varphi)^\alpha}(1-\tilde \gamma_1 \kappa \varphi^p),
\end{equation}
 \begin{equation}\label{Phix}
    \frac{\partial \Phi}{\partial x}=g_\varepsilon'(\varphi)\frac{\partial \varphi}{\partial x}\ge \frac{\partial \varphi}{\partial x},
\end{equation}
and 
\begin{equation}\label{Phixx}
\frac{\partial^2 \Phi}{\partial x^2}=g''_\varepsilon(\varphi)\left(\frac{\partial \varphi}{\partial x}\right)^2+g'_\varepsilon(\varphi)\varphi_{xx}\ge -C_\varepsilon\left(\frac{\partial \varphi}{\partial x}\right)^2.
\end{equation}
Unless otherwise specified, throughout the subsection, we assume that $t\in[\tilde t,\infty)$.
\subsubsection{Estimation of \texorpdfstring{$\mathcal{D}[\Phi](t,\cdot)$ on $\mathbb{R}$}{D[phi]}}
Now, let us estimate $\mathcal{D}[\Phi]$. We split into three space zones, as follows.

\begin{lemma}\label{eplbde}
Assume $\varepsilon\in(0,\varepsilon_0)$. For all $t\ge \tilde t$, we have the following estimate for $\mathcal{D}[\Phi]$.
\begin{itemize}
    \item[I.]  When $x\in \left(-\infty, x_{\frac{3}{2}\varepsilon}(t)\right) $: Given some quantity $B_1>1$, we have
\[
    \mathcal{D}[\Phi]\ge -\frac{\mathcal{J}_0}{2s} \frac{\varepsilon}{B_1^{2s}} -C_\varepsilon\left(\mathcal{J}_1+\mathcal{J}_0\int_{1}^{2B_1} y^{1-2s} dy\right)\left(\frac{\partial \varphi}{\partial x}\right)^2\left(t, x_{\frac{3}{2}\varepsilon}(t)\right).
\]
\item[II.] When $x\in\left[ x_{\frac{3}{2}\varepsilon}(t) ,x_{\frac{\varepsilon}{4}}(t)\right]$: Given some quantity $B_2>1$, we have
    \[
\mathcal{D}[\Phi](t,x)\ge -\frac{\mathcal{J}_0}{2s}\frac{\varepsilon }{B_2^{2s}}-C_\varepsilon\left(\mathcal{J}_1+\mathcal{J}_0\int_1^{B_2} y^{1-2s}dy\right)\left(\frac{\partial \varphi}{\partial x}\right)^2\left(t, x_{\frac{3}{2}\varepsilon}(t)\right).    \]
\item[III.] When $x\in \left(x_{\frac{\varepsilon}{4}}(t),+\infty\right):$ For $s\in\left(0,\frac{1}{2}\right]$, given some $B_3>1$, we have 
    \[
        \mathcal{D}[\Phi]\ge  -\frac{\mathcal{J}_0}{2s}\frac{\varepsilon}{B_3^{2s}}+\mathcal{J}_0\frac{\partial \varphi}{\partial x}(t,x)\int_{1}^{B_3}y^{-2s} dy,
    \]
while,  for $s\in\left(\frac{1}{2},+\infty\right)$, we have
    \[
    \mathcal{D}[\Phi]\ge  \frac{\mathcal{J}_0}{2s-1}\frac{\partial \varphi}{\partial x}(t,x).
\]
\end{itemize}
\end{lemma}
\begin{proof}[Proof of Lemma \ref{eplbde}]
We divide the proof of the above lemma into three parts, corresponding to the three zones of positions: $x\in\left(-\infty,  x_{\frac{3}{2}\varepsilon}(t)\right) $, $x\in\left[ x_{\frac{3}{2}\varepsilon}(t) ,x_{\frac{\varepsilon}{4}}(t)\right]$ and $x\in \left(x_{\frac{\varepsilon}{4}}(t),+\infty\right)$.
\medskip
\par
\noindent
\#\textbf{Start with $I$: $x\in \left(-\infty,  x_{\frac{3}{2}\varepsilon}(t) \right)$.}
 We split the interval $\left(-\infty, x_{\frac{3}{2}\varepsilon}(t) \right)$ into two parts $\left(-\infty, x_{\frac{3}{2}\varepsilon}(t) -B_1\right]$ and $\left( x_{\frac{3}{2}\varepsilon}(t) -B_1, x_{\frac{3}{2}\varepsilon}(t) \right)$ with some quantity $B_1>1$ to be chosen later.
\par\#\#
\textbf{When  $x\in \left(-\infty, x_{\frac{3}{2}\varepsilon}(t) -B_1\right]$:} By the definition of $\mathcal{D}[\Phi]$ and \eqref{Je}, since $\Phi(t,x)=\varepsilon$ for $x\le  x_{\frac{3}{2}\varepsilon}(t) $, we get
\[
\begin{aligned}
    \mathcal{D}[\Phi]&= \int_{ x_{\frac{3}{2}\varepsilon}(t) -x}^{+\infty}J(y)\Big(\Phi(t,x+y)-\varepsilon\Big)dy\ge -\varepsilon\mathcal{J}_0\int_{ x_{\frac{3}{2}\varepsilon}(t) -x}^{+\infty}\frac{1}{y^{1+2s}}dy\ge -\frac{\varepsilon\mathcal{J}_0}{2s}\frac{1}{\left( x_{\frac{3}{2}\varepsilon}(t) -x\right)^{2s}}\ge- \frac{\varepsilon\mathcal{J}_0}{2s}\frac{1}{B_1^{2s}}.
\end{aligned}
\]
\par\#\#
\textbf{When $x\in \left( x_{\frac{3}{2}\varepsilon}(t) -B_1, x_{\frac{3}{2}\varepsilon}(t) \right)$: } By the definition of $\mathcal{D}[\Phi]$, we write
\[
    \mathcal{D}[\Phi]=\int_{ x_{\frac{3}{2}\varepsilon}(t) -x+B_1}^{+\infty}J(y)\Big(\Phi(t,x+y)-\Phi(t,x)\Big)dy+\int_{ x_{\frac{3}{2}\varepsilon}(t) -x}^{ x_{\frac{3}{2}\varepsilon}(t) -x+B_1}J(y)\Big(\Phi(t,x+y)-\Phi(t,x)\Big)dy:=I_1+I_2.
\]
For the integral $I_1$, since $\Phi(t,\cdot)=\varepsilon$ on $\left(-\infty, x_{\frac{3}{2}\varepsilon}(t)\right)$, by \eqref{Je}, we get
\[
    I_1\ge -\varepsilon\int_{ x_{\frac{3}{2}\varepsilon}(t) -x+B_1}^{+\infty}J(y)dy \ge -\varepsilon\mathcal{J}_0\int_{ x_{\frac{3}{2}\varepsilon}(t) -x+B_1}^{+\infty}\frac{1}{y^{1+2s}}dy\ge -\frac{\varepsilon\mathcal{J}_0}{2s} \frac{1}{B_1^{2s}}.
\]
For the integral $I_2$,  by Taylor's theorem, $I_2$ can be written as 
\[
    I_2=\int_{ x_{\frac{3}{2}\varepsilon}(t) -x}^{ x_{\frac{3}{2}\varepsilon}(t) -x+B_1}J(y)\Big(\Phi(t,x+y)-\Phi(t,x)\Big)dy=\int_{ x_{\frac{3}{2}\varepsilon}(t) -x}^{ x_{\frac{3}{2}\varepsilon}(t) -x+B_1}\int_0^1 yJ(y)\frac{\partial \Phi}{\partial x}(t,x+\tau y)d\tau dy.
\]
Since $\frac{\partial \Phi}{\partial x}(t,\cdot)\in C^1(\mathbb{R})$ and $\frac{\partial \Phi}{\partial x}(t,\cdot)=0$ on $\left(-\infty, x_{\frac{3}{2}\varepsilon}(t)\right)$, we get
\[
\begin{aligned}
   I_2&=\int_{ x_{\frac{3}{2}\varepsilon}(t) -x}^{ x_{\frac{3}{2}\varepsilon}(t) -x+B_1}\int_0^1 yJ(y)\left(\frac{\partial \Phi}{\partial x}(t,x+\tau y)-\frac{\partial \Phi}{\partial x}(t,x)\right)d\tau dy\\
   &=\int_{ x_{\frac{3}{2}\varepsilon}(t) -x}^{ x_{\frac{3}{2}\varepsilon}(t) -x+B_1}\int_0^1 \int_0^1 \tau y^2J(y)\frac{\partial^2 \Phi}{\partial x^2}(t,x+\sigma\tau y)d\sigma d\tau dy.
\end{aligned}
\]
It follows from $\frac{\partial^2 \Phi}{\partial x^2}(t,\cdot)=0$ on $\left(-\infty, x_{\frac{3}{2}\varepsilon}(t)\right)$ that 
\[
    I_2=\int_0^1 \int_0^1\int_{0}^{ x_{\frac{3}{2}\varepsilon}(t) -x+B_1}\tau y^2 J(y)\frac{\partial^2 \Phi}{\partial x^2}(t,x+\sigma\tau y)dyd\sigma d\tau.
\]
Since $ x_{\frac{3}{2}\varepsilon}(t) -x+B_1\le 2B_1$, by the convexity of $\varphi(t,\cdot)$, using \eqref{Phixx}, we get 
\[
\begin{aligned}
    I_2&\ge -C_\varepsilon\left(\frac{\partial \varphi}{\partial x}\right)^2\left(t, x_{\frac{3}{2}\varepsilon}(t)\right)\int_0^1 \int_0^1\int_{0}^{2B_1} J(y)y^2\tau dyd\sigma d\tau\\
    &\ge -C_\varepsilon\left(\frac{\partial \varphi}{\partial x}\right)^2\left(t, x_{\frac{3}{2}\varepsilon}(t)\right)\left(\int_{0}^{1} J(y)y^2dy+\int_{1}^{2B_1} J(y)y^2 dy\right)\\
    &\ge  -C_\varepsilon\left(\frac{\partial \varphi}{\partial x}\right)^2\left(t, x_{\frac{3}{2}\varepsilon}(t)\right)\left(\mathcal{J}_1+\mathcal{J}_0\int_{1}^{2B_1} y^{1-2s} dy\right),
    \end{aligned}
\]
where the last inequality uses \eqref{Je}.
As a result, for all $t\ge \tilde t$ and $x<  x_{\frac{3}{2}\varepsilon}(t)$,  in both cases we obtain 
\[
    \mathcal{D}[\Phi]\ge -\frac{\mathcal{J}_0}{2s} \frac{\varepsilon}{B_1^{2s}} -C_\varepsilon\left(\frac{\partial \varphi}{\partial x}\right)^2\left(t, x_{\frac{3}{2}\varepsilon}(t)\right)\left(\mathcal{J}_1+\mathcal{J}_0\int_{1}^{2B_1} y^{1-2s} dy\right).
\]

\medskip
 \par
 \noindent \#
 \textbf{Continue with $II$:}
$x\in\left[ x_{\frac{3}{2}\varepsilon}(t) ,x_{\frac{\varepsilon}{4}}(t)\right].$ Let us estimate $\mathcal{D}[\Phi]$ on $\left[x_{\frac{3}{2}\varepsilon}(t) ,x_{\frac{\varepsilon}{4}}(t)\right]$.
 Since $\Phi(t,\cdot)$ is non-increasing on $\mathbb{R}$, by \eqref{Phixx}, for some $B_2>1$, we get
\[
\begin{aligned}
    \mathcal{D}[\Phi]&=\frac{1}{2}\int_{|y|\le B_2 } J(y)\Big(\Phi(t,x+y)-2\Phi(t,x)+\Phi(t,x-y)\Big)dy+\int_{|y|> B_2 } J(y)\Big(\Phi(t,x+y)-\Phi(t,x)\Big)dy\\
    &\ge\frac{1}{2}\int_{|y|\le B_2 }\int_0^1\int_{0}^{1} J(y)y^2\tau\frac{\partial^2 \Phi}{\partial x^2}(t,x+\sigma\tau y) dy+\int_{y> B_2 } J(y)\Big(\Phi(t,x+y)-\Phi(t,x)\Big)dy\\
    &\ge -\frac{C_\varepsilon}{2}\sup_{\substack{|\xi|<B_2\\ x+\xi \ge x_{\frac{3}{2}\varepsilon(t)}}}\left(\frac{\partial \varphi}{\partial x}\right)^2(t,x+\xi)\int_{|y|\le B_2 }J(y)y^2 dy+\int_{y> B_2 } J(y)\Big(\Phi(t,x+y)-\Phi(t,x)\Big)dy:=II_1+II_2.
    \end{aligned}
\]
For $II_1$, by \eqref{Je} and the convexity of $\varphi(t,\cdot)$, we have
\[
\begin{aligned}
    II_1&=-\frac{C_\varepsilon}{2}\sup_{\substack{|\xi|<B_2\\ x+\xi \ge x_{\frac{3}{2}\varepsilon(t)}}}\left(\frac{\partial \varphi}{\partial x}\right)^2(t,x+\xi)\left(\int_{|y|\le 1 } J(y)y^2 dy+\int_{1<|y|\le B_2 } J(y)y^2dy\right)\\
    &=-C_\varepsilon\left(\mathcal{J}_1+\mathcal{J}_0\int_1^{B_2} y^{1-2s}dy\right)\left(\frac{\partial \varphi}{\partial x}\right)^2\left(t, x_{\frac{3}{2}\varepsilon}(t)\right),
    \end{aligned}
\]
while for $II_2$, since $0\le\Phi\le \varepsilon$, by \eqref{Je}, we have
\[
    II_2\ge-\Phi(t,x)\int_{y> B_2 } J(y)dy=-\frac{\mathcal{J}_0}{2s}\frac{\varepsilon }{B_2^{2s}}.
\]
Therefore, by collecting the estimates of $II_1$ and $II_2$, we achieve
\[
    \mathcal{D}[\Phi]\ge -\frac{\mathcal{J}_0}{2s}\frac{\varepsilon }{B_2^{2s}}-C_\varepsilon\left(\mathcal{J}_1+\mathcal{J}_0\int_1^{B_2} y^{1-2s}dy\right)\left(\frac{\partial \varphi}{\partial x}\right)^2\left(t, x_{\frac{3}{2}\varepsilon}(t)\right).
\]
\medskip
\par
\noindent \#
\textbf{Finally, focus on $III$:} $x\in \left(x_{\frac{\varepsilon}{4}}(t),+\infty\right).$
By the definition of $\mathcal{D}[\Phi]$, since $\Phi(t,\cdot)$ is non-increasing, we write
\[
\begin{aligned}
    \mathcal{D}[\Phi]&\ge   P.V. \int_{-1}^{1}J(y)\Big(\Phi(t,x+y)-\Phi(t,x)\Big)dy+\int_{1}^{+\infty}J(y)\Big(\Phi(t,x+y)-\Phi(t,x)\Big)dy:=III_1+III_2.
\end{aligned}
\]
By \eqref{Je} and the convexity of $\Phi(t,\cdot)$ on $\left(x_{\frac{\varepsilon}{4}}(t),+\infty\right)$, we get
\[
    III_1=\frac{1}{2}\int_{-1}^{1}\int_0^1\int_{0}^{1}J(y)y^2\tau\frac{\partial^2 \Phi}{\partial x^2}(t,x+\sigma\tau y)d\sigma d\tau dy\ge 0.
\]
To end the estimate of $\mathcal{D}[\Phi]$ for $x> x_{\frac{\varepsilon}{4}}(t)$, let us consider two cases: $s\in\left(0,\frac{1}{2}\right]$ and $s\in\left(\frac{1}{2},+\infty\right)$.
\par\#\#
\textbf{Case $1$: $s\in(0,\frac{1}{2}]$}.
Let us estimate $III_2$. For some $B_3>1$, we write
\[
    III_2=\int_{1}^{B_3}J(y)\Big(\Phi(t,x+y)-\Phi(t,x)\Big)dy+\int_{B_3}^{+\infty}J(y)\Big(\Phi(t,x+y)-\Phi(t,x)\Big)dy:=III_3+III_4.
\]
By \eqref{Phix}, the convexity of $\varphi(t,\cdot)$ and \eqref{Je}, we have
\[
\begin{aligned}
    III_3=\int_{1}^{B_3}\int_0^1J(y)y\frac{\partial \Phi}{\partial x}(t,x+\tau y)d\tau dy\ge\frac{\partial \varphi}{\partial x}(t,x)\int_{1}^{B_3}J(y)y dy
    &\ge\mathcal{J}_0\frac{\partial \varphi}{\partial x}(t,x)\int_{1}^{B_3}y^{-2s} dy.
\end{aligned}
\]
Since $0< \Phi(t,\cdot)< \varepsilon$ on $\left( x_{\frac{\varepsilon}{4}}(t),+\infty\right)$, it follows from \eqref{Je} that
\[
    III_4\ge -\mathcal{J}_0\Phi(t,x)\int_{B_3}^{+\infty}y^{-1-2s}dy\ge -\frac{\mathcal{J}_0}{2s}\frac{\varepsilon}{B_3^{2s}}.
\]
As a result, collecting the above estimates, for all $t\ge \tilde t$ and $x>x_\frac{\varepsilon}{4}(t)$, we have
\[
    \mathcal{D}[\Phi]\ge   -\frac{\mathcal{J}_0}{2s}\frac{\varepsilon}{B_3^{2s}}+\mathcal{J}_0\frac{\partial \varphi}{\partial x}(t,x)\int_{1}^{B_3}y^{-2s} dy.
\]

\par\#\#
\textbf{Case $2$: $s\in\left(\frac{1}{2},+\infty\right)$.} Let us estimate $III_2$. By \eqref{Phix}, the convexity of $\varphi(t,\cdot)$ and using the fundamental theorem of calculus, we get
\[
III_2=\int_{1}^{+\infty}\int_0^1J(y)y\frac{\partial \Phi}{\partial x}(t,x+\tau y)d\tau dy \ge \frac{\partial \varphi}{\partial x}(t,x)\int_{1}^{+\infty}J(y)y dy.  
\]
Since $\frac{\partial \varphi}{\partial x}(t,x)\le 0$, by \eqref{Je}, we have 
\[
    III_2\ge\frac{\mathcal{J}_0}{2s-1}\frac{\partial \varphi}{\partial x}(t,x).  
\]
As a result, in this case, collecting  the above estimates, for all $t\ge \tilde t$ and $x>x_\frac{\varepsilon}{4}(t)$, we have
\[
    \mathcal{D}[\Phi]\ge \frac{\mathcal{J}_0}{2s-1}\frac{\partial \varphi}{\partial x}(t,x).
\]

\end{proof}

\subsubsection{The function \texorpdfstring{$\Phi$}{Phi} is a sub-solution to \texorpdfstring{\eqref{oeq1}}{[1]} at later times}
In this subsubsection, we show that $\Phi$ is a sub-solution to equation \eqref{oeq1} for large times $t$. 

\begin{proposition}\label{eslbprop1}
There exists $\underline \varepsilon>0$ such that, for any $\varepsilon\in(0, \underline\varepsilon)$, there exists $\underline t(\varepsilon)$ such that 
    \[
   \frac{\partial \Phi}{\partial t}- \mathcal{D}[\Phi]-f(\Phi)\le 0\quad \text{for all }t\ge \underline t\text{ and } x\in\mathbb{R}.
\]
\end{proposition}

\begin{proof}[Proof of Proposition \ref{eslbprop1}]
To prove the proposition, we split into three space zones: $\left(-\infty,  x_{\frac{3}{2}\varepsilon}(t)\right)$, $\left[ x_{\frac{3}{2}\varepsilon}(t) , x_{\frac{\varepsilon}{4}}(t)\right]$ and $\left(x_{\frac{\varepsilon}{4}}(t),+\infty\right)$. 
\medskip
\par
\noindent
\textbf{Zone 1:} $x\in \left(-\infty,  x_{\frac{3}{2}\varepsilon}(t)\right).$  In view of  the definition of $\Phi$ and Hypothesis \ref{fu},  for $x<  x_{\frac{3}{2}\varepsilon}(t) $, we have
$$\frac{\partial \Phi}{\partial t}=0\quad\text{and}\quad f(\Phi)\ge \frac{r\varepsilon}{(1-\ln \varepsilon)^\alpha}(1-K\varepsilon).$$
Now, by the part I of Lemma \ref{eplbde} and choosing $B_1=\left(\frac{2(1-\ln \varepsilon)^\alpha\mathcal{J}_0}{sr(1-K\varepsilon)}+1\right)^\frac{1}{2s}$, it follows that 
\[
    \mathcal{D}[\Phi]\ge -\frac{1}{4}\frac{r\varepsilon}{(1-\ln \varepsilon)^\alpha}(1-K\varepsilon) -C_\varepsilon\left(\frac{\partial \varphi}{\partial x}\right)^2\left(t, x_{\frac{3}{2}\varepsilon}(t)\right)\left(\mathcal{J}_1+\mathcal{J}_0\int_{1}^{2B_1} y^{1-2s} dy\right).
\]
In view of  \eqref{eslbpx0}, we have uniform convergence of $\left(\frac{\partial \varphi}{\partial x}\right)^2\left(t, x_{\frac{3}{2}\varepsilon}(t)\right)$ to $0$ as $t$ goes to infinity. It follows that there is a time $\underline t_1\ge \tilde t$ such that for all $t\ge \underline t_1$ and $x<  x_{\frac{3}{2}\varepsilon}(t)$, we have
\[
    C_\varepsilon\left(\frac{\partial \varphi}{\partial x}\right)^2\left(t, x_{\frac{3}{2}\varepsilon}(t)\right)\left(\mathcal{J}_1+\mathcal{J}_0\int_{1}^{2B_1} y^{1-2s} dy\right)\le-\frac{1}{4}\frac{r\varepsilon}{(1-\ln \varepsilon)^\alpha}(1-K\varepsilon).
\]
Thus, for all $t\ge \underline t_1$ and $x<  x_{\frac{3}{2}\varepsilon}(t)$, we obtain
\[
    \mathcal{D}[\Phi]\ge -\frac{1}{2}\frac{r\varepsilon}{(1-\ln \varepsilon)^\alpha}(1-K\varepsilon).
\]
Therefore, for all $t\ge \underline t_1$ and $x<  x_{\frac{3}{2}\varepsilon}(t)$, we achieve 
\begin{equation}\label{eplbp1}
\frac{\partial \Phi}{\partial t}-\mathcal{D}[\Phi]-f(\Phi)\le -\mathcal{D}[\Phi]-\frac{r\varepsilon}{(1-\ln \varepsilon)^\alpha}(1-K\varepsilon)\le 0.
\end{equation}
\medskip
\par 
\noindent
\textbf{Zone 2:} $x\in\left[ x_{\frac{3}{2}\varepsilon}(t) ,x_{\frac{\varepsilon}{4}}(t)\right].$
In view of part II of Lemma \ref{eplbde}, by \eqref{eslbephi} and taking $B_2:= x(1+2s\ln  x)^\frac{\alpha}{2s}+1$, for all $t\ge \tilde t$ and $ x_{\frac{3}{2}\varepsilon}(t) \le x\le x_{\frac{\varepsilon}{4}}(t)$, we have
\[
\begin{aligned}
\mathcal{D}[\Phi](t,x)&\ge -\frac{\mathcal{J}_0}{2s}\frac{\varepsilon \frac{1}{ x^{2s}}}{(1+2s\ln  x)^\alpha}-C_\varepsilon\left(\mathcal{J}_1+\mathcal{J}_0\int_1^{B_2} y^{1-2s}dy\right)\left(\frac{\partial \varphi}{\partial x}\right)^2\left(t, x_{\frac{3}{2}\varepsilon}(t)\right)\\
&\ge -\frac{\mathcal{J}_0}{2s} \frac{\varepsilon\varphi(t,x)}{(1-\ln\varphi(t,x))^\alpha}-C_\varepsilon\left(\mathcal{J}_1+\mathcal{J}_0\int_1^{B_2} y^{1-2s}dy\right)\left(\frac{\partial \varphi}{\partial x}\right)^2\left(t, x_{\frac{3}{2}\varepsilon}(t)\right).
\end{aligned}
\]
By \eqref{eslbpx0}, we have uniform convergence of $\left(\frac{\partial \varphi}{\partial x}\right)^2\left(t, x_{\frac{3}{2}\varepsilon}(t)\right)$ to $0$ as $t$ goes to infinity, and thus, for $ x_{\frac{3}{2}\varepsilon}(t) \le x\le x_{\frac{\varepsilon}{4}}(t)$, there is a time $t_1$ such that for all $t\ge t_1$, 
\[
\begin{aligned}
&C_\varepsilon\left(\mathcal{J}_1+\mathcal{J}_0\int_1^{B_2} y^{1-2s}dy\right)\left(\frac{\partial \varphi}{\partial x}\right)^2\left(t, x_{\frac{3}{2}\varepsilon}(t)\right)\le \frac{\varepsilon^2}{(1-\ln \varepsilon)^\alpha}.
\end{aligned}
\]
It follows from $\frac{\varepsilon}{4} \le\varphi\le \frac{3}{2}\varepsilon$ for all $ x_{\frac{3}{2}\varepsilon}(t) \le x\le x_{\frac{\varepsilon}{4}}(t)$ that 
\[
\frac{\varepsilon^2}{(1-\ln \varepsilon)^\alpha}\le \mathcal{C}_1\frac{\varphi^2}{(1-\ln\varphi)^\alpha},
\]
for some $\mathcal{C}_1>0$.
As a result, since $\Phi\ge \frac{\varepsilon}{4} \ge \frac{\varphi}{6} $ for all $ x_{\frac{3}{2}\varepsilon}(t) \le x\le x_{\frac{\varepsilon}{4}}(t)$, we achieve, for all $t\ge\underline t_1:=\max\{t_1,\tilde t\}$ and $x\in\left[ x_{\frac{3}{2}\varepsilon}(t) ,x_{\frac{\varepsilon}{4}}(t)\right]$,
\[
\mathcal{D}[\Phi](t,x)\ge -\left(\frac{2\mathcal{J}_0}{s}+\mathcal{C}_1\right)\frac{\varphi^2}{(1-\ln\varphi)^\alpha}\ge -r\mathcal{C}_2\frac{\Phi\varphi}{(1-\ln\Phi)^\alpha},
\]
for some $\mathcal{C}_2>0$.
\par
In view of Hypothesis \ref{fu}, since $0<\Phi\le \varepsilon\le \xi_0$ and $\Phi\le \varphi$ for any $x>  x_{\frac{3}{2}\varepsilon}(t) $, we have for all $x\in\left[ x_{\frac{3}{2}\varepsilon}(t) ,x_{\frac{\varepsilon}{4}}(t)\right]$,
\[
f(\Phi)\ge \frac{r\Phi}{(1-\ln \Phi)^\alpha}(1-K\varphi) .
\]
By \eqref{nb} and \eqref{eslbPhit}, we have
\[
\frac{\partial \Phi}{\partial t}=rg_\varepsilon'(\varphi)\frac{\varphi}{(1-\ln \varphi)^\alpha}\left(1 -\tilde \gamma_1 \kappa \varphi^p\right)
\le \frac{r\Phi}{(1-\ln \Phi)^\alpha}\left(1 -\tilde \gamma_1 \kappa \varphi^p\right)
\]
Therefore, we have, for all $t\ge\underline t_1$ and $x\in\left[ x_{\frac{3}{2}\varepsilon}(t) ,x_{\frac{\varepsilon}{4}}(t)\right]$,
\begin{equation}
\label{eplbp2}
\begin{aligned}
\frac{\partial \Phi}{\partial t}-\mathcal{D}[\Phi]-f(\Phi)
\le& -\frac{r\Phi}{(1-\ln \Phi)^\alpha}\left(\tilde \gamma_1 \kappa -\left(\frac{3}{2}\right)^\frac{1}{2}(K+\mathcal{C}_2)\varepsilon^\frac{1}{2}\right)\varphi^\frac{1}{2}\le 0,
\end{aligned}
\end{equation}
as long as $\varepsilon\le \tilde\varepsilon:=\frac{2}{3}\left(\frac{\tilde\gamma_1\kappa}{K+\mathcal{C}_2}\right)^2$.
\medskip
\par\noindent
\textbf{Zone 3:} $x\in\left( x_{\frac{\varepsilon}{4}}(t),+\infty\right).$ Let us show that $\Phi$ is a sub-solution for all $x> x_{\frac{\varepsilon}{4}}(t)$.
  In view of Hypothesis \ref{fu}, since $\Phi=\varphi$  for all $x> x_{\frac{\varepsilon}{4}}(t)$,  we have, for all $t\ge \tilde t$ and $x> x_{\frac{\varepsilon}{4}}(t)$,
\[
f(\Phi)\ge \frac{r\varphi}{(1-\ln \varphi)^\alpha}(1-K\varphi).
\]
\par\#
\textbf{Case 1: $s>\frac{1}{2}$.}
In this case, $p=\frac{1}{2}\min\left\{\frac{1}{2s},2s\right\}=\frac{1}{4s}$.
In view of part III of Lemma \ref{eplbde}, by \eqref{eslbphix0} and letting $\varepsilon<e^{\frac{1}{2s}-\alpha}$,  we have for all $t\ge \tilde t$ and $x> x_{\frac{\varepsilon}{4}}(t)$,
   \[
    \mathcal{D}[\Phi]\ge  \frac{\mathcal{J}_0}{2s-1}\frac{\partial \varphi}{\partial x}(t,x)\ge  -\frac{2s\mathcal{J}_0}{2s-1}\varphi^{1+\frac{1}{2s}}.
\]
Since $y\mapsto y^\frac{1}{4s}(1-\ln y)^\alpha$ is non-decreasing for all $y\in\left(0,e^{1-4s\alpha}\right)$ and $0<\varphi< \varepsilon$ for all $x> x_{\frac{\varepsilon}{4}}(t)$, by \eqref{eslbPhit} and letting $\varepsilon< \min\left\{e^{\frac{1}{2s}-\alpha},e^{1-4s\alpha},\xi_0\right\}$, we obtain for all $t\ge \tilde t$ and $x> x_{\frac{\varepsilon}{4}}(t)$,
\[
\begin{aligned}
    \frac{\partial \Phi}{\partial t}-\mathcal{D}[\Phi]-f(\Phi)\le& \frac{\varphi^{1+\frac{1}{4s}}}{(1-\ln \varphi)^\alpha}\left\{-r\tilde \gamma_1 \kappa +\mathcal{J}_0\frac{2s}{2s-1}\varepsilon^{\frac{1}{4s}}(1-\ln\varepsilon)^\alpha+rK\varepsilon^{1-\frac{1}{4s}}\right\}.
    \end{aligned}
\]
Notice that $\mathcal{J}_0\frac{2s}{2s-1}\varepsilon^{\frac{1}{4s}}(1-\ln\varepsilon)^\alpha+ rK\varepsilon^{1-\frac{1}{4s}}$ goes to $0$ as $\varepsilon$ goes to $0^+$. Therefore, there is $\bar\varepsilon_1\in\left(0, \min\left\{e^{1-4s\alpha},e^{\frac{1}{2s}-\alpha},\xi_0\right\}\right)$ such that for any $\varepsilon\le \bar \varepsilon_1$, we have
\[
\frac{\partial \Phi}{\partial t}-\mathcal{D}[\Phi]-f(\Phi)\le 0\quad \text{ for all }  t\ge\tilde t\text{ and }x> x_{\frac{\varepsilon}{4}}(t).
\]
\par\#
\textbf{Case 2: $s=\frac{1}{2}$.}
By the definition of $p$, one may notice $p=\frac{1}{2}\min\left\{\frac{1}{2s},2s\right\}=\frac{1}{2}$. In view of \eqref{eslbphix0}, by letting $\varepsilon\le e^{\frac{1}{2s}-\alpha}$, we have
\[
\frac{\partial \varphi}{\partial x}\ge -\varphi^2 \ge -\varepsilon^2\quad  \text{for all } t\ge \tilde t \text{ and } x> x_{\frac{\varepsilon}{4}}(t).
\]
In view of part III of Lemma \ref{eplbde}, since the map $y\mapsto y(\ln\varepsilon-\ln y)$ is non-decreasing for $y\in\left(0, \frac{\varepsilon}{e}\right)$, by letting $\varepsilon\le \min\left\{e^{-1}, e^{\frac{1}{2s}-\alpha}\right\}$ and taking $B_3=\varepsilon\left(-\frac{\partial \varphi}{\partial x}\right)^{-1} \left[\ln \varepsilon-\ln \left(-\frac{\partial \varphi}{\partial x}\right)\right]^{-1} $, we get 
\[
  \mathcal{D}[\Phi]\ge  -\mathcal{J}_0\frac{\varepsilon}{B_3}+\mathcal{J}_0\frac{\partial \varphi}{\partial x}(t,x)\ln B_3 \ge 2 \mathcal{J}_0\frac{\partial \varphi}{\partial x}\left[\ln \varepsilon-\ln \left(-\frac{\partial \varphi}{\partial x}\right)\right]\ge  -2\mathcal{J}_0 \varphi^2 (\ln \varepsilon-2\ln \varphi).
\]
Hence, by using \eqref{eslbPhit}, since $y\mapsto y^\frac{1}{2}(1-2\ln y)^{\alpha+1}$ is non-decreasing for all $y\in\left(0,e^{-\frac{3}{2}-2\alpha}\right)$  and $0<\varphi< \varepsilon$ for all $x> x_{\frac{\varepsilon}{4}}(t)$, we arrive at, for any $0<\varepsilon\le  \min\left\{e^{-1}, e^{\frac{1}{2s}-\alpha},e^{-\frac{3}{2}-2\alpha} \right\}$,
\[
\begin{aligned}
\frac{\partial \Phi}{\partial t}-\mathcal{D}[\Phi]-f(\Phi)
\le& \frac{\varphi^\frac{3}{2}}{(1-\ln \varphi)^\alpha}\left[-r\tilde \gamma_1\kappa  +2\mathcal{J}_0\varphi^\frac{1}{2}(\ln \varepsilon-2\ln\varphi)^{\alpha+1}+rK\varphi^\frac{1}{2}\right]\\
\le &  -\frac{\varphi^{\frac{3}{2}}}{(1-\ln \varphi)^\alpha}\left(r\kappa \tilde \gamma_1 -2\mathcal{J}_0 \varepsilon^\frac{1}{2}\lvert\ln \varepsilon\rvert^{\alpha+1}-rK\varepsilon^{\frac{1}{2}}\right).
\end{aligned}
\]
It therefore follows from the fact that $2\mathcal{J}_0 \varepsilon^\frac{1}{2}\lvert\ln \varepsilon\rvert^{\alpha+1}+rK\varepsilon^{\frac{1}{2}}$ goes to $0$ as $\varepsilon$ goes to $0^+$ that there is $\bar\varepsilon_2\in\left(0, \min\left\{ e^{-1},e^{-\frac{1}{2}-\alpha},\xi_0\right\}\right)$ such that for all $0<\varepsilon\le \bar\varepsilon_2$, we have
\[
\frac{\partial \Phi}{\partial t}-\mathcal{D}[\Phi]-f(\Phi)\le 0\quad \text{ for all }t\ge\tilde t \text{ and }x\ge x_{\frac{\varepsilon}{4}}(t).
\]
\par\#
\textbf{Case 3: $s<\frac{1}{2}$.}
In this case, $p=\frac{1}{2}\min\left\{\frac{1}{2s},2s\right\}=s$. In view of part III of Lemma \ref{eplbde}, by using \eqref{eslbphix0}, taking $B_3=\varepsilon\left(-\frac{\partial \varphi}{\partial x}\right)^{-1}>1$ and letting $\varepsilon\le e^{\frac{1}{2s}-\alpha}$, we have for all $t\ge \tilde t$ and $x> x_{\frac{\varepsilon}{4}}(t)$,
\[
\begin{aligned}
  \mathcal{D}[\Phi]\ge  -\frac{\mathcal{J}_0}{2s}\frac{\varepsilon}{B_3^{2s}}+\frac{1}{1-2s}\mathcal{J}_0\frac{\partial \varphi}{\partial x}(t,x)B_3^{1-2s}
  &=- (2s)^{-2s}\mathcal{C}_3\varepsilon^{1-2s}\left(-\frac{\partial \varphi}{\partial x}\right)^{2s}\ge-\mathcal{C}_3\varepsilon^{1-2s}\varphi^{2s+1},
\end{aligned}
\]
for some positive constant $\mathcal{C}_3$.
Thus, since $y\mapsto y^{s}(1-\ln y)^{\alpha}$ is non-decreasing for all $y\in \left(0,e^{1-\frac{\alpha}{s}}\right)$  and $0<\varphi< \varepsilon$ for all $x> x_{\frac{\varepsilon}{4}}(t)$, by letting $\varepsilon\le \min\left\{e^{\frac{1}{2s}-\alpha}, e^{1-\frac{\alpha}{s}},\xi_0\right\}$, we arrive at for all $t\ge \tilde t$ and $x> x_{\frac{\varepsilon}{4}}(t)$,
\[
\begin{aligned}
\frac{\partial \Phi}{\partial t}-\mathcal{D}[\Phi]-f(\Phi)
\le& \frac{\varphi^{1+s}}{(1-\ln \varphi)^\alpha}\left[-r\tilde \gamma_1\kappa +\mathcal{C}_3\varepsilon ^{1-2s}\varphi^{s}(1-\ln \varphi)^\alpha+rK\varphi^{1-s}\right]\\
\le &  -\frac{\varphi^{1+s}}{(1-\ln \varphi)^\alpha}\left(r\kappa \tilde \gamma_1 -\mathcal{C}_3 \varepsilon^{1-s}(1-\ln \varepsilon)^\alpha-rK\varepsilon^{1-s}\right) .
\end{aligned}
\]
It therefore follows from $\mathcal{C}_3\varepsilon^{1-s}(1-\ln \varepsilon)^\alpha-rK\varepsilon^{1-s}\to 0$ as $\varepsilon\to 0^+$ that there is $\bar\varepsilon_3\in\left(0, \min\left\{e^{\frac{1}{2s}-\alpha}, e^{1-\frac{\alpha}{s}},\xi_0\right\}\right)$, for any $0<\varepsilon\le \bar\varepsilon_3$,
\[
\frac{\partial \Phi}{\partial t}-\mathcal{D}[\Phi]-f(\Phi)\le 0\quad \text{ for all } t\ge\tilde t\text{ and } x\ge x_{\frac{\varepsilon}{4}}(t).
\]
\par 
Let $\bar \varepsilon:=\min \{\bar\varepsilon_1,\bar\varepsilon_2,\bar\varepsilon_3\}$. As a result, collecting results in three cases, for any $0<\varepsilon\le \bar\varepsilon$, we have
\begin{equation}
\label{eplbp3}
\frac{\partial \Phi}{\partial t}-\mathcal{D}[\Phi]-f(\Phi)\le 0\quad \text{ for all } t\ge\tilde t\text{ and } x\ge x_{\frac{\varepsilon}{4}}(t).
\end{equation}

Therefore, combining \eqref{eplbp1}, \eqref{eplbp2} and \eqref{eplbp3}, by letting $\underline t:=\max\{\underline t_1,\underline t_2, \tilde t\}$ and $\varepsilon^*:=\min\{\tilde \varepsilon,\bar \varepsilon\}$, we conclude for any $\varepsilon\in(0,\varepsilon^*)$, 
 $\Phi$ is a sub-solution for all $t\ge\underline t$ and  $x\in\mathbb{R}$.

\end{proof}

\subsubsection{Proof of the lower bound in Theorem \ref{al}}

To get the lower bound on the level set $E_\lambda(t)$ using the sub-solution $\Phi$, we only need to find two times $T_1$ and $T_2$ and a position $R^\star$, such that 
\begin{equation} \label{eslbclaim}
u(T_1,x-R^\star)\ge \Phi(T_2,x) \quad \text{for all }x\in\mathbb{R}.
\end{equation}
If so, by using the parabolic comparison principle, we can achieve 
\[
u(t+T_1,x-R^\star)\ge \Phi(t+T_2,x) \quad \text{for all }t>0\text{ and } x\in\mathbb{R}.
\]
It follows from the definition of $\Phi$ that 
\[
u(t+T_1,x-R^\star)\ge \varepsilon\quad \text{for all }t>0 \text{ and }x\le x_{\frac{3}{2}\varepsilon}(t+T_2).
\]
 Thus, for a small height $\lambda\in(0,\varepsilon)$, by Lemma \ref{eslbl2}, there is a time $T_\lambda>0$ such that, for all $t\ge T_\lambda$,
 \[
 X_\lambda(t)\ge x_{\frac{3}{2}\varepsilon}(t+T_2-T_1)-R^\star\gtrsim e^{\frac{1}{2s}[r(\alpha+1)t]^\frac{1}{\alpha+1}}.
 \]
 Meanwhile, by using the argument in subsubsection  \ref{sslb}, for all $\lambda\in(0,1)$, we obtain
 \[
 X_\lambda(t)\gtrsim e^{\frac{1}{2s}[r(\alpha+1)t]^\frac{1}{\alpha+1}}\quad \text{for } t \text{ large enough}.
 \]
 \par 
 To complete this step, we need to proof the remaining claim \eqref{eslbclaim}. The comparison principle implies that $u(t,x)\ge \underline u(t,x)$, where $\underline u(t,x)$ is given by
\begin{equation}\label{eslble1}
\begin{cases}
   \frac{\partial \underline u}{\partial t}(t,x)=\mathcal{D}[\underline u](t,x), &t>0,\ x\in\mathbb{R},\\
    \underline u(0,x)= u_0(x), &x\in\mathbb{R}.
\end{cases}
\end{equation}

\begin{lemma}[\cite{bouin2023simple}]
Assume that the kernel $J$ and the initial data $u_0$ satisfy Hypothesis \ref{ker} and Hypothesis \ref{ic1}, respectively.
   Then, there exists a constant $\Gamma>0$, such that the solution to \eqref{eslble1} has the following asymptotic flattening behaviour at infinity:
    \[
        \forall t\in(0,+\infty), \quad \lim_{x\to +\infty} x^{2s}\underline u(t,x)\ge \Gamma t.
    \]
\end{lemma}
\par 

By the above lemma, for some $T_1>\frac{1}{\Gamma}e^{[r(\alpha+1)\underline t]^\frac{1}{\alpha+1}}$, there is $\xi_1>1$ such that for all $x\ge \xi_1$, we have $u(T_1,x)\ge \underline u(T_1,x)\ge \Gamma\frac{T_1}{x^{2s}}$. 
In view of the definition of $\Phi$, by \eqref{eslbephi}, denoting $T_2=\underline t$, we have for all $x\ge \xi_1$,
\[
\Phi(T_2,x)\le \frac{e^{[r(\alpha+1)T_2]^\frac{1}{\alpha+1}}}{x^{2s}}\le \Gamma\frac{T_1}{x^{2s}}\le u(T_1,x).
\]
On the other hand,  since propagation for Allee effect occurs (see \cite{alfaro2017propagation}), the comparison implies that there is $\xi_2$ such that $u(t,\xi_2)\ge \varepsilon$, then, since $u(t,\cdot)$ is non-increasing, we have
\[
u(t,x)\ge \Phi(T_2,x)\quad\text{for all }t>0 \text{ and } x\le \xi_2.
\]
As a result, by taking $R^\star=|\xi_1-\xi_2|$, we achieve
\[
u(T_1,x-R^\star)\ge \Phi(T_2,x)\quad\text{for all } x\in\mathbb{R}.
\]

\section{A numerical method based on an IMEX scheme}\label{s5}
In this section, we present an efficient numerical method based on Implicit-explicit scheme to get approximations to the singular integro-differential equation to \eqref{oeq1}. It is worth mentioning that the combination of non-local operators and accelerated propagation makes the numerical approximation very difficult and costly.
\par
To do so, we truncate the real line to work on a bounded interval $\Omega:=[0,L]$ for some $L>0$ and impose an exterior boundary condition
\[
    g(x)=\left\{
    \begin{aligned}
        &1, &x< 0,\\
        &0, &x> L.
    \end{aligned}\right.
\]
In other words, we consider the solution to the following initial value problem with Dirichlet boundary conditions, on a finite time interval $[0,T]$, for some $T>0$:
\begin{equation}\label{oeq2}
\begin{cases}
    \frac{\partial u}{\partial t}(t,x)-\mathcal{D}[u](t,x)-f(u(t,x))=0, &t\in(0,T),\ x\in\Omega,\\
u(t,x)= g(x), &t\in(0,T),\ x\in\mathbb{R}\setminus\Omega,\\
   u(0,x)=u_0(x), &x\in \Omega.
\end{cases}
\end{equation}
\subsection{Discretisation of the singular integral operator \texorpdfstring{$\mathcal{D}[\cdot]$}{D[.]}}\label{5.1}
Let us explain how we discretise the singular integral operator $\mathcal{D}[\cdot]$ under Hypothesis \ref{ker}. We start from
\[
    \mathcal{D}[u](x)=\int_{\mathbb{R}^+}\Big(u(x-y)+u(x+y)-2u(x)\Big)J(y)dy,
\]
that is obtained using the symmetry of $J$.
Without loss of generality, we assume that the closure of the interval $\Omega$ belongs to $\mathbb{R}^+$ and contains the singular point of the kernel dispersal $J$.  
The integral operator $\mathcal{D}[\cdot]$ involves a singular integrand, which is typically handled by decomposing it into a singular part and a smooth part, allowing standard quadrature methods to be applied. This strategy has been applied to both convolution (see \cite{tian2013analysis}) and fractional Laplace operators (see \cite{duo2018novel,huang2014numerical,minden2020simple,bouin2021sharp}). Inspired by these works,
we write, for $x\in \Omega$,
\begin{align*}
   \mathcal{D}[u](x)&=\int_\Omega \frac{u(x-y)+u(x+y)-2u(x)}{y^\gamma}y^\gamma J(y)dy+\int_{\mathbb{R}^+\setminus \Omega}\Big(u(x-y)+u(x+y)-2u(x)\Big)J(y)dy\\
   &:=\mathcal{D}_1[u](x)+\mathcal{D}_2[u](x),
\end{align*}
where the parameter $\gamma$ satisfies the following condition
\begin{equation}\label{necon}
    \gamma\in [0,2]\quad\text{and}\quad y\mapsto y^\gamma J(y)\in L^1(\Omega).
\end{equation}
\par
For example, if $J$ belongs to $L^1(\mathbb{R})$, then $\mathcal{D}[u]$ is the convolution operator $J*u-u$ with an integrable kernel. In this case, any $\gamma \in[0,2]$ fits. One may notice that if one takes $J(y)=\frac{C_{1,s}}{|y|^{1+2s}}$ with some $s\in(0,1)$ and $C_{1,s}=\frac{2^{2s}\Gamma\left(\frac{1+2s}{2}\right)}{\pi^\frac{1}{2}\Gamma(1-s)}$, the operator $-\mathcal{D}[u]$ is the fractional Laplace operator
\[\left(-\triangle \right)^{s}u(x)=C_{1,s} P.V. \int_{\mathbb{R}}\frac{u(x)-u(y)}{|x-y|^{1+2s}}dy,\] 
and the condition \eqref{necon} becomes $\gamma\in(2s,2].$
Moreover, if $J$ satisfies Hypothesis \ref{ker}, then $\gamma=2$ is the only left choice.
\par 
To approximate $\mathcal{D}_1[u]$ on $\Omega$, assuming $u$ is $C^2(\mathbb{R})$, for $x\in \Omega$, we define 
\[
  U_\gamma(x,y)=\frac{u(x-y)+u(x+y)-2u(x)}{y^\gamma}  \quad \text{for all }y\neq 0,
\]
and 
\[
U_\gamma(x,0)=
\begin{cases}
 0, &\gamma \neq2,\\
 u''(x), &\gamma =2.
\end{cases}
\]
The function $U(x,\cdot)$ belongs to $C(\mathbb{R})$, since
\begin{equation}
 \label{nml}   
\lim_{y\to 0^+}U_\gamma(x,y)=u''(x)\lim_{y\to 0^+} y^{2-\gamma}=U_\gamma(x,0),
\end{equation}
using the L'H\^opital's rule.
For some positive integer $N$, let $\{y_n\}_{n=0,1,\dots,N}$ be a partition of $\Omega$ such that $y_0<y_1<\cdots <y_{N-1}<y_N$ where $y_0$ and $y_N$ are the endpoints of $\Omega$.
Then, denoting $I_n:=[y_{n-1},y_n]$ and $\T^\gamma_n:=\int_{I_n}y^\gamma J(y)dy$ for $n\in \{1,\dots, N\}$, we split $\mathcal{D}_1[u]$ into $N$ parts:
\[
     \mathcal{D}_1[u](x)  =\int_\Omega U_\gamma(x,y) y^\gamma J(y)dy
    =\sum_{n=1}^N \int_{I_n}U_\gamma(x,y)y^\gamma J(y)dy.
\]
To help readability, we denote by $\T^\gamma_n:=\int_{I_n}y^\gamma J(y)dy$ for $n\in \{1,\dots, N\}$. On the one hand, for $n\in\{2,3,\dots,N\}$, by using the weighted trapezoidal rule, we obtain
\[
\int_{I_n}U_\gamma(x,y)y^\gamma J(y)dy=\frac{\T^\gamma_n}{2}[U_\gamma(x,y_n)+U_\gamma(x,y_{n-1})].
\]
On the other hand, for $n=1$, since this part contains the singular integral, we apply \eqref{nml} and the weighted trapezoidal rule to obtain 
\[
\int_{I_1}U_\gamma(x,y)y^\gamma J(y)dy=\frac{\T^\gamma_1}{2}[U_\gamma(x,y_1)+U_\gamma(x,0)].
\]
As a result, the numerical approximation of $\mathcal{D}_2[u]$ on $\Omega$ is given by:
\begin{equation}\label{nmd1}
   \mathcal{D}_1[u](x)= \sum_{n=1}^N \frac{\T^\gamma_n}{2}[U_\gamma(x,y_n)+U_\gamma(x,y_{n-1})].
\end{equation}
\par 
Notice that for all $x\in\Omega$, if $y$ belongs to $\mathbb{R}^+\setminus \Omega$, then $x\pm y\in \mathbb{R}\setminus \Omega$. Thus, by use the boundary condition $g$, the numerical approximation of $\mathcal{D}_2[u]$ on $\Omega$ can be written as
\begin{equation}\label{nmd2}
    \mathcal{D}_2[u](x) =\int_{\mathbb{R}^+\setminus \Omega}\Big(g(x-y)+g(x+y)\Big)J(y)dy-2u(x)\int_{\mathbb{R}^+\setminus \Omega}J(y)dy=(1-2u(x))\int_{\mathbb{R}^+\setminus \Omega}J(y)dy.
\end{equation}
Therefore, collecting \eqref{nmd1} and \eqref{nmd2}, we achieve the numerical approximation of the singular integral operator $\mathcal{D}[\cdot]$ as, for all $x\in\Omega$,
\begin{equation}\label{dD}
    \mathcal{D}[u](x) \approx \sum_{n=1}^N \frac{\T^\gamma_n}{2}[U_\gamma(x,y_n)+U_\gamma(x,y_{n-1})]+(1-2u(x))\int_{\mathbb{R}^+\setminus \Omega}J(y)dy.
\end{equation}
In practical applications, the integral $\int_{\mathbb{R}^+\setminus \Omega}J(y)dy$
 can be achieved either through straightforward analytical calculations or by taking advantage of available software packages. For example, the \texttt{quad} function from the \texttt{scipy.integrate} library in Python can be used for numerical integration.
\subsection{Toeplitz equation}
 In this subsection, we explain the discretisation of equation \eqref{oeq2}.  More precisely, we use finite differences in space and an IMEX scheme in time. The diffusion term $\mathcal{D}[\cdot]$ is dealt implicitly, while the reaction term $f$ is integrated explicitly. 
 \par
We work on the time interval $[0,T]$ and the space $\Omega:=[0,L]$.
We use a uniform Cartesian grid $\{(t_i, x_j)=(i\T t, j\T x)|(i,j)\in\mathbb{N}\times \mathbb{Z}\}$\footnote{In this paper, we use $\mathbb{N}$ to denote the set of non-negative integers, i.e., \{0,1,2,3,\dots\}.}, with the time step $\T t=\frac{T}{M}$ and the space step $\T x=\frac{L}{N}$ for some positive integers $M$ and $N$.
Some notations should be introduced:
\[
u_{i,j}=u(t_i, x_j),\quad \mathcal{D}_{i,j}=\mathcal{D}[u](t_i, x_j)\quad \text{and}\quad f_{i,j}=f(u_{i,j}).
\]
  Denote $y_n:=n\T x$ with $n\in \mathbb{Z}$ and $\T^\gamma_n:=\int_{y_{n-1}}^{y_n}y^\gamma J(y)dy$ for $n\in \{1,\dots, N\}$. In view of \eqref{dD}, if $x_j\in\Omega=[0,L]$ and $y_n\in(L,+\infty)$, then $x_j- y_n\in (-\infty,0)$ and $x_j+y_n\in(L,+\infty)$, whence by the extended boundary value, for all $i\in\mathbb{Z}^+$,
\[
    u_{i,j}=1\quad \text{for }j<0\quad \text{and}\quad u_{i,j}=0 \quad \text{for }j>N.
\]
Thus, in view of \eqref{dD}, denoting 
\[
U^i_{j,n}=\frac{u_{i,j-n}+u_{i,j+n}-2u_{i,j}}{(n\T x)^\gamma}\quad \text{for }n\neq 0,
\]
and 
\[
U^i_{j,0}=
\begin{cases}
    0, &\gamma\neq 2,\\
    \frac{u_{i,j-1}+u_{i,j+1}-2u_{i,j}}{(\T x)^2}, &\gamma =2,
\end{cases}
\]
the approximate formula of $\mathcal{D}_{i,j}$ can be written as 
\[
    \mathcal{D}_{i,j} \approx \sum_{n=1}^N \frac{\T^\gamma_n}{2}[U^i_{j,n}+U^i_{j,n-1}]-2\mathcal{J}_L u_{i,j}+\mathcal{J}_L,
\]
where $\mathcal{J}_L:=\int_L^{+\infty}J(y)dy$.
Here, we use the central difference approximation for $u''(x)$ to determine the value of $U^i_{j,0}$ when $\gamma=2$, which equals to $U^i_{j,1}$.
\par
Furthermore, by using the method of finite differences, the approximation of the first equation in \eqref{oeq2} is:
 \[
     \frac{u_{i+1,j}-u_{i,j}}{\T t}-\mathcal{D}_{i+1,j}=f_{i,j}\quad \text{for each }(i,j)\in \{0,1,\dots,M-1\}\times\{0,1,\dots,N\}.
 \]
 Rearranging the above discretized equation, we get
 \[
     u_{i+1,j}-\T t \mathcal{D}_{i+1,j}=u_{i,j}+\T t f_{i,j},
 \]
which is, in matrix form,
\begin{equation}\label{ts}
   (\mathcal{A}_{i+1}+\mathcal{B}(\gamma) )X_{i+1}=\mathcal{E}_i+\mathcal{F}(\gamma),
\end{equation}
where $X_{i}=(u_{i,0},u_{i,1},\dots,u_{i,N})^T$.
Using $E_{N+1}$ to represent an $(N+1)$-dimensional identity matrix, the matrixes $\mathcal{A}_{i+1}$ and $\mathcal{B}(\gamma)$ are given by
\[
\resizebox{\textwidth}{!}{$
\left(1+2\T t \mathcal{J}_L\right)E_{N+1}-\T t
       \begin{pmatrix}
        &-\frac{\T^\gamma_1}{(\T x)^\gamma}-\sum_{n=2}^N\T^\gamma_n\left[\frac{1}{(n\T x)^\gamma}+\frac{1}{((n-1)\T x)^\gamma}\right] &\frac{\T^\gamma_1+\T^\gamma_2}{2(\T x)^\gamma} &\dots & \frac{\T^\gamma_N+\T^\gamma_{N+1}}{2(N\T x)^\gamma}\\
        &\frac{\T^\gamma_1+\T^\gamma_2}{2(\T x)^\gamma}&-\frac{\T^\gamma_1}{(\T x)^\gamma}-\sum_{n=2}^N \T^\gamma_n\left[\frac{1}{(n\T x)^\gamma}+\frac{1}{((n-1)\T x)^\gamma}\right]&\dots  &\frac{\T^\gamma_{N-1}+\T^\gamma_{N}}{2((N-1)\T x)^\gamma}\\
        &\vdots &\vdots &\ddots   &\vdots\\
     & \frac{\T^\gamma_N+\T^\gamma_{N+1}}{2(N\T x)^\gamma}&  \frac{\T^\gamma_{N-1}+\T^\gamma_N}{2((N-1)\T x)^\gamma} & \dots &-\frac{\T^\gamma_1}{(\T x)^\gamma}-\sum_{n=2}^N\T^\gamma_n\left[\frac{1}{(n\T x)^\gamma}+\frac{1}{((n-1)\T x)^\gamma}\right]
    \end{pmatrix}
    $},
\]
 and 
\[
\quad \mathcal{B}(\gamma)=0_{(N+1)\times(N+1)}\quad \text{for }\gamma \neq 2,\quad
\mathcal{B}(2)=\frac{\T t\T _1^\gamma}{(\T x)^2}
\begin{pmatrix}
   & 1 &-\frac{1}{2} &  & & \\
   & -\frac{1}{2}   &1 &\ddots &  &\\
   & &  \ddots &  \ddots &\ddots & \\
& &   & \ddots &1 &-\frac{1}{2} \\
  & &   &   &-\frac{1}{2} &1 \\
\end{pmatrix}
,
\]
where $\T^\gamma_n=\int_{y_{n-1}}^{y_n}y^\gamma J(y)dy$ with $y_n=n\T x$.
Denoting $\T^\gamma_{N+1}=0$, $F_i=(f_{i,0},f_{i,1},\dots,f_{i,N})^T$ and $e_{N+1}={\underbrace{(1,1,\dots,1)}_{N+1 \ \text{ elements}}}^T$, the vectors $\mathcal{E}_i$ and $\mathcal{F}(\gamma)$ are given by
\[
    \mathcal{E}_i:= X_i+
    \T t F_i+\T t\mathcal{J}_L e_{N+1}+\T t
    \begin{pmatrix}
        &\sum_{n=1}^{N}\frac{\T^\gamma_n+\T^\gamma_{n+1}}{2(n\T x)^\gamma}\\
        &\sum_{n=2}^{N}\frac{\T^\gamma_n+\T^\gamma_{n+1}}{2(n\T x)^\gamma}\\
        &\vdots\\
        &\sum_{n=N}^{N}\frac{\T^\gamma_n+\T^\gamma_{n+1}}{2(n\T x)^\gamma}\\
        &0\\
    \end{pmatrix},
\]
and
\[
\mathcal{F}(\gamma)=0_{(N+1)\times 1}\quad \text{for }\gamma \neq 2,\quad \mathcal{F}(2)=  \left(\frac{\T t\T _1^\gamma }{2(\T x)^2},0,\dots,0\right)^T_{N+1}.
\]
One may observe that since the matrix $\mathcal{A}_{i+1}+\mathcal{B}(\gamma)$ is a Toeplitz matrix, \eqref{ts} is a Toeplitz system. See \cite[Chapter 5.2.6]{hayes1996statistical} for a method of resolution of such systems.
%-------
\subsection{Numerical experiments}
In this subsection, we shall make some numerical experiments to illustrate our previous theoretical results. The influence of the tail of the dispersal kernel $J$ on the rate of propagation is illustrated for some fixed $\alpha$ in Figures \ref{fig:subexp0.2}, \ref{fig:subexp0.4}, \ref{fig:algebraic0.5} and \ref{fig:algebraic1.0}. We consider sub-exponential and algebraic kernels to illustrate our main results.
\par
As the previous discussion about \eqref{necon} in Section \ref{5.1}, under Hypothesis \ref{ker}, the splitting parameter $\gamma$ should be chosen to be equal to $2$. 
In the following simulations, we always assume that $L>1$, $z^2J(z)=1$ for $z\in[-1,1]$, so that $J$ satisfies $\int_{|y|\le 1}y^2J(y)dy=2$, and that the initial data $u_0:=\mathds{1}_{x\le 0}$, and that $f(u)=\frac{u(1-u)}{(1-\ln u)^\alpha}$.
\subsubsection{Sub-exponential kernel}
To illustrate theoretical results given in Theorems \ref{lin1} and \ref{ei1} for a sub-exponential kernel, we take $J(|z|)=e^{1-|z|^\beta}$ with $\beta>0$ for $|z|\ge 1$. Thus, recalling $\T^\gamma_n=\int_{y_{n-1}}^{y_n}y^\gamma J(y)dy$ with $y_n=n\T x$ for $n\in \{1,\dots, N\}$, we have
\[
\T^2_n =\begin{cases}
    \int_{y_{n-1}}^{y_n} y^2 e^{1-y^\beta}dy, &\text{for }\left[\frac{N}{L}\right]+1\le n\le N ,\\
 \T x,& \text{for }1\le n\le \left[\frac{N}{L}\right],
\end{cases}
\]
where the integral $\int_{y_{n-1}}^{y_n} y^2 e^{1-y^\beta}dy$ can be numerically evaluated using the \texttt{quad} function from the \texttt{scipy.integrate} library in Python.
\par 
For this type of dispersal kernel $J$, Figures \ref{fig:subexp0.2} and \ref{fig:subexp0.4} demonstrate that acceleration occurs when $\beta$ is small, and the time range over which this phenomenon can be observed becomes shorter as $\beta$ decreases.
This is consistent with our theoretical results, that is, $X_\lambda(t)\asymp t^\frac{1}{\beta(\alpha+1)}$ increases as $\beta$ decreases. Additionally, the width of the solution grows as time increases when acceleration occurs: it is the flattening effect \cite{flatteningeffect,bouin2021sharp,BOUIN2024113557}.
When $\beta$ is large enough, as shown in Figure \ref{fig:algebraic0.5}(c) and Figure \ref{fig:algebraic1.0}(c), the solution propagates at a finite rate.

\begin{figure}[H]
\centering 
    \begin{subfigure}[b]{0.25\textwidth} 
        \includegraphics[width=\textwidth]{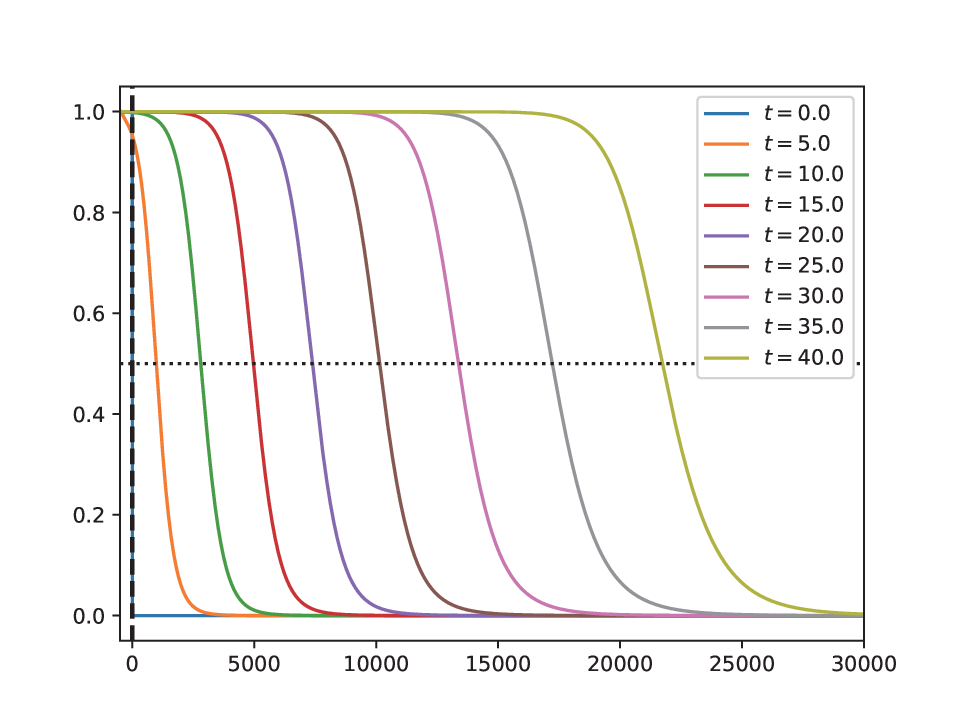} \caption{$\beta=0.35$} 
    \end{subfigure} 
    \quad
    \begin{subfigure}[b]{0.25\textwidth}  
        \includegraphics[width=\textwidth]{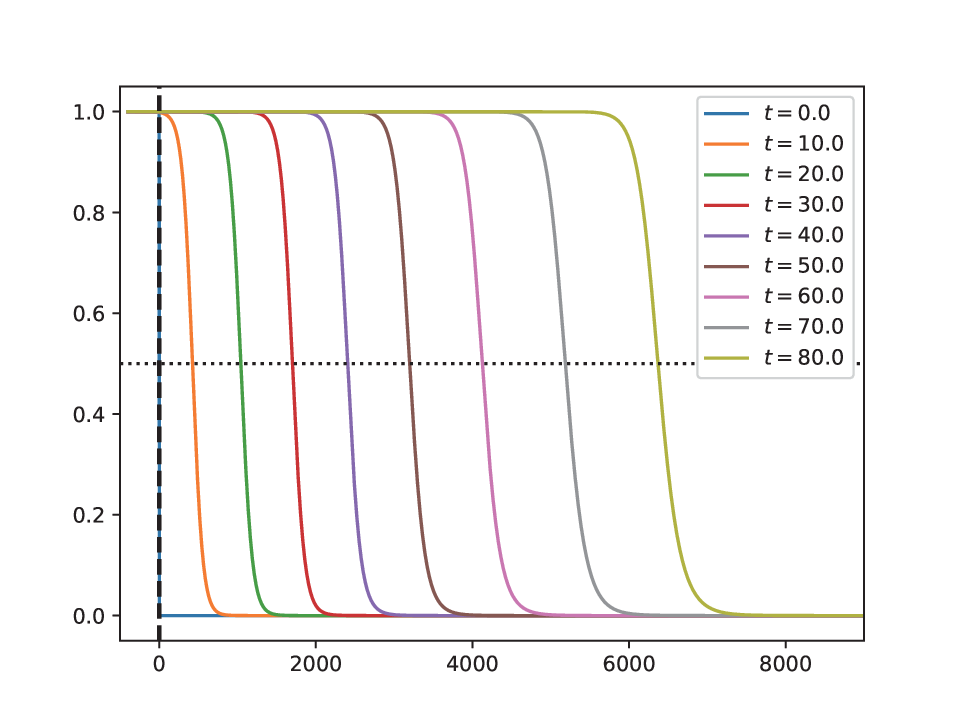}\caption{$\beta=0.45$} 
    \end{subfigure} 
     \quad
    \begin{subfigure}[b]{0.25\textwidth}  
        \includegraphics[width=\textwidth]{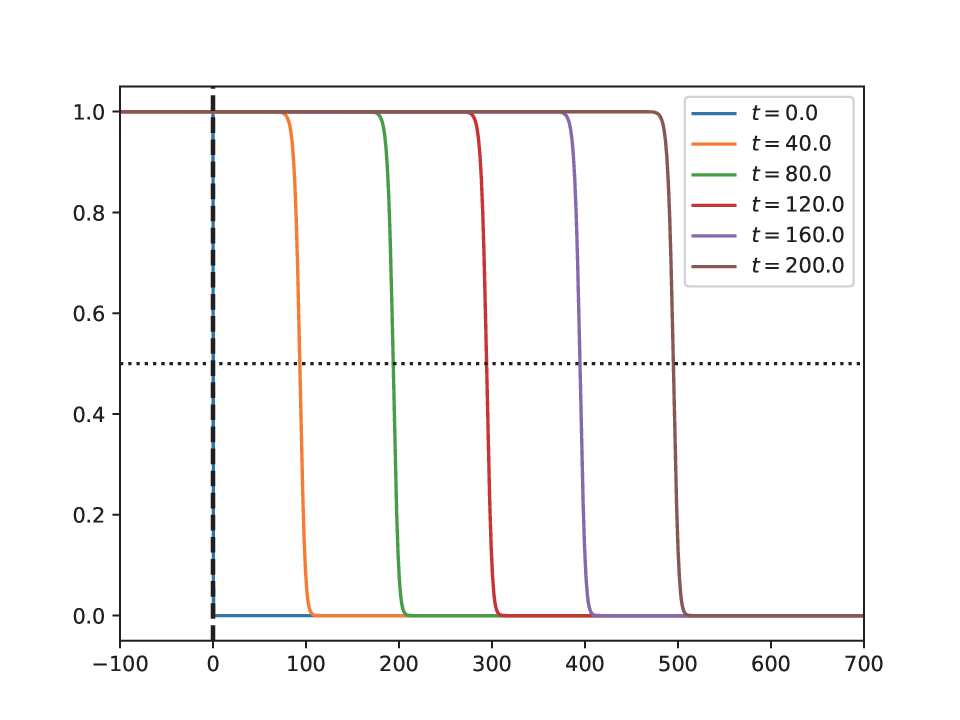} \caption{$\beta=2$}  
    \end{subfigure} 
\caption{ \small Numerical approximations of the solution to \eqref{oeq1} with the dispersal kernel $J(z)\sim e^{-|z|^\beta}$ at different times for $\alpha=0.2$ and different values of $\beta$. The threshold for acceleration is $\beta=\frac{1}{\alpha+1}=\frac{5}{6}$.}
\label{fig:subexp0.2}
\end{figure} 

\begin{figure}[H]
\centering 
    \begin{subfigure}[b]{0.25\textwidth} 
        \includegraphics[width=\textwidth]{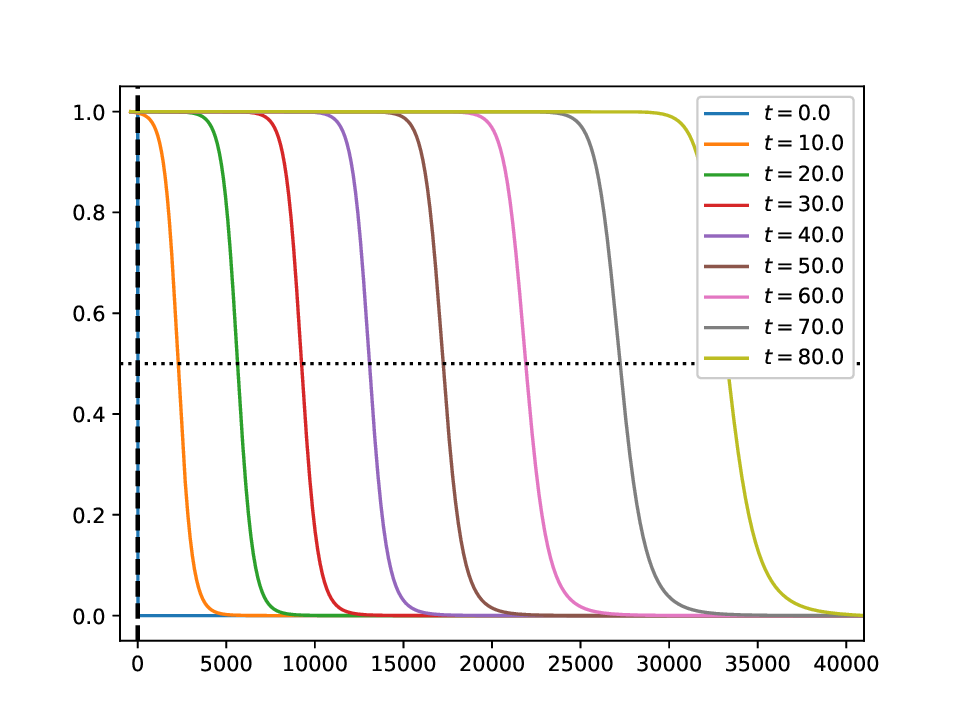} \caption{$\beta=0.35$} 
    \end{subfigure} 
    \quad
    \begin{subfigure}[b]{0.25\textwidth}  
        \includegraphics[width=\textwidth]{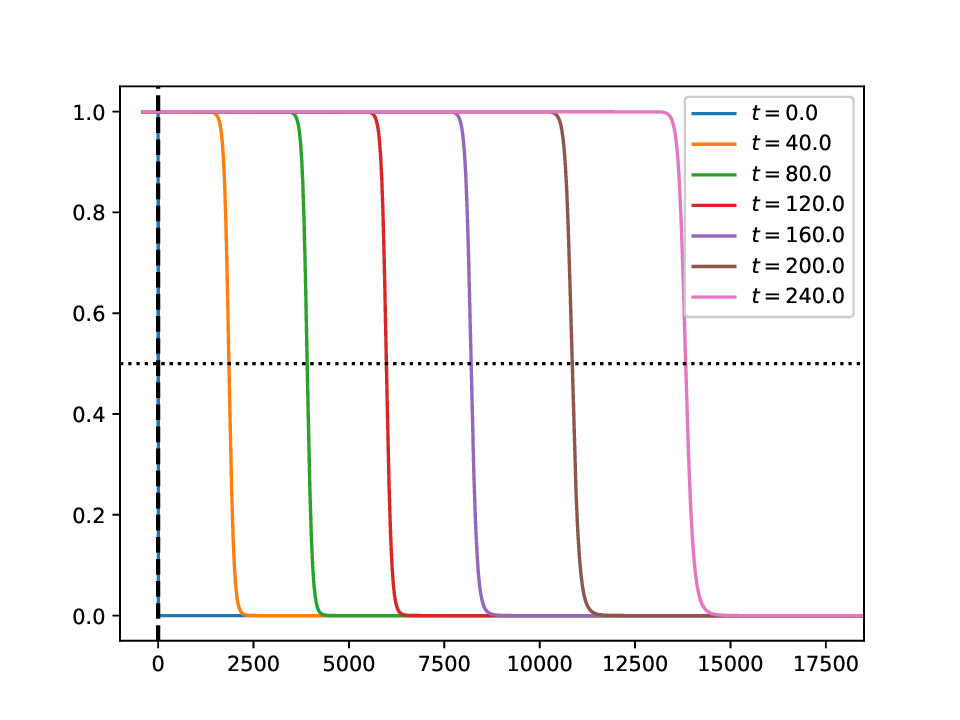}\caption{$\beta=0.45$} 
    \end{subfigure} 
     \quad
    \begin{subfigure}[b]{0.25\textwidth}  
        \includegraphics[width=\textwidth]{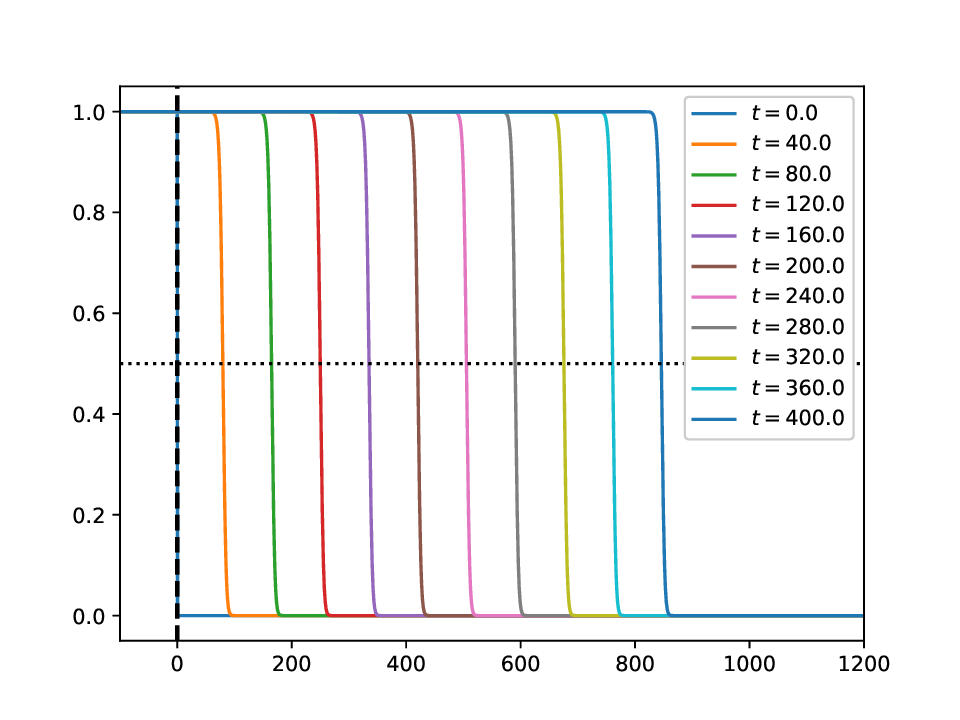} \caption{$\beta=2$}  
    \end{subfigure} 
\caption{\small Numerical approximations of the solution to \eqref{oeq1} with the dispersal kernel $J(z)\sim e^{-|z|^\beta}$ at different times for $\alpha=0.4$ and different values of $\beta$. The threshold for acceleration is $\beta=\frac{1}{\alpha+1}=\frac{5}{7}$.}
\label{fig:subexp0.4}
\end{figure} 

\subsubsection{Algebraic kernel}
To illustrate theoretical results in Theorem \ref{al} for algebraic kernel, we take $J(|z|)=\frac{1}{|z|^{1+2s}}$ with $s> 0$ for $|z|\ge 1$. One may observe that for $\left[\frac{N}{L}\right]+1\le n\le N$, $$\T^{2}_n=\int_{y_{n-1}}^{y_n} y^{1-2s}dy= 
\begin{cases}
    \frac{(\T x)^{2-2s}}{2-2s}\left(n^{2-2s}-(n-1)^{2-2s}\right), &\text{as } s\neq 1, \\
    \ln \frac{n}{n-1},&\text{as } s= 1,
\end{cases}
$$
while, for $1\le n\le \left[\frac{N}{L}\right]$, 
$$\T^2_n= y_n-y_{n-1}=\T x.$$ The constant $\mathcal{J}_L$ can be written explicitly as 
\[
\mathcal{J}_L=\int_L^{+\infty}\frac{1}{y^{1+2s}}dy=\frac{L^{2s}}{2s}.
\]
\par 
For this type of dispersal kernel $J$, Figures \ref{fig:algebraic0.5} and \ref{fig:algebraic1.0} show that acceleration always occurs, and that decreasing $s$ leads to an increase in the propagation speed. These observations are consistent with our theoretical findings, which show that $\ln X_\lambda(t)\sim_{t\to\infty} \frac{[r(\alpha+1)t]^\frac{1}{\alpha+1}}{2s}$ increases as $s$ decreases. We can also observe the flattening effect, which becomes more pronounced as $s$ decreases. Therefore, the tail of the dispersal kernel $J$ plays a key role in the propagation of the solution to \eqref{oeq1}, and as the tail becomes heavier, the propagation accelerates.

\begin{figure}[H] 
\centering 
    \begin{subfigure}[b]{0.25\textwidth} 
        \includegraphics[width=\textwidth]{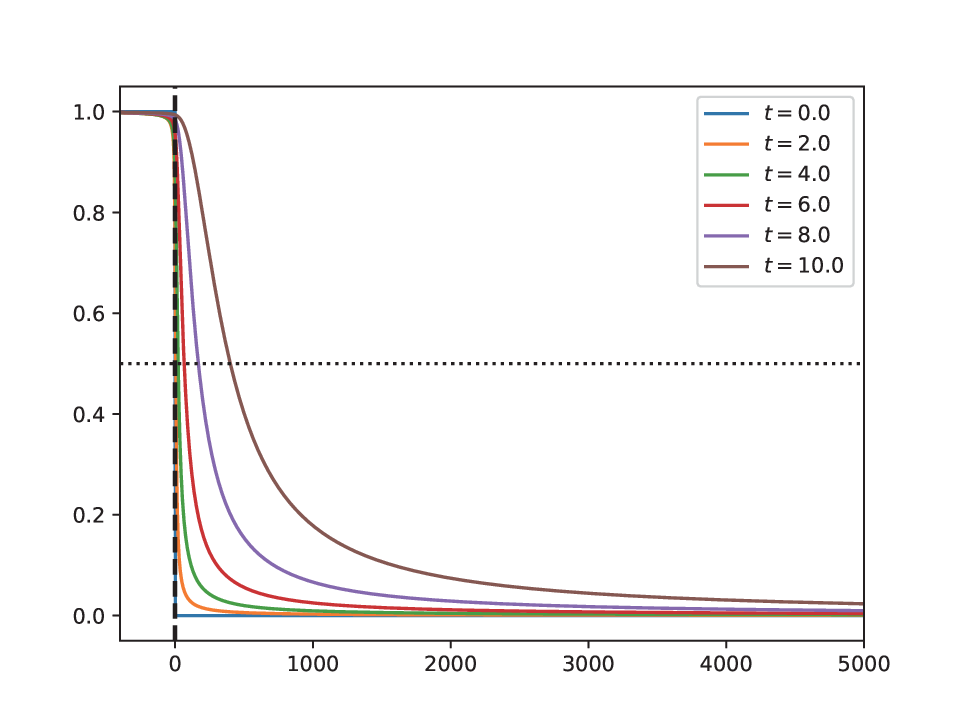} \caption{$s=0.5$} 
    \end{subfigure} 
    \quad
    \begin{subfigure}[b]{0.25\textwidth}  
        \includegraphics[width=\textwidth]{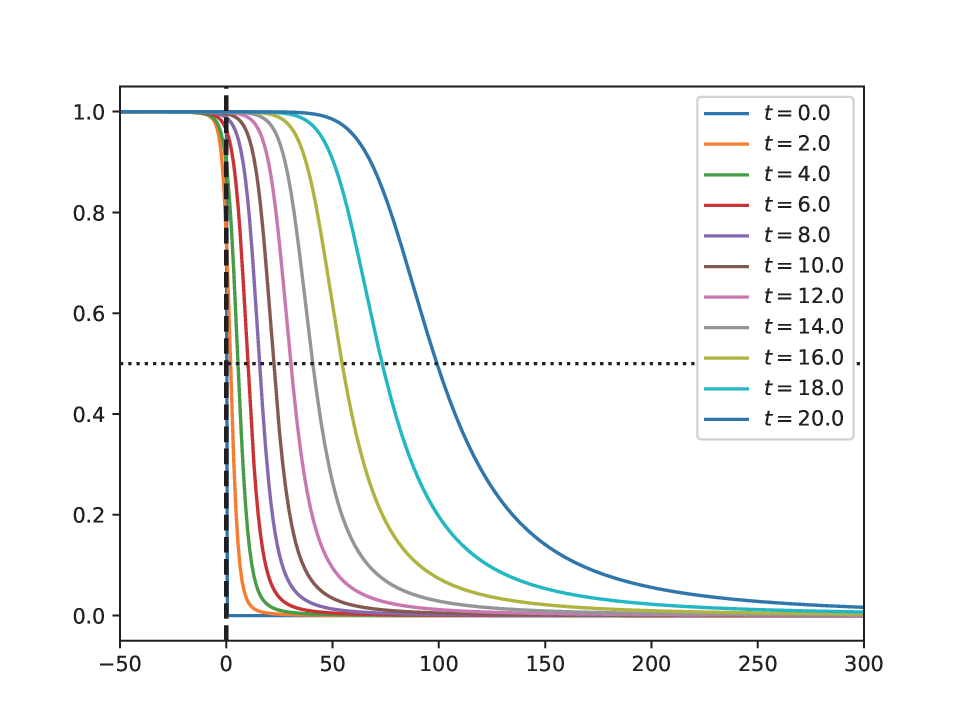}\caption{$s=1$} 
    \end{subfigure}  
     \quad
    \begin{subfigure}[b]{0.25\textwidth}  
        \includegraphics[width=\textwidth]{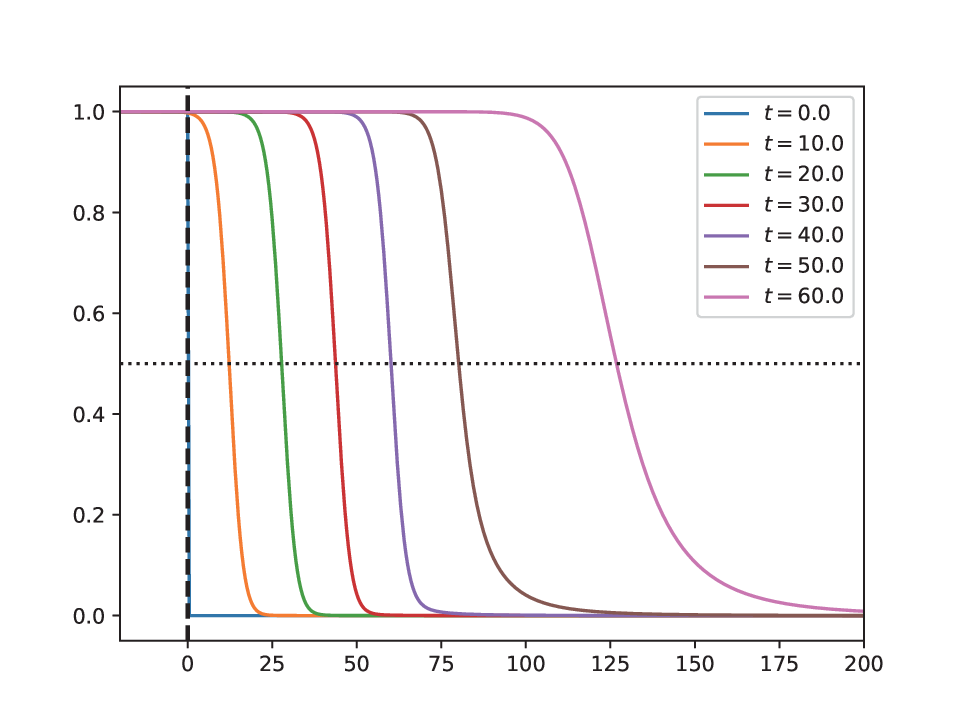} \caption{$s=2$}  
    \end{subfigure} 
\caption{\small Numerical approximations of the Cauchy problem of \eqref{oeq1} with the dispersal kernel $J(z)=\frac{1}{z^{1+2s}}$ at different times for $\alpha=0.5$ and different values of $s$. }
\label{fig:algebraic0.5}
\end{figure} 

%-------
\begin{figure}[H] 
\centering 
    \begin{subfigure}[b]{0.25\textwidth} 
        \includegraphics[width=\textwidth]{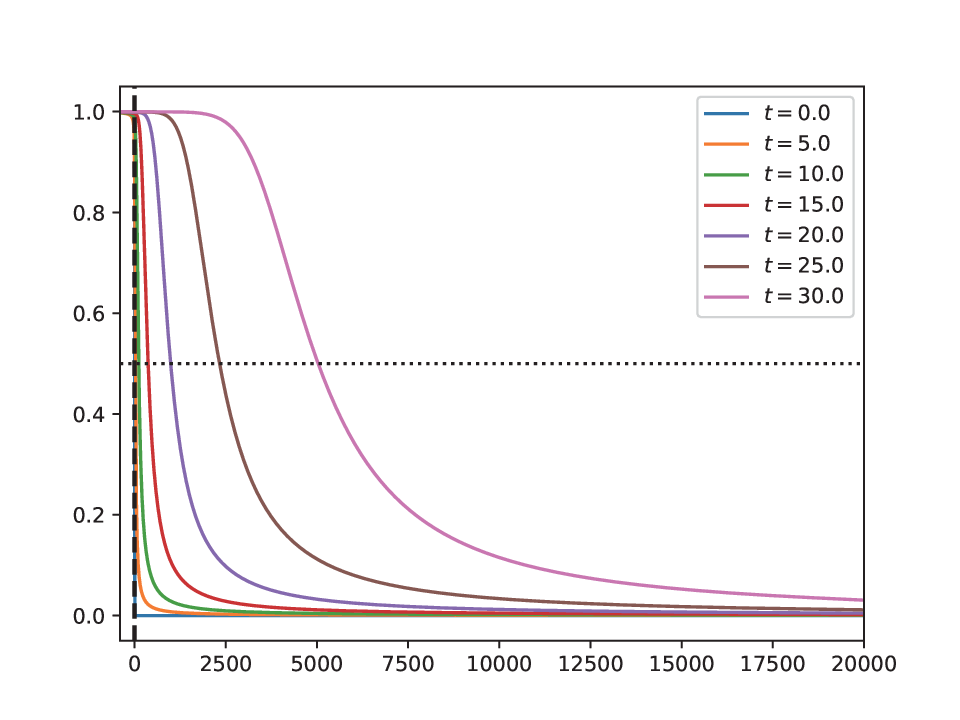} \caption{$s=0.5$} 
    \end{subfigure} 
    \quad
    \begin{subfigure}[b]{0.25\textwidth}  
        \includegraphics[width=\textwidth]{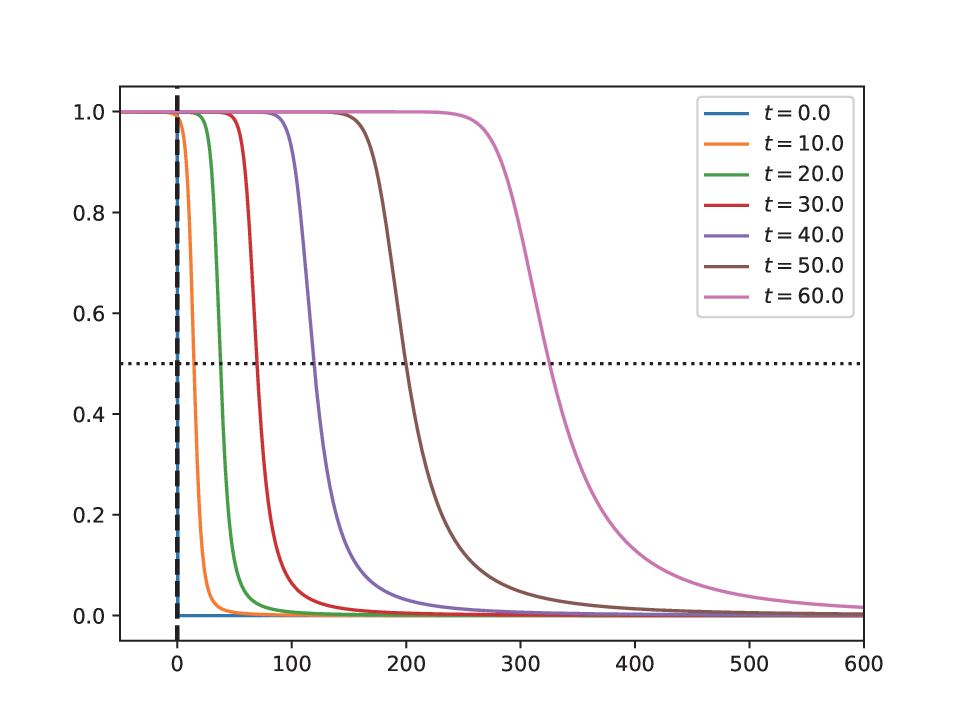}\caption{$s=1$} 
    \end{subfigure} 
     \quad
    \begin{subfigure}[b]{0.25\textwidth}  
        \includegraphics[width=\textwidth]{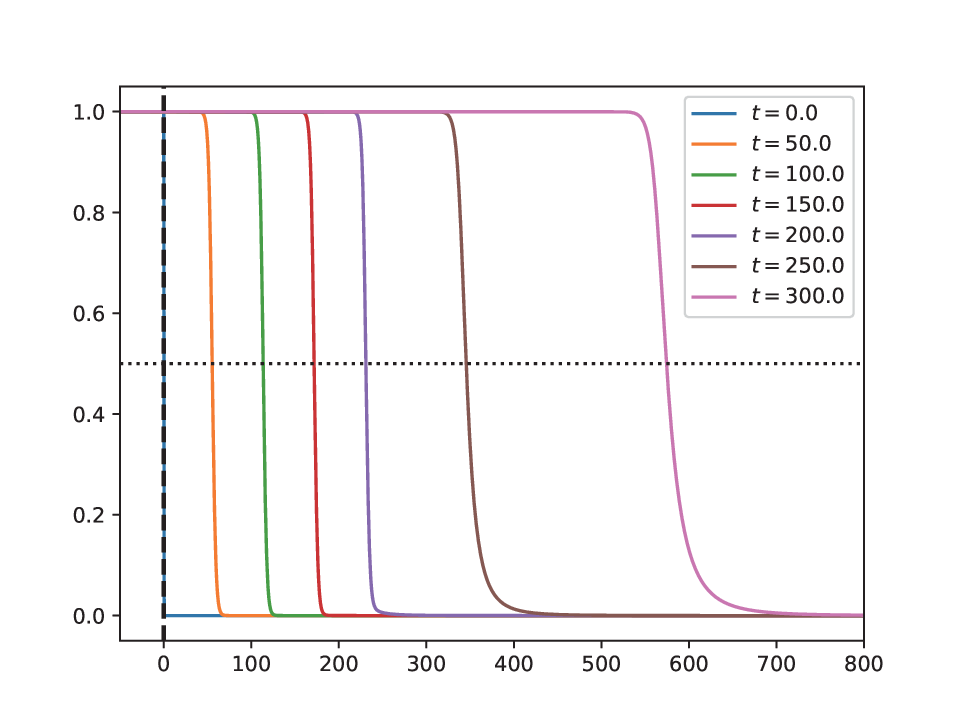} \caption{$s=2$}  
    \end{subfigure} 
\caption{\small Numerical approximations of the solution to \eqref{oeq1} with the dispersal kernel $J(z)=\frac{1}{z^{1+2s}}$ at different times for $\alpha=1$ and different values of $s$. }
\label{fig:algebraic1.0}
\end{figure} 

 \appendix

 \section{Technical proofs: case of an algebraic acceleration}\label{appendix-alg-acc}
 In this appendix, we collect all the technical proofs related to the construction of super- and sub-solution we are using in \Cref{s3}.  
 \subsection{The upper bound}\label{appendix-alg-acc-upper}
 Let us start with the proof of Lemma \ref{apub2l},

\begin{proof} [Proof of Lemma \ref{apub2l}]
A direct calculation shows that for $x> L$,
\begin{equation}\label{av0'}
    v_0'(x)= C_L(-\beta + px^{-\beta} )x^{p+\beta-1}e^{-x^\beta}\asymp -x^{p+\beta-1}e^{-x^\beta}.
\end{equation}
When $y=v_0(x)$, we have for all $x\ge L$,
\[
    1\le \frac{x^\beta}{\ln \frac{C_L}{y}}=\frac{x^\beta}{x^\beta-p\ln x}\to 1\quad \text{as }x\to +\infty,
\]
and 
\[
   \frac{p}{\beta} \le\frac{x^\beta-\ln \frac{C_L}{y}}{\ln\ln \frac{e^{L^\beta}}{y}}=\frac{p\ln x}{\beta\ln x+\ln\left(1-\frac{p\ln x-p\ln L}{x^\beta}\right)}\to \frac{p}{\beta} \quad \text{as }x\to +\infty.
\]
As a result, we achieve for any $y\in(0,1)$,
\begin{equation*}
 \left[\ln \frac{C_L}{y}+\frac{p}{\beta}\ln\ln\frac{e^{L^\beta}}{y} \right]^\frac{1}{\beta} \le v_0^{-1}(y) \le \left[\ln \frac{C_L}{y}+\left(\frac{p}{\beta}+\gamma\right)\ln\ln \frac{e^{L^\beta}}{y}\right],
\end{equation*}
for some positive constant $\gamma>0$.
Therefore, by the above estimate and \eqref{av0'}, for all $y\in(0,1)$, we get
\[
   (v_0'\circ v_0^{-1})(y)\gtrsim -\lvert\ln y\rvert^{\frac{\beta-1}{\beta}}y.
\]
This proves part $(i)$ in Lemma \ref{apub2l}.
\par
Since for all $t\in [t^\star,\infty)$ and $\Lambda\in(0,2]$, 
\begin{equation}\label{apube1} 
0<\exp\left\{1-[\rho (\alpha+1)t+(1-\ln \Lambda)^{\alpha+1}]^\frac{1}{\alpha+1} \right\}< 1,
\end{equation}
 then by the definition of $x_\Lambda$ and \eqref{aiv0}, we have
\[
  x_\Lambda(t)\asymp_\Lambda t^\frac{1}{\beta(1+\alpha)} \quad \text{for all }t\in [t^\star,\infty) \text{ and }\Lambda\in (0,e].
\]
 By the definition of $x_\Lambda$ and the chain rule, we get 
\[
   \frac{\partial x_\Lambda(t)}{\partial \Lambda} =\frac{\frac{(1-\ln \Lambda)^\alpha}{\Lambda}[\rho(\alpha+1)t+(1-\ln \Lambda)^{\alpha+1}]^{-\frac{\alpha}{\alpha+1}}\exp\Big[1-\Big(\rho(\alpha+1)t+(1-\ln \Lambda)^{\alpha+1}\Big)^{\frac{1}{\alpha+1}}\Big]}{(v_0'\circ v_0^{-1})\Big\{\exp\Big[1-\Big(\rho(\alpha+1)t+(1-\ln \Lambda)^{\alpha+1}\Big)^{\frac{1}{\alpha+1}}\Big] \Big\}}.
\]
It follows from \eqref{alubvv} that 
\[
   \frac{\partial x_\Lambda(t)}{\partial \Lambda} \lesssim - t^{\frac{1}{\beta(\alpha+1)}-1}\quad \text{for all }t\in [t^\star,\infty) \text{ and }\Lambda\in (0,2].
\]
Therefore, by the mean value theorem, for any $0<\Lambda_1< \Lambda_2\le 2$, there exists $\xi\in(\Lambda_1,\Lambda_2)$ such that, for all $t\in [t^\star,\infty)$,
\[
   x_{\Lambda_1}(t)-x_{\Lambda_2}(t)= -(\Lambda_2-\Lambda_1)\frac{\partial  x_\Lambda(t)}{\partial \Lambda} \vline_{\Lambda=\xi} \gtrsim_{\Lambda_1,\Lambda_2}t^{\frac{1}{\beta(\alpha+1)}-1}.
\]
This proves part $(ii)$ in Lemma \ref{apub2l} and ends the proof.
\end{proof}
Next, we prove Lemma \ref{lewxx}

\begin{proof}[Proof of Lemma \ref{lewxx}]
By the definitions of $w$ and $x_\Lambda$, for any $\Lambda\in(0,2]$ and $(t,x)\in[t^\star,\infty)\times [x_\Lambda(t),+\infty)$, we have
\[
(1-\ln v_0(x))^{\alpha+1}-\rho (\alpha+1)t\ge (1-\ln v_0(x_\Lambda(t)))^{\alpha+1}-\rho (\alpha+1)t =(1-\ln \Lambda)^{\alpha+1}\ge 0,
\]
and it follows that, for any $\Lambda\in(0,2]$, the function $w$ is well-defined on  $[t^\star,\infty)\times [x_\Lambda(t),+\infty)$.
\par
By \eqref{apube1} and the definition of $x_\Lambda$, it is easy to get that $x_\Lambda(t)\ge L$ for any $t\in[t^\star,\infty)$. 
Let us denote $\varphi_0=-\ln v_0$. By some direct calculations, for all $(t,x)\in[t^\star,\infty)\times [x_\Lambda(t),+\infty)$, we get  
\begin{equation}\label{wx}
    \frac{\partial w}{\partial x}=-\frac{w}{(1-\ln w)^\alpha}\varphi_0'(1+\varphi_0)^\alpha,
\end{equation}
and 
\begin{equation}\label{wxx}
\begin{aligned}
		\frac{\partial^2 w}{\partial x^2}&=\frac{w}{(1-\ln w)^{\alpha}}\Big\{(\varphi_0')^2 (1+\varphi_0)^{2\alpha}\left[(1-\ln w)^{-\alpha}
		+\alpha (1-\ln w)^{-(\alpha+1)}-\alpha(1+\varphi_0)^{-(\alpha+1)} \right]\\
  &\qquad \qquad\qquad -\varphi_0''(1+\varphi_0)^\alpha \Big\}.
\end{aligned}
\end{equation}
Since for $x>L$, $\varphi_0'=(\beta-p x^{-\beta})x^{\beta-1}\ge 0$, we get that, 
for any $t\in[t^\star,\infty)$, $w(t,\cdot)$ is non-increasing on $[x_\Lambda(t),+\infty)$, and it follows that
\[
    w(t,x)\le w(t,x_\Lambda(t))=\Lambda\le 2
    \quad\text{for all }(t,x)\in[t^\star,\infty)\times [x_\Lambda(t),+\infty).
\]
Since $\left[(1-\ln v_0(x))^{\alpha+1}-\rho(\alpha+1)t\right]^\frac{1}{\alpha+1}\le 1-\ln v_0(x)$, it follows from the definition of $w$ that
\[
w(t,x) \ge v_0(x)\quad\text{for all }(t,x)\in[t^\star,\infty)\times [x_\Lambda(t),+\infty).
\]
Thus, we get
\begin{equation}\label{lewxxe1}
    0<(1+\varphi_0)^{-(\alpha+1)}\le (1-\ln w)^{-(\alpha+1)}\le \ln^{-(\alpha+1)}\left(\frac{e}{2}\right).
\end{equation}
It follows from $0<w\le 2$ for all $(t,x)\in[t^\star,\infty)\times [x_\Lambda(t),+\infty)$ that, for any $(t,x)\in[t^\star,\infty)\times [x_\Lambda(t),+\infty)$, one has 
\[
    0<(1-\ln w)^{-\alpha}+\alpha(1-\ln w)^{-(\alpha+1)}-\alpha(1+\varphi_0)^{-(\alpha+1)}<2\ln^{-(\alpha+1)}\left(\frac{e}{2}\right).
\] 
Since for $x\ge L$, $\varphi(x)=-\ln v_0(x)= x^\beta-p\ln x-\ln C_L$ and $0<\beta<\frac{1}{\alpha+1}$, we get 
\[
    0\le (\varphi_0')^2(1+\varphi_0)^{2\alpha}=\left(\beta x^{-(1-\beta)}-px^{-1}\right)^2\left(1-\ln C_L+x^\beta-p\ln x\right)^{2\alpha}\le \frac{\mathcal{C}_0}{4\ln^{-(\alpha+1)}\left(\frac{e}{2}\right)},
\]
and 
\begin{equation}\label{lewxxe2}
    0\le -\varphi_0''(1+\varphi_0)=\left(\beta(1-\beta)x^{-(2-\beta)}-px^{-2}\right)\left(1-\ln C_L+x^\beta-p\ln x\right)^\alpha\le \frac{\mathcal{C}_0}{2},
\end{equation}
 for some positive constant $\mathcal{C}_0$.
Therefore, we conclude that, for any $t\in[t^\star,\infty)$, $w(t,\cdot)$ is convex on $[x_\Lambda(t),+\infty)$, and  there is some positive constant $\mathcal{C}_0$ such that for $(t,x)\in[t^\star,\infty)\times [x_\Lambda(t),+\infty)$,
\[
    0\le \frac{\partial^2 w}{\partial x^2}(t,x)\le \mathcal{C}_0\frac{w(t,x)}{(1-\ln w(t,x))^\alpha}.
\]
\par 
      To complete this lemma, for $t\ge t^\star$ and $x\ge x_\Lambda(t)$, we shall prove that $\frac{\partial^2 \ln w(t,x)}{\partial x^2}\ge 0$, that is, 
   \[
   \frac{1}{w^2}\left[w\frac{\partial^2 w}{\partial x^2}-\left(\frac{\partial w}{\partial x}\right)^2\right]\ge 0.
   \]
    It follows from \eqref{wx} and \eqref{wxx} that 
    \[
   w\frac{\partial^2 w}{\partial x^2}-\left(\frac{\partial w}{\partial x}\right)^2= \frac{w^2}{(1-\ln w)^{\alpha}}\Bigg\{\alpha(\varphi_0')^2 (1+\varphi_0)^{2\alpha}\left[ (1-\ln w)^{-(\alpha+1)}-(1+\varphi_0)^{-(\alpha+1)} \right]-\varphi_0''(1+\varphi_0)^\alpha \Bigg\}.
    \]
    Therefore, by \eqref{lewxxe1} and \eqref{lewxxe2}, we achieve $\frac{\partial^2 \ln w}{\partial x^2}\ge 0$ for $(t,x)\in[t^\star,\infty)\times [x_\Lambda(t),+\infty)$.
\end{proof}

Last, we provide a proof for Lemma \ref{asuble1}:
\begin{proof}[Proof of Lemma \ref{asuble1}]
Unless otherwise specified, throughout the proof of Lemma \ref{asuble1}, we assume that $(t,x)\in[t^\star,\infty)\times[x_1(t)+1,+\infty)$.
    By the definition of $w$, for all $z\in\left(\frac{1}{x},1-\frac{x_1(t)}{x}\right)$, we write
\[
    \begin{aligned}
    \frac{1-\ln w(t,x(1-z))}{1-\ln w(t,x)}
    =\left\{1-\frac{(1-\ln v_0(x))^{\alpha+1}}{(1-\ln w(t,x))^{\alpha+1}}\left[ 1-\left(1-\frac{x^\beta-x^\beta(1-z)^\beta+p\ln (1-z)}{1-\ln v_0(x)}\right)^{\alpha+1}\right]\right\}^\frac{1}{\alpha+1}.
    \end{aligned}
\]
Since $z\mapsto x^\beta-x^\beta(1-z)^\beta+p\ln(1-z)$ is increasing on $\left(\frac{1}{x},1-\frac{x_1(t)}{x}\right)$, we obtain, for all $z\in\left(\frac{1}{x},1-\frac{x_1(t)}{x}\right)$, 
\begin{equation}\label{asuble2}
    0\le \frac{x^\beta-x^\beta(1-z)^\beta+p\ln (1-z)}{1-\ln v_0(x)}\le  \frac{x^\beta-p\ln x-x_1^\beta(t)+p\ln x_1(t)}{1-\ln v_0(x)}=1-\frac{1-\ln v_0(x_1(t))}{1-\ln v_0(x)}\le 1,
\end{equation}
it follows that
\[
\begin{aligned}
   & 0< \frac{(1-\ln v_0(x))^{\alpha+1}}{(1-\ln w(t,x))^{\alpha+1}}\left[ 1-\left(1-\frac{x^\beta-x^\beta(1-z)^\beta+p\ln (1-z)}{1-\ln v_0(x)}\right)^{\alpha+1}\right]\\
    &\le \frac{[(1-\ln v_0(x))^{\alpha+1}-\rho(\alpha+1)t]-[(1-\ln v_0(x_1(t)))^{\alpha+1}-\rho(\alpha+1)t]}{(1-\ln w(t,x))^{\alpha+1}}\\
    &=1-\frac{1}{(1-\ln w(t,x))^{\alpha+1}}\le 1.
\end{aligned}
\]
Since $\beta\in(0,1)$ and $\frac{1}{\alpha+1}\in(0,1)$,  by using the inequality $(1-y)^n\ge 1-ny-(1-n)y^2$  for all $y\in[0,1]$ and some $n\in(0,1)$\footnote{Let us give a quick proof. For $n\in(0,1)$, denote $h(y):=(1-y)^n-1+ny +(1-n) y^2$ on $[0,1]$. By a direct calculation, $h'(y)=-n(1-y)^{n-1}+n+2(1-n)y$ and $h''(y)=-n(1-n)(1-y)^{-(2-n)}+2(1-n)$. Observe that, for $h''$, there is a unique root $y_1:=1-\left(\frac{n}{2}\right)^\frac{1}{2-n}\in (0,1)$, and that $h''> 0$ on $\left(0,y_1\right)$ and $h''< 0$ on $\left(y_1,1\right)$, that is, $h'$ is strictly increasing on $\left(0,y_1\right)$ and strictly decreasing on $\left(y_1,1\right)$. It follows from the facts that $h'(0)=0$, $h'\left(\frac{1}{2}\right)=-n2^{1-n}+2>0$ and $h'(y)\to -\infty$ as $y\to 1^-$ that there is a unique root $y_2\in(0,1)$ such that $h'(y_2)=0$, and thus, $h'> 0$ on $(0,y_2)$ and $h'<0$ on $(y_2,1)$, that is, $h$ is strictly increasing on $\left(0,y_2\right)$ and strictly decreasing on $\left(y_2,1\right)$. Therefore, we have $h(y)\ge \min\{h(0),h(1)\}=0$, that is, $(1-y)^n-1+ny +(1-n) y^2\ge 0$ on $[0,1]$. 
 Moreover, since $ny+(1-n)y^2\le y$ for all $n\in(0,1)$ and $y\in[0,1]$, we have $(1-y)^n\ge 1-y$ for all $y\in[0,1]$. \label{twoineq}}, we obtain
\[
\begin{aligned}
     \frac{1-\ln w(t,x(1-z))}{1-\ln w(t,x)}
     &\ge 1-\frac{1}{\alpha+1}\frac{(1-\ln v_0(x))^{\alpha+1}}{(1-\ln w(t,x))^{\alpha+1}}\left[ 1-\left(1-\frac{x^\beta-x^\beta(1-z)^\beta+p\ln (1-z)}{1-\ln v_0(x)}\right)^{\alpha+1}\right]\\
     &\ -\frac{\alpha}{\alpha+1}\left\{\frac{(1-\ln v_0(x))^{\alpha+1}}{(1-\ln w(t,x))^{\alpha+1}}\left[ 1-\left(1-\frac{x^\beta-x^\beta(1-z)^\beta+p\ln (1-z)}{1-\ln v_0(x)}\right)^{\alpha+1}\right]\right\}^2.
\end{aligned}
\]
In view of the definitions of $x_\Lambda$ and $C_L$, one may notice that $x_1(t)\ge L\ge \left(\frac{e}{C_L}\right)^\frac{1}{p}$, and thus, 
\begin{equation}
    \label{1lv0}
1-\ln v_0(x)\le x^\beta\quad \text{for all } (t,x)\in[t^\star,\infty)\times[x_1(t)+1,+\infty).
\end{equation}
 By using \eqref{asuble2} and the inequality $(1-y)^{\alpha+1}\ge 1-(\alpha+1)y$ for $y\le 1$, we have
\[
\begin{aligned}
	 0\le  1-\left(1-\frac{x^\beta-x^\beta(1-z)^\beta+p\ln (1-z)}{1-\ln v_0(x)}\right)^{\alpha+1}
 &\le (\alpha+1)\frac{x^\beta\left(1-(1-z)^\beta\right)+p\ln (1-z)}{1-\ln v_0(x)}\\
 &\le (\alpha+1) \frac{x^\beta}{1-\ln v_0(x)}\left(1-(1-z)^\beta\right),
\end{aligned}
\]
and thus, by \eqref{1lv0} and using the inequalities $(1-y)^n\ge 1-ny-(1-n)y^2$  and $(1-y)^n\ge 1-y$ for all $y\in[0,1]$ and some $n\in(0,1)$ (see footnote \ref{twoineq}), we achieve, for all $z\in\left(\frac{1}{x},1-\frac{x_1(t)}{x}\right)$,
\[
\begin{aligned}
\frac{1-\ln w(t,x(1-z))}{1-\ln w(t,x)}
        &\ge 1- \frac{x^{\alpha\beta+\beta}}{(1-\ln w(t,x))^{\alpha+1}} \left(1-(1-z)^\beta\right)  -\frac{\alpha(\alpha+1)x^{2\alpha\beta+2\beta}}{(1-\ln w(t,x))^{2\alpha+2}} \left(1-(1-z)^\beta\right)^2 \\
   &\ge 1-\frac{\beta x^{\alpha\beta+\beta}}{(1-\ln w(t,x))^{\alpha+1}}  z-  C_{\alpha,\beta} \frac{x^{2\alpha\beta+2\beta}}{(1-\ln w(t,x))^{2\alpha+2}}z^2,
\end{aligned}
\]
where $ C_{\alpha,\beta} :=1-\beta+\alpha(\alpha+1)>0$.
 Therefore, for all $z\in\left(\frac{1}{x},1-\frac{x_1(t)}{x}\right)$, we have
\[
\begin{aligned}
    \frac{w(t,x(1-z))}{w(t,x)}&=\exp\left\{(1-\ln w(t,x))\left(1-\frac{1-\ln w(t,x(1-z))}{1-\ln w(t,x)}\right) \right\} \\
    &\le\exp\left\{ \frac{ x^{\alpha\beta+\beta}}{(1-\ln w(t,x))^{\alpha}}   \left(1-(1-z)^\beta\right)+\alpha(\alpha+1)\frac{x^{2\alpha\beta+2\beta}}{(1-\ln w(t,x))^{2\alpha+1}} \left(1-(1-z)^\beta\right)^2\right\}\\
    &\le\exp\left\{ \frac{\beta x^{\alpha\beta+\beta}}{(1-\ln w(t,x))^{\alpha}}  z+ C_{\alpha,\beta} \frac{x^{2\alpha\beta+2\beta}}{(1-\ln w(t,x))^{2\alpha+1}}z^2\right\}.
\end{aligned}
\]
This completes the proof.
\end{proof}

\subsection{The lower bound}\label{appendix-alg-acc-low}
In this subsection, we proof a technical Lemma, Lemma \ref{apubl1}.
\begin{proof}[Proof of Lemma \ref{apubl1}]
\# \textbf{Let us prove (i) and $(ii)$.}
 By the definition of $x_\Lambda$, for any $\Lambda\in\left(0,e^{-1}\right)$,
 \[
     x_\Lambda(t)\ge \ln^\frac{1}{\beta}\frac{1}{\Lambda}>1\quad \text{for all }t\ge 0,
 \]  
 and for any $\Lambda\in(0,1)$,
 \begin{equation}\label{aplbxtl}    
 x_\Lambda(t)\ge Y(t):=\left\{-1+[\rho(\alpha+1)t]^\frac{1}{\alpha+1}\right\}^\frac{1}{\beta}\quad \text{for all }t\ge \frac{1}{\rho(\alpha+1)}.
 \end{equation}
 By the definition of $x_\Lambda$, one has
 \[
 \frac{\partial x_\Lambda(t)}{\partial \Lambda}\asymp_\Lambda -t^{\frac{1}{\beta(\alpha+1)}-1},
 \]
 and thus, by the mean value theorem, for any $0<\Lambda_1<\Lambda_2<1$, there is $\xi\in(\Lambda_1,\Lambda_2)$ such that 
 \[
 x_{\Lambda_1}(t)-x_{\Lambda_2}(t)= - (\Lambda_2-\Lambda_1)\frac{\partial x_\Lambda(t)}{\partial \Lambda}\vline_{\Lambda=\xi}\asymp t^{\frac{1}{\beta(\alpha+1)}-1}\to +\infty\quad \text{as }t\to +\infty,
 \]
  since $\alpha\beta+\beta-1<0$.
 \par \# \textbf{Let us prove (iii).} In view of the definitions of $w$ and $x_\Lambda$, it is easy to get that, for any $\Lambda\in(0,1)$, $w$ is well-defined on $\mathbb{R}^+\times [x_\Lambda(t),+\infty)$. 
 Let us denote $\varphi_0=-\ln v_0$,  by some direct calculations, for all $t>0$ and $x\ge x_\Lambda(t)$, we get 
\begin{equation}\label{alwx}
    \frac{\partial w}{\partial x}=-\frac{w}{(1-\ln w)^\alpha}\varphi_0'(1+\varphi_0)^\alpha,
\end{equation}
and 
\begin{equation}\label{alwxx}
		\frac{\partial^2 w}{\partial x^2}=\frac{w}{(1-\ln w)^{\alpha}}\Big\{(\varphi_0')^2 (1+\varphi_0)^{2\alpha}\Big((1-\ln w)^{-\alpha}
		+\alpha (1-\ln w)^{-(\alpha+1)}-\alpha(1+\varphi_0)^{-(\alpha+1)} \Big)-\varphi_0''(1+\varphi_0)^\alpha \Big\}.
\end{equation}
We then observe that for all $t\ge 0$, $w(t,\cdot)$ is non-increasing on $[x_\Lambda(t),+\infty)$, and that 
\begin{equation*}
\begin{aligned}
\left|\frac{\partial w}{\partial x}(t,x_\Lambda(t))\right|
     &=\frac{\Lambda}{(1-\ln\Lambda)^\alpha}x_\Lambda^{\alpha\beta+\beta-1}(t)\left(1+\frac{1}{x_\Lambda^\beta(t)}\right)^\alpha\\
    &\le Y^{\alpha\beta+\beta-1}(t)\left(1+\frac{1}{Y^\beta(t)}\right)^\alpha \to 0\quad \text{as }t\to+\infty,
\end{aligned}
\end{equation*}
by \eqref{aplbxtl} and $\alpha\beta+\beta-1<0$.
Similarly to Lemma \ref{lewxx}, one may see that
\[
    \frac{\partial^2 w}{\partial x^2}(t,x)\ge 0\quad \text{for all }t>0 \text{ and }x\ge x_\Lambda(t).
\]
\par \# \textbf{Let us prove (iv).}
By the definition of $w$, for all $t>0$ and $x\ge x_\Lambda(t)$, one may find 
\begin{equation*}
    w(t,x)= v_0(x)\exp\left\{1-\ln v_0(x)-[(1-\ln v_0(x))^{\alpha+1}-\rho(\alpha+1)t]^\frac{1}{\alpha+1}\right\}\le e^{-x^\beta+[\rho(\alpha+1)t]^\frac{1}{\alpha+1}}.
\end{equation*}
\end{proof}

 \section{Technical proofs: case of an exponential acceleration}\label{appendix-exp-acc}
 In this appendix, we collect all the technical proofs related to the construction of super- and sub-solution we are using in \Cref{s4}.
  \subsection{The upper bound}\label{appendix-exp-acc-upper}
 Let us start with the proof of Lemma \ref{epubl1},
\begin{proof}[Proof of Lemma \ref{epubl1}]
The proof is very similar to Lemma \ref{apub2l}, so we streamline the arguments.
\medskip
\par\noindent
\# \textbf{Proof of Part $I.$}
By a direct calculation, one gets for $x> L$,
\begin{equation}\label{ev0'}
    v_0'(x)=C_L(-2s+q \ln^{-1} x)\frac{\ln^{q} x}{x^{2s+1}}\asymp -\frac{\ln^{q} x}{x^{2s+1}}.
\end{equation}
When $y= v_0(x)$, we have
\[
\begin{aligned}
     \frac{v_0^{-1}(y)}{C_L^\frac{1}{2s} y^{-\frac{1}{2s}}\left(\ln L -\frac{1}{2s}\ln y\right)^\frac{q}{2s}}&=\frac{x}{C_L^\frac{1}{2s}(C_Lx^{-2s}\ln^{q} x)^{-\frac{1}{2s}}\left(\ln L-\frac{1}{2s}\ln (C_Lx^{-2s}\ln^{q} x)\right)^\frac{q}{2s} }\\
     &= \left(1-\frac{q}{2s}\frac{\ln\ln x-\ln\ln L}{\ln x}\right)^{-\frac{q}{2s}}\to 1 \quad \text{as }x\to +\infty, 
    \end{aligned}
\]
and notice that $\left(1-\frac{q}{2s}\frac{\ln\ln x-\ln\ln L}{\ln x}\right)^{-\frac{q}{2s}}\ge 1$ for all $x\ge L$. This implies \eqref{xysim}.
\par Collecting \eqref{xysim} and \eqref{ev0'}, since $x\mapsto -\frac{\ln^{q} x}{x^{2s+1}}$ is increasing for $x> L$, one may obtain for $y\in(0,1)$,
\[
    (v'\circ v^{-1})(y)\asymp - y^{1+\frac{1}{2s}}(\ln y^{-1})^{-\frac{q}{2s}},
\]
which gives \eqref{esubv'vi}.
\medskip
\par
\noindent
\# \textbf{Proof of Part $II$.}
 Since for all $t\in[t^\star,+\infty)$ and $\Lambda \in(0,2]$, 
\begin{equation}\label{epube1}   
   0< (1+\kappa t^{p})^{-1}\exp\left\{1-\left[r(\alpha+1) t+(1-\ln \Lambda)^{\alpha+1}\right]^{\frac{1}{\alpha+1}}\right\}< 1,
\end{equation}
then by \eqref{xysim} and the definition of $x_\Lambda$, for any $t\ge t^\star$, one may get  
\[
    x_\Lambda(t)\asymp_\Lambda t^{\frac{1}{2s}\left(p+\frac{q}{\alpha+1}\right)}\exp\left\{\frac{1}{2s}[r(\alpha+1)t]^\frac{1}{\alpha+1}\right\}.
\]

\par
In view of \eqref{esubv'vi} and the definition of $x_\Lambda$, for all $t\ge t^\star$, we get 
\[
\begin{aligned}
  \frac{\partial x_\Lambda(t)}{\partial \Lambda}&=\frac{(1-\ln \Lambda)^\alpha \frac{1}{\Lambda}(1+\kappa t^p)^{-1}\left[r(\alpha+1)t+(1-\ln \Lambda)^{\alpha+1}\right]^{-\frac{\alpha}{\alpha+1}}\exp\left\{1-\left[r(\alpha+1) t+(1-\ln \Lambda)^{\alpha+1}\right]^{\frac{1}{\alpha+1}}\right\}}{v_0'\circ v_0^{-1}\left\{(1+\kappa t^p)^{-1}\exp\left\{1-\left[r(\alpha+1) t+(1-\ln \Lambda )^{\alpha+1}\right]^{\frac{1}{\alpha+1}}\right\}\right\}}\\
   &\asymp_\Lambda - t^{-\frac{\alpha}{\alpha+1}+\frac{1}{2s}\left(p+\frac{q}{\alpha+1}\right)} e^{\frac{1}{2s}[r(\alpha+1)t]^\frac{1}{\alpha+1}}.
   \end{aligned} 
\]
Therefore, by the mean value theorem, for any $0<\Lambda_1<\Lambda_2\le 2$, there is $\xi\in (\Lambda_1,\Lambda_2)$ such that, for all $t\ge t^\star$,
\[
x_{\Lambda_1}(t)-x_{\Lambda_2}(t)=-(\Lambda_2-\Lambda_1) \frac{\partial x_\Lambda(t)}{\partial \Lambda}\vline_{\Lambda=\xi}\asymp_{\Lambda_1,\Lambda_2} t^{-\frac{\alpha}{\alpha+1}+\frac{1}{2s}\left(\frac{q}{\alpha+1}+p\right)}e^{\frac{1}{2s}[r(\alpha+1)t]^\frac{1}{\alpha+1}}.
\]
This completes the proof.
\end{proof}

 Next, we prove Lemma \ref{esuble1}.
\begin{proof}[Proof of Lemma \ref{esuble1}]
By the definitions of $x_\Lambda$ and $w$, for all $\Lambda\in(0,2]$ and $t\ge t^\star$, we have
\[
\left[1-\ln[(1+ \kappa t^{p})v_0(x)]\right]^{\alpha+1}-r(\alpha+1)t\ge \left[1-\ln[(1+ \kappa t^{p})v_0(x_\Lambda(t))]\right]^{\alpha+1}-r(\alpha+1)t=(1-\ln \Lambda)^{\alpha+1}> 0,
\]
and thus, we have that $\psi$ is well defined for all $x\ge x_\Lambda(t)$ and $\Lambda \in(0,2]$.
\par
  By the definition of $\psi$ and some direct calculations, we have 
\[
\begin{aligned}
        \frac{\partial \psi}{\partial t}
        =\frac{\psi}{(1-\ln \psi)^\alpha}\left\{\left\{1-\ln [\left(1+\kappa t^p\right)v_0]\right\}^\alpha \frac{p \kappa t^{p-1}}{1+\kappa t^p}+r\right\},
        \end{aligned}
\]
\begin{equation}\label{esubpsix}
\begin{aligned}
    \frac{\partial\psi}{\partial x}=\frac{\psi}{(1-\ln\psi)^\alpha}\left\{1-\ln [\left(1+\kappa t^p\right)v_0]\right\}^\alpha\frac{v_0'}{v_0},
    \end{aligned}
\end{equation}
and 
\begin{equation}\label{esubpsixx}
\begin{aligned}
    \frac{\partial^2 \psi}{\partial x^2}&=\frac{\psi}{(1-\ln\psi )^\alpha}\Bigg\{\left(\frac{v_0'}{v_0} \right)^2 \{1-\ln [(1+\kappa t^p)v_0]\}^{2\alpha}\Big\{(1-\ln \psi)^{-\alpha}+\alpha(1-\ln \psi)^{-(\alpha+1)}   \\
    &\qquad\qquad\qquad\quad-\alpha[1-\ln [(1+\kappa t^p) v_0]]^{-(\alpha+1)}\Big\} +[1-\ln [(1+\kappa t^p)v_0]]^\alpha\left(\frac{v_0'}{v_0}\right)'\Bigg\}.
\end{aligned}
\end{equation}
By \eqref{epube1} and the definition of $x_\Lambda$, for any $t\ge t^\star$, one may get  
\[
    x_\Lambda(t)>L.
\]
In view of the definition of  $v_0$, for all $x> L$, one gets
\begin{equation}\label{epubv0'v0}
    \frac{v_0'}{v_0}=\left(-2s+\frac{q}{\ln x}\right)\frac{1}{x}<0,
\end{equation}
and thus, one may notice that $\psi(t,x)$ is increasing in $t$ and decreasing in $x$ for all $t\ge t^\star$ and $x\ge x_\Lambda(t)$.
In addition, by \eqref{epubv0'v0}, for all $x\ge  x_\Lambda(t)$, there exist $t_0>t^\star$, $\gamma_1>0$ and $\gamma_2>0$, such that for all $t\ge t_0$, we have\footnote{By the definition of $x_\Lambda$, we have $1-\ln [\left(1+\kappa t^p\right) v_0( x_\Lambda(t))]=[(1-\ln \Lambda)^{\alpha+1}+r(\alpha+1)t]^\frac{1}{\alpha+1}\gtrsim t^\frac{1}{\alpha+1}$.}
\[
\begin{aligned}
        \frac{\partial \psi}{\partial t}\ge\left(\gamma_1 t^{-\frac{1}{\alpha+1}}+r\right)\frac{\psi}{(1-\ln \psi)^\alpha},
        \end{aligned}
\]
and
\[
\begin{aligned}
  -\gamma_2^{-1}\frac{\ln^\alpha x}{x}\frac{\psi}{(1-\ln \psi)^\alpha}\le   \frac{\partial\psi}{\partial x}&\le \left[r(\alpha+1)t\right]^\frac{\alpha}{\alpha+1}\left(-2s+\frac{q}{\ln x}\right)\frac{1}{x}\frac{\psi}{(1-\ln\psi)^\alpha}\le -\gamma_2 \frac{t^\frac{\alpha}{\alpha+1}}{x}\frac{\psi}{(1-\ln\psi)^\alpha}.
    \end{aligned}
\]
\par
  In view of the definition of $\psi$, we obtain
\begin{equation}\label{esubl1e}
    0<\{1-\ln [(1+\kappa t^p)v_0]\}^{-(\alpha+1)}\le (1-\ln \psi)^{-(\alpha+1)}\le \ln^{-(\alpha+1)}\left(\frac{e}{2}\right)  \quad \text{for all }t\ge t^\star \text{ and }  x\ge  x_\Lambda(t).
\end{equation}
It follows from $0<\psi\le 2$ for all $t\ge t^\star$ and $x\ge x_\Lambda(t)$ that, for any $t\ge t^\star$ and $x\ge x_\Lambda(t)$, we obtain 
\[
    0<(1-\ln \psi)^{-\alpha}+\alpha(1-\ln \psi)^{-(\alpha+1)}-\alpha\{1-\ln[(1+ \kappa t^p) v_0]\}^{-(\alpha+1)}<2\ln^{-(\alpha+1)}\left(\frac{e}{2}\right) .
\] 
Since $\frac{v_0'}{v_0}=\left(-2s+\frac{q}{\ln x}\right)\frac{1}{x}$ and $\left(\frac{v_0'}{v_0}\right)'=\left(2s- \frac{q}{\ln x}-\frac{q}{\ln^2 x}\right)\frac{1}{x^2}$, there is some positive constant $\mathcal{C}_0$  such that for all $t>0$ and $x\ge x_1(t)$,
\[
    0\le \left(\frac{v_0'}{v_0}\right)^2\{1-\ln[(1+ \kappa t^p) v_0]\}^{2\alpha}\le \left(-2s+\frac{q}{\ln x}\right)^2\frac{1}{x^2}\left(1-\ln \left(C_L\frac{\ln^q x}{x^{2s}}\right)\right)^{2\alpha}\le \frac{\mathcal{C}_0}{4\ln^{-(\alpha+1)}\left(\frac{e}{2}\right)}\frac{1}{x}
\]
and 
\[
    0\le \left(\frac{v_0'}{v_0}\right)'\left\{1-\ln[\left(1+\kappa t^p \right)v_0]\right\}^\alpha\le \left(2s-\frac{q}{\ln x}-\frac{q}{\ln^2 x}\right)\frac{1}{x^2}\left(1-\ln\left( C_L\frac{\ln^q x}{x^{2s}}\right)\right)^\alpha\le \frac{\mathcal{C}_0}{2}\frac{1}{x}
\]
Therefore, by \eqref{esubpsixx}, we conclude that there is some constant $\mathcal{C}_0$ such that for $t\ge t^\star $ and $x\ge x_\Lambda(t)$,
\[
    0\le \frac{\partial^2 \psi}{\partial x^2}(t,x)\le \frac{\mathcal{C}_0}{x}\frac{\psi(t,x)}{(1-\ln \psi(t,x))^\alpha}.
\]
\par 
Now, let us show that, for all $t\ge t^\star$, $\psi(t,\cdot)$ is log-convex on $[ x_\Lambda(t),+\infty)$. By a direct calculation, one may have
\[
\frac{\partial^2 \ln \psi}{\partial x^2}=\frac{1}{\psi^2}\left[\psi\frac{\partial^2 \psi}{\partial x^2}-\left(\frac{\partial \psi}{\partial x}\right)^2\right].
\]
It follows from \eqref{esubpsix}, \eqref{esubpsixx} and \eqref{esubl1e}  that for all $x\ge x_\Lambda(t)$,
\[
\begin{aligned}
\psi \frac{\partial^2 \psi}{\partial x^2}-\left(\frac{\partial \psi}{\partial x}\right)^2&=\frac{\psi^2}{(1-\ln\psi )^\alpha}\Bigg\{\left(\frac{v_0'}{v_0} \right)^2 \{1-\ln [(1+\kappa t^p)v_0]\}^{2\alpha}\Big\{\alpha(1-\ln \psi)^{-(\alpha+1)}   \\
    &\qquad\qquad\qquad\quad-\alpha[1-\ln [(1+\kappa t^p) v_0]]^{-(\alpha+1)}\Big\} +[1-\ln [(1+\kappa t^p)v_0]]^\alpha\left(\frac{v_0'}{v_0}\right)'\Bigg\}\ge 0.
\end{aligned}
\]
Therefore, for all $t\ge t^\star$, we achieve $\frac{\partial^2 \ln \psi}{\partial x^2}(t,x)\ge 0$ for all $x\ge  x_\Lambda(t)$, that is, $\psi(t,\cdot)$ is log-convex on $[x_\Lambda(t),+\infty)$.

\end{proof}

Last, we proof our last technical Lemma, Lemma \ref{lee2}.
\begin{proof}[Proof of Lemma \ref{lee2}]
Unless otherwise specified, throughout the proof of Lemma \ref{lee2}, we assume that $(t,x)\in[t^\star,\infty)\times[x_1(t)+1,+\infty)$.
 By the definition of $\psi$, for all $z\in\left(\frac{1}{x},\frac{x- x_1(t)}{x}\right)$, we have
 \[
		\begin{aligned}
			&\frac{1-\ln \psi (t,x(1-z))}{1-\ln \psi(t,x)}\\
			=&\left[1-\frac{[1-\ln[(1+\kappa t^p) v_0(x)]]^{\alpha+1}}{(1-\ln \psi(t,x))^{\alpha+1}}\left(1-\left(1-\frac{-2s\ln(1-z)+q\ln\left(1+\frac{\ln(1-z)}{\ln x}\right)}{1-\ln(1+\kappa t^p) v_0(x)}\right)^{\alpha+1}\right)\right]^\frac{1}{\alpha+1}.
		\end{aligned}
 \]
 Since $z\mapsto 2s\ln(1-z)-q\ln\left(1+\frac{\ln(1-z)}{\ln x}\right)$ is non-increasing on $\left(\frac{1}{x},1-\frac{ x_1(t)}{x}\right)$, we have, for all $z\in\left(\frac{1}{x},1-\frac{ x_1(t)}{x}\right)$,
\begin{equation}\label{esubf1}
\begin{aligned}
0\le\frac{-2s\ln(1-z)+q\ln\left(1+\frac{\ln(1-z)}{\ln x}\right)}{1-\ln[(1+\kappa t^p) v_0(x)]}&\le \frac{-2s\ln\frac{ x_1(t)}{x}+q\ln\left(1+\frac{\ln\frac{ x_1(t)}{x}}{\ln x}\right)}{1-\ln[(1+\kappa t^p) v_0(x)]}\\
&=1-\frac{1-\ln [(1+\kappa t^p)v_0( x_1(t))]}{1-\ln [(1+\kappa t^p)v_0(x)]}\le 1,
\end{aligned}
\end{equation}
and thus, by the definition of $ x_1$ and $\psi$, we get
\[
\begin{aligned}
    &0\le \frac{[1-\ln[(1+\kappa t^p) v_0(x)]]^{\alpha+1}}{(1-\ln \psi(t,x))^{\alpha+1}}\left(1-\left(1+\frac{2s\ln(1-z)-q\ln\left(1+\frac{\ln(1-z)}{\ln x}\right)}{1-\ln[(1+\kappa t^p) v_0(x)]}\right)^{\alpha+1}\right)\\
    &\le \frac{\left[[1-\ln [(1+\kappa t^p)v_0(x)]]^{\alpha+1}-r(\alpha+1)t\right]-\left[[1-\ln [(1+\kappa t^p)v_0( x_1(t))]]^{\alpha+1}-r(\alpha+1)t\right]}{(1-\ln \psi(t,x))^{\alpha+1}}\\
    &=1-\frac{1}{(1-\ln \psi(t,x))^{\alpha+1}}\le 1.
\end{aligned}
\]
It follows from the inequality $(1-y)^n\ge 1-ny -(1-n)y^2$ for all $y\in[0,1]$ and some $n\in(0,1)$ (see footnote \ref{twoineq}) that 
\[
\begin{aligned}
   & \frac{1-\ln \psi (t,x(1-z))}{1-\ln \psi(t,x)}\\
    &\ge 1-\frac{1}{\alpha+1}\frac{[1-\ln[(1+\kappa t^p) v_0(x)]]^{\alpha+1}}{(1-\ln \psi(t,x))^{\alpha+1}}\left(1-\left(1-\frac{-2s\ln(1-z)+q\ln\left(1+\frac{\ln(1-z)}{\ln x}\right)}{1-\ln[(1+\kappa t^p) v_0(x)]}\right)^{\alpha+1}\right)\\
    & \quad -\frac{\alpha}{\alpha+1}\left\{\frac{[1-\ln[(1+\kappa t^p) v_0(x)]]^{\alpha+1}}{(1-\ln \psi(t,x))^{\alpha+1}}\left(1-\left(1-\frac{-2s\ln(1-z)+q\ln\left(1+\frac{\ln(1-z)}{\ln x}\right)}{1-\ln[(1+\kappa t^p) v_0(x)]}\right)^{\alpha+1}\right)\right\}^2.
\end{aligned}
\]
By \eqref{esubf1} and the inequality $(1-y)^{\alpha+1}\ge 1-(\alpha+1)y$ for all $y\le 1$, we have
\[
\begin{aligned}
& \frac{[1-\ln[(1+\kappa t^p) v_0(x)]]^{\alpha+1}}{(1-\ln \psi(t,x))^{\alpha+1}}\left(1-\left(1-\frac{-2s\ln(1-z)+q\ln\left(1+\frac{\ln(1-z)}{\ln x}\right)}{1-\ln[(1+\kappa t^p) v_0(x)]}\right)^{\alpha+1}\right)\\
&\le (\alpha+1)\frac{[1-\ln[(1+\kappa t^p) v_0(x)]]^{\alpha}}{(1-\ln \psi(t,x))^{\alpha+1}}\left[-2s\ln(1-z)+q\ln\left(1+\frac{\ln(1-z)}{\ln x}\right)\right]\\
&\le (\alpha+1)\frac{[1-\ln[(1+\kappa t^p) v_0(x)]]^{\alpha}}{(1-\ln \psi(t,x))^{\alpha+1}}\left[-2s\ln(1-z)\right].
\end{aligned}
\]
In view of the definition of $v_0$, there is a time $t^\dagger$ such that for all $(t,x)\in [t^\dagger,\infty)\times [x_1(t)+1,+\infty)$, we have
\[
1-\ln [(1+\kappa t^p)v_0(x)]\le 2s \ln x.
\]
Thus, 
for all $z\in\left(\frac{1}{x},1-\frac{ x_1(t)}{x}\right)$ and  $(t,x)\in [t^\dagger,\infty)\times [x_1(t)+1,+\infty)$, we arrive at 
\[
\begin{aligned}
   &  \frac{1-\ln \psi (t,x(1-z))}{1-\ln \psi(t,x)}\\
    &\ge 1-\frac{[1-\ln[(1+\kappa t^p) v_0(x)]]^{\alpha}}{(1-\ln \psi(t,x))^{\alpha+1}}\left[-2s\ln(1-z)+q\ln\left(1+\frac{\ln(1-z)}{\ln x}\right)\right]\\
    & \quad -\alpha(\alpha+1)\left\{\frac{[1-\ln[(1+\kappa t^p) v_0(x)]]^{\alpha}}{(1-\ln \psi(t,x))^{\alpha+1}}\left[-2s\ln(1-z)+q\ln\left(1+\frac{\ln(1-z)}{\ln x}\right)\right]\right\}^2\\
    &\ge 1-\frac{[1-\ln[(1+\kappa t^p) v_0(x)]]^{\alpha}}{(1-\ln \psi(t,x))^{\alpha+1}}\left[-2s\ln(1-z)\right] -\alpha(\alpha+1)\frac{[1-\ln[(1+\kappa t^p) v_0(x)]]^{2\alpha}}{(1-\ln \psi(t,x))^{2\alpha+2}}\left[-2s\ln(1-z)\right]^2\\
    &\ge 1+(2s)^{\alpha+1}(\ln^\alpha x)\ln (1-z)-C_\alpha(\ln^{2\alpha} x)\ln^2(1-z),
\end{aligned}
\]
where $C_\alpha=\alpha(\alpha+1)(2s)^{2\alpha+2}$.
Therefore, for all $z\in\left(\frac{1}{x},1-\frac{ x_1(t)}{x}\right)$ and $(t,x)\in [t^\dagger,\infty)\times [x_1(t)+1,+\infty)$, we achieve
\[
\begin{aligned}
   &\frac{\psi(t,x(1-z))}{\psi(t,x)}\\
&=\exp\left\{(1-\ln \psi(t,x))\left(1-\frac{1-\ln \psi(t,x(1-z))}{1-\ln \psi(t,x)}\right)\right\} \\
    &\le \exp\left\{ \frac{[1-\ln[(1+\kappa t^p) v_0(x)]]^{\alpha}}{(1-\ln \psi(t,x))^{\alpha}}\left[-2s\ln(1-z)\right] +\alpha(\alpha+1)\frac{[1-\ln[(1+\kappa t^p) v_0(x)]]^{2\alpha}}{(1-\ln \psi(t,x))^{2\alpha+1}}\left[-2s\ln(1-z)\right]^2\right\}\\
    &\le \exp\left\{ -(2s)^{\alpha+1}(\ln^\alpha x)\ln (1-z)+C_\alpha (\ln^{2\alpha} x)\ln^2(1-z) \right\}.
\end{aligned}
\]
	This completes the proof.
\end{proof}

\subsection{The lower bound}\label{appendix-exp-acc-low}

Firstly, we prove Lemma \ref{eslbl1}.
\begin{proof}[Proof of Lemma \ref{eslbl1}]
Denote $Y_0(t):=e^{-\frac{1}{s} }e^{\frac{1}{2s}\left[r(\alpha+1)t\right]^\frac{1}{\alpha+1}}$ for all $t>0$ and $t_0:=\frac{2^{\alpha+1}}{r(\alpha+1)}$. Let us define 
$$\varphi_{log}(t,x):=(1+2s\ln x)^{\alpha+1}-r(\alpha+1)[t-h(t,x)],$$
for all $(t,x)\in [t_0,+\infty)\times [Y_0(t),+\infty)$, satisfying $\varphi= \exp\{1-\varphi_{log}^\frac{1}{\alpha+1}\}$. 
By some direct calculations, since $\kappa \le\frac{[r(\alpha+1)]^\frac{\alpha}{\alpha+1}}{2^{\alpha+1} rpe^{2p}} $,  for all $(t,x)\in [t_0,+\infty)\times [Y_0(t),+\infty)$, we have
\[
\begin{aligned}
\frac{\partial \varphi_{log}}{\partial x}
&=  2s (\alpha+1)\frac{1}{x}\left[(1+2s\ln x)^\alpha-rp \kappa \frac{ t^{\frac{\alpha}{\alpha+1}}e^{p[r(\alpha+1)t]^\frac{1}{\alpha+1}}}{x^{2sp}}\right]\\
&\ge 2s (\alpha+1)\left\{\left[[r(\alpha+1)]^\frac{1}{\alpha+1}-t^{-\frac{1}{\alpha+1}}\right]^\alpha-rp e^{2p}\kappa \right\}\frac{t^\frac{\alpha}{\alpha+1}}{x}\\
&\ge 2s (\alpha+1)\left\{2^{-\alpha}[r(\alpha+1)]^\frac{\alpha}{\alpha+1}-rp e^{2p}\kappa \right\}\frac{t^\frac{\alpha}{\alpha+1}}{x}> 0.
\end{aligned}
\]
  As a result, for all fixed $t\ge t_0$, $\varphi_{log}(t,\cdot)$ is strictly increasing on $[Y_0(t),+\infty]$. One may notice that for any $t\ge t_0$, since $\kappa \le\frac{1}{2e^{2p}[r(\alpha+1)]^\frac{1}{\alpha+1}} $,
 \[
 \begin{aligned}
 \varphi_{log}(t,Y_0(t))&= [r(\alpha+1)t]\left\{\left[1-[r(\alpha+1)t]^{-\frac{1}{\alpha+1}}\right]^{\alpha+1}-1+e^{2p}\kappa t^{-\frac{1}{\alpha+1}} \right\}\\
 &\le  [r(\alpha+1)]\left\{-[r(\alpha+1)]^{-\frac{1}{\alpha+1}}+e^{2p}\kappa \right\}t^{\frac{\alpha}{\alpha+1}}<0,
 \end{aligned}
 \]
 and, denoting $Y_1(t):=e^{-\frac{1}{2s}+\frac{1}{2s}[r(\alpha+1)t]^\frac{1}{\alpha+1}}$,
 \[
 \varphi_{log}(t,Y_1(t))= r(\alpha+1)h(t,Y_1(t))>0.
 \]
Thus, for all $t\ge t_0$, there exists a unique $Y(t)\in (Y_0(t),Y_1(t))$ such that $\varphi_{log}(t,Y(t))=0$ and $\varphi_{log}(t,x)>0$ for all $x\in(Y(t),+\infty)$. It therefore follows from the definition of $\varphi$ that, for all $t\ge t_0$, $\varphi(t,x)$ is well-defined if only if $x\in [Y(t),+\infty)$ and $\varphi(t,\cdot)$ is strictly decreasing on $[Y(t),+\infty)$. Thus, by the definition of $x_\Lambda$, one may have $\varphi_{log}(t,x_e(t))=0$, which implies $x_e(t)=Y(t)$. Therefore, it follows from the fact that, for all $t\ge t_0$, $\varphi(t,\cdot)$ is strictly decreasing on $[x_e(t),+\infty)$ that for any $\Lambda\in(0,e)$ and all $t\ge t_0$, the position $x_\Lambda(t)$ is decreasing in $\Lambda$ unique such that $\varphi(t,x_\Lambda(t))=\Lambda$ and 
\[
x_\Lambda(t)\ge x_e(t)> Y_0(t)= e^{-\frac{1}{s} }e^{\frac{1}{2s}\left[r(\alpha+1)t\right]^\frac{1}{\alpha+1}},
\]
which gives \eqref{eslbxtl}.
\par
    Let us show that $\varphi(t,x)$ is strictly increasing in $t$ for all $t$ large enough and $x\ge x_\Lambda(t)$.  By the definition of $h$, for $x\ge x_\Lambda(t)$, by \eqref{eslbxtl}, there is $\tilde t_1\ge t_0$ such that for all $t\ge \tilde t_1$, we have
\[
    \frac{\partial h}{\partial t}=  \kappa\frac{ e^{p[r(\alpha+1)t]^\frac{1}{\alpha+1}}}{x^{2sp}}\left(\tilde \gamma_1+\frac{\alpha}{\alpha+1}t^{-\frac{1}{\alpha+1}} \right)\le  \kappa \frac{e^{p[r(\alpha+1)t]^\frac{1}{\alpha+1}}}{x_\Lambda^{2sp}(t)}\left(\tilde \gamma_1 +\frac{\alpha}{\alpha+1}t^{-\frac{1}{\alpha+1}}\right)
\le 2e^{2p}\tilde \gamma_1 \kappa,
\]
where $\tilde \gamma_1:= \frac{p}{\alpha+1}[r(\alpha+1)]^\frac{1}{\alpha+1}$. 
As a result, since $\kappa \le \frac{1}{4\tilde \gamma_1 e^{2p}}$, for all $t\ge \tilde t_1$ and $x\ge  x_\Lambda(t)$, we get
\[
\frac{\partial \varphi}{\partial t}=\frac{r\varphi}{(1-\ln \varphi)^\alpha}\left[1 -\frac{\partial h}{\partial t}\right]\ge r\left(1 -2e^{2p}\tilde \gamma_1\kappa   \right)\frac{\varphi}{(1-\ln \varphi)^\alpha}> 0.
\]
\par Let us now show that, for $t$ large enough, $\varphi(t,\cdot)$ is convex on $[x_\Lambda(t),+\infty)$. By the definition of $h$, we have
\[
    \frac{\partial h}{\partial x}= -2sp\kappa \frac{t^\frac{\alpha}{\alpha+1}e^{p[r(\alpha+1)t]^\frac{1}{\alpha+1}}}{x^{2sp+1}},
\]
and thus, for any $(t,x)\in [t_0,\infty)\times[x_\Lambda(t),+\infty)$, by \eqref{eslbxtl},  we have $$\frac{\partial^2 h}{\partial x^2}=2sp(2sp+1)\kappa\frac{ t^{\frac{\alpha}{\alpha+1}}e^{p[r(\alpha+1)t]^\frac{1}{\alpha+1}}}{ x^{2sp+2}}\le \tilde \gamma_2 \kappa \frac{t^{\frac{\alpha}{\alpha+1}}}{ x^{2}},$$
where $\tilde \gamma_2:=2sp(2sp+1)e^{2p}$.
 Some direct calculations show that 
\begin{equation}\label{eslbphix}
\frac{\partial \varphi}{\partial x}=\frac{\varphi}{(1-\ln \varphi)^\alpha}\left[-2s\frac{(1+2s\ln x)^\alpha}{x}-r\frac{\partial h}{\partial x}\right],
\end{equation}
and
\begin{equation}
    \label{eslbphixx}
\begin{aligned}
\frac{\partial^2 \varphi}{\partial x^2}&=\frac{\varphi}{(1-\ln \varphi)^\alpha}\Bigg\{\left[2s\frac{(1+2s\ln x)^\alpha}{x}+r\frac{\partial h}{\partial x}\right]^2 \left[(1-\ln \varphi)^{-\alpha}+\alpha (1-\ln\varphi)^{-1}\right]\\
&\qquad+\left[2s\frac{(1+2s\ln  x)^\alpha }{ x^2}\left(1-2s\alpha (1+2s\ln  x)^{-1}\right)-r\frac{\partial^2 h}{\partial x^2}\right]\Bigg\}.
\end{aligned}
\end{equation}
By substituting $\frac{\partial h}{\partial x}$ and $\frac{\partial^2 h}{\partial x^2}$  into  \eqref{eslbphixx}, for any $\Lambda\in(0,e)$ there is $\tilde t_2\ge t_0$ such that for all $t\ge \tilde t_2$ and  $x\ge x_\Lambda(t)$,
\[
\begin{aligned}
\frac{\partial^2 \varphi}{\partial x^2}&\ge\frac{\varphi}{(1-\ln \varphi)^\alpha}\Bigg\{2s\frac{(1+2s\ln  x)^\alpha }{ x^2}\left[1-2s\alpha (1+2s\ln  x)^{-1}\right]- r\tilde \gamma_2 \kappa \frac{t^{\frac{\alpha}{\alpha+1}}}{ x^{2}} \Bigg\}\\
&\ge\frac{t^{\frac{\alpha}{\alpha+1}}}{ x^{2}} \frac{\varphi}{(1-\ln \varphi)^\alpha}\Bigg\{2s\frac{\left[-1+[r(\alpha+1)t]^\frac{1}{\alpha+1} \right]^\alpha}{t^\frac{\alpha}{\alpha+1}}\left[1-2s\alpha \left[-1+[r(\alpha+1)t]^\frac{1}{\alpha+1}\right]^{-1}\right]- r\tilde \gamma_2 \kappa \Bigg\} \ge 0,
\end{aligned}
\]
since $\kappa\le \frac{s[r(\alpha+1)]^\frac{\alpha}{\alpha+1}}{2^\alpha r\tilde \gamma_2}$.
\par
Therefore, by taking $ \tilde t:=\max\{t_0,\tilde t_1,\tilde t_2\}$, for any $\Lambda\in(0,e)$, we conclude that for any $t\ge \tilde t$ and  $x\ge x_\Lambda(t)$, $\varphi(t,x)$ is well-defined, strictly increasing in $t$, and strictly decreasing and convex in $x$.

\end{proof}
Next, we prove Lemma \ref{varphietx}.
\begin{proof}[Proof of Lemma \ref{varphietx}]
By the inequality $a^q+b^q\ge (a+b)^q$ for any $a>0$, $b>0$ and $q\in(0,1)$, one may obtain
\[
\left[\left(1+2s\ln x\right)^{\alpha+1}- r(\alpha+1)t\right]^{\frac{1}{\alpha+1}}\ge 1+2s\ln x-[r(\alpha+1)t]^\frac{1}{\alpha+1}.
\]
Thus, by the definition of $\varphi$, one has 
\[
\varphi(t,x)\le \frac{e^{[r(\alpha+1)t]^\frac{1}{\alpha+1}}}{x^{2s}}\quad \text{ for all }t> \tilde t\text{ and }x\ge x_\Lambda(t).
\]
On the other hand, one may notice that $t-h(t,x)\ge 0$ for all $t\ge \tilde t$ and $x\ge x_\Lambda(t)$, and it follows from the definition of $\varphi$ that
\[
  \varphi(t,x)\ge \frac{1}{x^{2s}} \quad \text{ for all }t> \tilde t\text{ and }x\ge x_\Lambda(t).
\]
\par In view of the definition of $\varphi$, some direct calculations show that 
\[
\frac{\partial \varphi}{\partial t}=\frac{r\varphi}{(1-\ln \varphi)^\alpha}\left[1-\frac{\partial h}{\partial t}\right]\le \frac{r\varphi}{(1-\ln \varphi)^\alpha}\left[1-\tilde \gamma_1 \kappa \frac{e^{p[r(\alpha+1)t]^\frac{1}{\alpha+1}}}{x^{2sp}}\right],
\]
where $\tilde \gamma_1= \frac{p}{\alpha+1}[r(\alpha+1)]^\frac{1}{\alpha+1}$. 
Thus, by \eqref{eslbephi}, one may notice that
\[
    \frac{\partial \varphi}{\partial t}\le \frac{r\varphi}{(1-\ln \varphi)^\alpha}(1-\tilde \gamma_1 \kappa \varphi^p),
\]
and by \eqref{eslbphix}, since $y\mapsto y(1-2s\ln y)^\alpha$ is non-decreasing for $0<y\le e^{\frac{1}{2s}-\alpha}$, one may have for $\varphi< \min\{e^{\frac{1}{2s}-\alpha},e\}$,
\[
\frac{\partial \varphi}{\partial x}  \ge -2s\frac{\varphi}{(1-\ln \varphi)^\alpha}\frac{(1+2s\ln x)^\alpha}{x}\ge -2s\varphi^{1+\frac{1}{2s}}.
\]

\par
By \eqref{eslbxtl}, \eqref{eslbphix} and the definition of $x_\Lambda(t)$, for any $\Lambda\in(0,1)$, we have
\[
\begin{aligned}
    0\ge\frac{\partial \varphi}{\partial x}(t,x_\Lambda(t))&\ge -2s\frac{\Lambda}{(1-\ln \Lambda)^\alpha}\frac{\left[1+2s\ln  x_\Lambda(t)\right]^\alpha}{ x_\Lambda(t)}\\
    &\ge -2s e^\frac{1}{s} [r(\alpha+1)t]^\frac{\alpha}{\alpha+1}e^{-\frac{1}{2s}[r(\alpha+1)t]^\frac{1}{\alpha+1}} \to 0 \quad \text{as } t\to \infty,
    \end{aligned}
\]
and thus, 
\[
    \lim_{t\to \infty}\frac{\partial \varphi}{\partial x}(t,x_\Lambda(t))= 0 \quad \text{uniformly for }\Lambda\in(0,1).
\]

\end{proof}
Now, we prove Lemma \ref{eslbl2}.
\begin{proof}[Proof of Lemma \ref{eslbl2}]
The lower bound of $x_\Lambda(t)$ has already been given by \eqref{eslbxtl}. Let us now show the upper bound of $x_\Lambda(t)$. 
Since 
\[
 \left(1+2s\ln  x\right)^{\alpha+1}- r(\alpha+1)[t-h(t,x)]\ge \ln^{\alpha+1} \left(x^{2s}\right) -r(\alpha+1)t,
\]
by the inequality $(a+b)^q\le a^q+b^q$ for some $a>0$, $b>0$ and $q\in(0, 1)$,  we have
\[
\begin{aligned}
\varphi\left(t,\bar C_\Lambda e^{\frac{1}{2s}[r(\alpha+1)t]^\frac{1}{\alpha+1}}\right)
&\le \exp\left\{1-\left[\ln^{\alpha+1} \left(\bar C_\Lambda^{2s} e^{[r(\alpha+1)t]^\frac{1}{\alpha+1}}\right) -r(\alpha+1)t\right]^\frac{1}{\alpha+1}\right\}\\
&\le \frac{e}{\bar C_\Lambda^{2s}}\le \Lambda=\varphi(t,x_\Lambda(t)),
\end{aligned}
\]
as long as $\bar C_\Lambda\ge \left(\frac{e}{\Lambda}\right)^\frac{1}{2s}$, and then, since $\varphi(t,\cdot)$ is decreasing on $[x_\Lambda(t),+\infty)$ for all $t\ge \tilde t$, we obtain for any $\bar C_\Lambda\ge \left(\frac{e}{\Lambda}\right)^\frac{1}{2s}$,
\[
x_\Lambda(t)\le\bar C_\Lambda e^{\frac{1}{2s}[r(\alpha+1)t]^\frac{1}{\alpha+1}},
\]
which completes the estimate of $x_\Lambda$.
\par
Let us denote
\[
\mathcal{X}(x,y ):=\exp\left\{-\frac{1}{2s}+\frac{1}{2s}\left[(1-\ln x)^{\alpha+1}+r(\alpha+1)\left[t-h(t,y)\right]\right]^\frac{1}{\alpha+1}\right\}.
\]
Some direct calculations show that 
\[
\frac{\partial \mathcal{X}}{\partial x}= -\frac{(1-\ln x)^\alpha}{2sx}\frac{\mathcal{X}}{(1+2s\ln \mathcal{X})^\alpha},
\]
and 
\[
\frac{\partial \mathcal{X}}{\partial y}= -\frac{r}{2s}\frac{\partial h}{\partial y}\frac{\mathcal{X}}{(1+2s\ln \mathcal{X})^\alpha}.
\]
For any $\Lambda\in(0,1)$, notice that 
\[
x_\Lambda(t)=\mathcal{X}(\Lambda,x_\Lambda(t) ).
\]
Then, by the definition of $x_\Lambda$, we have
\[
\frac{\partial x_\Lambda(t)}{\partial \Lambda}= \frac{\partial \mathcal{X}}{\partial x}(\Lambda, x_\Lambda(t))+\frac{\partial x_\Lambda(t)}{\partial \Lambda}\frac{\partial \mathcal{X}}{\partial y}(\Lambda, x_\Lambda(t)).
\]
By \eqref{eslbxtl}, for large time $t$, say $t\ge \tilde t_3$, we get
\[
\frac{\partial \mathcal{X}}{\partial x}(\Lambda,x_\Lambda(t) )= -\frac{(1-\ln \Lambda)^\alpha}{2s\Lambda}\frac{ x_\Lambda(t)}{(1+2s\ln  x_\Lambda(t))^\alpha}\lesssim_\Lambda-t^{-\frac{\alpha}{\alpha+1}}e^{\frac{1}{2s}[r(\alpha+1)t]^\frac{1}{\alpha+1}},
\]
and since $\kappa=\kappa_0\le \frac{[r(\alpha+1)]^\frac{\alpha}{\alpha+1}}{2^{\alpha+1} rpe^{2p}}$,
\[
\begin{aligned}
 \frac{\partial \mathcal{X}}{\partial y}(\Lambda,x_\Lambda(t) )&= r p \kappa\frac{t^\frac{\alpha}{\alpha+1}e^{p[r(\alpha+1)t]^\frac{1}{\alpha+1}}}{ x_\Lambda(t)^{2sp}}\frac{1}{(1+2s\ln  x_\Lambda(t))^\alpha}\le  2rp e^{2p}[r(\alpha+1)]^{-\frac{\alpha}{\alpha+1}}\kappa \le \frac{1}{2^\alpha}<1.
 \end{aligned}
\]
Thus, for all $t\ge \tilde t_3$, we get
\[
\frac{\partial x_\Lambda(t)}{\partial \Lambda}=\frac{\frac{\partial \mathcal{X}}{\partial x}(\Lambda, x_\Lambda(t))}{1-\frac{\partial \mathcal{X}}{\partial y}(\Lambda, x_\Lambda(t))}
=\frac{ -\frac{(1-\ln \Lambda)^\alpha}{2s\Lambda}\frac{x_\Lambda(t)}{(1+2s\ln x_\Lambda(t))^\alpha}}{1-\frac{\partial \mathcal{X}}{\partial y}(\Lambda, x_\Lambda(t))}\lesssim_\Lambda -t^{-\frac{\alpha}{\alpha+1}}e^{\frac{1}{2s}[r(\alpha+1)t]^\frac{1}{\alpha+1}}.
\]
Therefore, for any $0<\Lambda_1<\Lambda_2<1$, by the definition of $x_\Lambda$ and the mean value theorem, there is $\xi\in(\Lambda_1,\Lambda_2)$ such that
\[
x_{\Lambda_1}(t)-x_{\Lambda_2}(t)=(\Lambda_1-\Lambda_2)\frac{\partial x_\Lambda(t)}{\partial \Lambda}\vline_{\Lambda=\xi}\gtrsim_{\Lambda_1,\Lambda_2,\xi} t^{-\frac{\alpha}{\alpha+1}}e^{\frac{1}{2s}[r(\alpha+1)t]^\frac{1}{\alpha+1}}\to +\infty,
\]
as $t\to \infty$, which completes the proof of the lemma.

\end{proof}

Lastly, we prove Lemma \ref{ge}.
\begin{proof}[Proof of Lemma \ref{ge}]
Let 
\[
    g_\varepsilon(y):=
    \begin{cases}
        y, &y\in \left[0, \frac{\varepsilon}{2}\right],\\
        \exp\left\{1-[(1-\ln y)^{\alpha+1}+P(y)]^\frac{1}{\alpha+1}\right\},& y\in\left(\frac{\varepsilon}{2},\frac{3}{2}\varepsilon\right),\\
         \varepsilon, &y\in \left[\frac{3}{2}\varepsilon,1\right],
    \end{cases}
\]
satisfying
\[
g_\varepsilon'(y)=
    \begin{cases}
        1, &y\in \left[0, \frac{\varepsilon}{2}\right],\\
        \frac{g_\varepsilon}{(1-\ln g_\varepsilon)^\alpha}\frac{(1-\ln y)^\alpha}{y}\left(1-\frac{1}{\alpha+1}\frac{y}{(1-\ln y)^\alpha}P'(y)\right),& y\in\left(\frac{\varepsilon}{2},\frac{3}{2}\varepsilon\right),\\
         0, &y\in \left[\frac{3}{2}\varepsilon,1\right],
    \end{cases}
\]
where we here select a function $P\in C^2\left(\left[\frac{\varepsilon}{2},\frac{3}{2}\varepsilon\right]\right) $ with $P\ge 0$ and $P'\ge 0$ on $\left[\frac{\varepsilon}{2},\frac{3}{2}\varepsilon\right]$ such that
\[
\begin{aligned} 
&P\left(\frac{\varepsilon}{2}\right)=0,\quad P\left(\frac{3}{2}\varepsilon\right)= (1-\ln\varepsilon)^{\alpha+1}-\left(1-\ln\left(\frac{3}{2}\varepsilon\right)\right)^{\alpha+1},\\
&P'\left(\frac{\varepsilon}{2}\right)=0 \quad \text{and}\quad P'\left(\frac{3}{2}\varepsilon\right)=(\alpha+1)\frac{\left(1-\ln\left(
\frac{3}{2}\varepsilon\right)\right)^\alpha}{\frac{3}{2}\varepsilon}.
\end{aligned}
\]
Thus, by the choice of $P$, one may notice that 
\[
g_\varepsilon\left(\frac{\varepsilon}{2}\right)=\frac{\varepsilon}{2},\quad  g_\varepsilon\left(\frac{3}{2}\varepsilon\right)=\varepsilon,\quad
g_\varepsilon'\left(\frac{\varepsilon}{2}\right)=1\quad\text{and} \quad g_\varepsilon'\left(\frac{3}{2}\varepsilon\right)=0,
\]
and then, one gets $g_\varepsilon\in C^1([0,1])$. Also, by choosing appropriate value for $P''\left(\frac{\varepsilon}{2}\right)$ and $P''\left(\frac{3}{2}\varepsilon\right)$ such that $g''\left(\frac{\varepsilon}{2}\right)=0$ and $g''\left(\frac{3}{2}\varepsilon\right)=0$. As a result, under the choice of $P$, one has $g_\varepsilon\in C^2([0,1])$.
\par
One may then get that, for all $y\in(0,1)$,
\[
g_\varepsilon(y)\le y,\quad g_\varepsilon'(y)\frac{y}{(1-\ln y)^\alpha}\le \frac{g_\varepsilon(y)}{(1-\ln g_\varepsilon(y))^\alpha}
\quad \text{and}\quad 0 \le g'_\varepsilon(y)\le 1.
\]
Moreover, since $g''_\varepsilon\in C([0,1])$, there is some constant $C_\varepsilon$, depending on $\varepsilon$, such that
 \[
     g_\varepsilon'' \ge -C_\varepsilon \quad \text{for some }C_\varepsilon>0.  
\]
\end{proof}
\section*{Acknowledgment}
X. Zhang was supported by the Fundamental Research Funds for the Central Universities of Central South University(No.2024ZZTS0118).
		\bibliographystyle{plain}
		\bibliography{ref}

\begin{thebibliography}{10}

\bibitem{alfaro2017slowing}
Matthieu Alfaro.
\newblock Slowing allee effect versus accelerating heavy tails in monostable
  reaction diffusion equations.
\newblock {\em Nonlinearity}, 30(2):687, 2017.

\bibitem{alfaro2017propagation}
Matthieu Alfaro and J{\'e}r{\^o}me Coville.
\newblock Propagation phenomena in monostable integro-differential equations:
  acceleration or not?
\newblock {\em Journal of Differential Equations}, 263(9):5727--5758, 2017.

\bibitem{Alfaro2019}
Matthieu Alfaro and Thomas Giletti.
\newblock When fast diffusion and reactive growth both induce accelerating
  invasions.
\newblock {\em Communications on Pure and Applied Analysis}, 18(6):3011--3034,
  2019.

\bibitem{aronson1978multidimensional}
Donald~G Aronson and Hans~F Weinberger.
\newblock Multidimensional nonlinear diffusion arising in population genetics.
\newblock {\em Advances in Mathematics}, 30(1):33--76, 1978.

\bibitem{Bialynicki1976}
Iwo Bialynicki-Birula and Jerzy Mycielski.
\newblock Nonlinear wave mechanics.
\newblock {\em Annals of Physics}, 100(1-2):62--93, 1976.

\bibitem{bouin2024}
Emeric Bouin and J{\'e}r{\^o}me Coville.
\newblock Existence of fronts in non-local reaction-dispersion equations with
  diffuse levy measures.
\newblock {\em preprint}, 2024.

\bibitem{bouin2021sharp}
Emeric Bouin, J{\'e}r{\^o}me Coville, and Guillaume Legendre.
\newblock Sharp exponent of acceleration in general nonlocal equations with a
  weak allee effect.
\newblock {\em arXiv preprint arXiv:2105.09911}, 2021.

\bibitem{bouin2023simple}
Emeric Bouin, J{\'e}r{\^o}me Coville, and Guillaume Legendre.
\newblock A simple flattening lower bound for solutions to some linear
  integro-differential equations.
\newblock {\em Zeitschrift f{\"u}r angewandte Mathematik und Physik},
  74(6):234, 2023.

\bibitem{BOUIN2024113557}
Emeric Bouin, Jérôme Coville, and Xi~Zhang.
\newblock Precise rates of propagation in reaction–diffusion equations with
  logarithmic allee effect.
\newblock {\em Nonlinear Analysis}, 245:113557, 2024.

\bibitem{MR3817762}
Emeric Bouin, Jimmy Garnier, Christopher Henderson, and Florian Patout.
\newblock Thin front limit of an integro-differential {F}isher-{KPP} equation
  with fat-tailed kernels.
\newblock {\em SIAM J. Math. Anal.}, 50(3):3365--3394, 2018.

\bibitem{BOUIN2021}
Emeric Bouin and Christopher Henderson.
\newblock The {B}ramson delay in a {F}isher–{KPP} equation with log-singular
  nonlinearity.
\newblock {\em Nonlinear Analysis}, 213:112508, 2021.

\bibitem{cabre2013influence}
Xavier Cabr{\'e} and Jean-Michel Roquejoffre.
\newblock The influence of fractional diffusion in {F}isher-{KPP} equations.
\newblock {\em Communications in Mathematical Physics}, 320(3):679--722, 2013.

\bibitem{carr2004uniqueness}
Jack Carr and Adam Chmaj.
\newblock Uniqueness of travelling waves for nonlocal monostable equations.
\newblock {\em Proceedings of the American Mathematical Society},
  132(8):2433--2439, 2004.

\bibitem{coville2007travelling}
J{\'e}r{\^o}me Coville.
\newblock Travelling fronts in asymmetric nonlocal reaction diffusion
  equations: the bistable and ignition cases.
\newblock {\em preprint}, 2007.

\bibitem{coville2008nonlocal}
J{\'e}r{\^o}me Coville, Juan D{\'a}vila, and Salom{\'e} Mart{\'\i}nez.
\newblock Nonlocal anisotropic dispersal with monostable nonlinearity.
\newblock {\em Journal of Differential Equations}, 244(12):3080--3118, 2008.

\bibitem{coville2007non}
J{\'e}r{\^o}me Coville and Louis Dupaigne.
\newblock On a non-local equation arising in population dynamics.
\newblock {\em Proceedings of the Royal Society of Edinburgh Section A:
  Mathematics}, 137(4):727--755, 2007.

\bibitem{coville2021propagation}
J{\'e}r{\^o}me Coville, Changfeng Gui, and Mingfeng Zhao.
\newblock Propagation acceleration in reaction diffusion equations with
  anomalous diffusions.
\newblock {\em Nonlinearity}, 34(3):1544, 2021.

\bibitem{Dennis1989}
Brian Dennis.
\newblock Allee effects: population growth, critical density, and the chance of
  extinction.
\newblock {\em Natural Resource Modeling}, 3(4):481--538, 1989.

\bibitem{duo2018novel}
Siwei Duo, Hans~Werner van Wyk, and Yanzhi Zhang.
\newblock A novel and accurate finite difference method for the fractional
  laplacian and the fractional poisson problem.
\newblock {\em Journal of Computational Physics}, 355:233--252, 2018.

\bibitem{finkelshtein2019multid}
Dmitri Finkelshtein, Yuri Kondratiev, and Pasha Tkachov.
\newblock Accelerated front propagation for monostable equations with nonlocal
  diffusion: multidimensional case.
\newblock {\em Journal of Elliptic and Parabolic Equations}, 5(2):423--471,
  2019.

\bibitem{finkelshtein2019}
Dmitri Finkelshtein and Pasha Tkachov.
\newblock Accelerated nonlocal nonsymmetric dispersion for monostable equations
  on the real line.
\newblock {\em Applicable Analysis}, 98(4):756--780, 2019.

\bibitem{fisher1937wave}
Ronald~Aylmer Fisher.
\newblock The wave of advance of advantageous genes.
\newblock {\em Annals of eugenics}, 7(4):355--369, 1937.

\bibitem{garnier2011accelerating}
Jimmy Garnier.
\newblock Accelerating solutions in integro-differential equations.
\newblock {\em SIAM Journal on Mathematical Analysis}, 43(4):1955--1974, 2011.

\bibitem{flatteningeffect}
Jimmy Garnier, François Hamel, and Lionel Roques.
\newblock Transition fronts and stretching phenomena for a general class of
  reaction-dispersion equations.
\newblock {\em Discrete and Continuous Dynamical Systems}, 37(2):743--756,
  2017.

\bibitem{Gompertz1825}
Benjamin Gompertz.
\newblock Xxiv. on the nature of the function expressive of the law of human
  mortality, and on a new mode of determining the value of life contingencies.
  in a letter to francis baily, esq. frs \&c.
\newblock {\em Philosophical transactions of the Royal Society of London},
  115:513--583, 1825.

\bibitem{gui2015traveling}
Changfeng Gui and Tingting Huan.
\newblock Traveling wave solutions to some reaction diffusion equations with
  fractional laplacians.
\newblock {\em Calculus of Variations and Partial Differential Equations},
  54:251--273, 2015.

\bibitem{Hamel2010}
Fran{\c{c}}ois Hamel and Lionel Roques.
\newblock Fast propagation for kpp equations with slowly decaying initial
  conditions.
\newblock {\em Journal of Differential Equations}, 249(7):1726--1745, 2010.

\bibitem{hayes1996statistical}
Monson~H Hayes.
\newblock {\em Statistical digital signal processing and modeling}.
\newblock John Wiley \& Sons, 1996.

\bibitem{huang2014numerical}
Yanghong Huang and Adam Oberman.
\newblock Numerical methods for the fractional laplacian: A finite
  difference-quadrature approach.
\newblock {\em SIAM Journal on Numerical Analysis}, 52(6):3056--3084, 2014.

\bibitem{kolmogorov1937etude}
Andrei Kolmogorov.
\newblock {\'E}tude de l’{\'e}quation de la diffusion avec croissance de la
  quantit{\'e} de mati{\`e}re et son application {\`a} un probl{\`e}me
  biologigue.
\newblock {\em Moscow Univ. Bull. Ser. Internat. Sect. A}, 1:1, 1937.

\bibitem{Kot2003}
J.~Medlock and M.~Kot.
\newblock Spreading disease: integro-differential equations old and new.
\newblock {\em Math. Biosci.}, 184(2):201--222, 2003.

\bibitem{minden2020simple}
Victor Minden and Lexing Ying.
\newblock A simple solver for the fractional laplacian in multiple dimensions.
\newblock {\em SIAM Journal on Scientific Computing}, 42(2):A878--A900, 2020.

\bibitem{Schumacher19805470}
Konrad Schumacher.
\newblock Travelling-front solutions for integro-differential equations. i.
\newblock {\em Journal für die reine und angewandte Mathematik},
  1980(316):54--70, 1980.

\bibitem{tian2013analysis}
Xiaochuan Tian and Qiang Du.
\newblock Analysis and comparison of different approximations to nonlocal
  diffusion and linear peridynamic equations.
\newblock {\em SIAM Journal on Numerical Analysis}, 51(6):3458--3482, 2013.

\bibitem{weinberger1982long}
Hans~F Weinberger.
\newblock Long-time behavior of a class of biological models.
\newblock {\em SIAM journal on Mathematical Analysis}, 13(3):353--396, 1982.

\bibitem{yagisita2010existence}
Hiroki Yagisita.
\newblock Existence and nonexistence of traveling waves for a nonlocal
  monostable equation.
\newblock {\em Publications of the Research Institute for Mathematical
  Sciences}, 45(4):925--953, 2010.

\bibitem{zhang2023optimal}
Yuming~Paul Zhang and Andrej Zlato{\v{s}}.
\newblock Optimal estimates on the propagation of reactions with fractional
  diffusion.
\newblock {\em Archive for Rational Mechanics and Analysis}, 247(5):93, 2023.

\end{thebibliography}
		\addcontentsline{toc}{section}{References} 
		
	\end{document}